\documentclass[10pt,leqno,twoside]{amsart}

\usepackage{enumerate}
\usepackage{tikz}
\usepackage{pgfplots}

\usepackage{amssymb}
\usepackage[english]{babel}
\usepackage{amssymb,amsthm,amsmath,eucal,mathrsfs}
\usepackage{bm}

\usepackage{amsmath}
\usepackage{amsfonts}
\usepackage{amsmath}
\usepackage{amssymb}
\usepackage{graphicx}
\usepackage{lscape}
\usepackage{amstext}
\usepackage{amsthm}
\usepackage{color}
\usepackage{float}
\usepackage{mathrsfs}
\usepackage{epsfig}
\usepackage{url}
\usepackage{fancyhdr}
\usepackage{pspicture}
\usepackage{graphicx}

\usepackage{cite}

\usepackage{amsmath,verbatim}

\usepackage{amsthm}

\usepackage{amssymb}

\usepackage{amsfonts}

\usepackage{dsfont}

\usepackage{hyperref}

\newtheorem{theorem}{Theorem}[section]

\newtheorem{lemma}[theorem]{Lemma}

\newtheorem{proposition}[theorem]{Proposition}

\theoremstyle{definition}

\newtheorem{remark}[theorem]{Remark}

\numberwithin{equation}{section}

\newcommand{\dist}{\mathrm{dist}}      

\newcommand{\diam}{\mathrm{diam}}      

\renewcommand{\Im}{{\ensuremath{\mathrm{Im\,}}}} 

\renewcommand{\Re}{{\ensuremath{\mathrm{Re\,}}}} 

\renewcommand{\div}{\mathrm{div}\,}    

\newcommand\restr[2]{{
  \left.\kern-\nulldelimiterspace 
  #1 
  \vphantom{\big|} 
  \right|_{#2} 
  }}

\usepackage{eucal}

\title[Photo-acoustic inversion using plasmonics]{Photo-acoustic inversion using plasmonic contrast agents:\\ The full Maxwell model}

\author[Ghandriche and Sini]{Ahcene Ghandriche  $^*$ and Mourad Sini$^{\ddag}$}
\thanks{$^*$ RICAM, Austrian Academy of Sciences, Altenbergerstrasse 69, A-4040, Linz, Austria. Email: ahcene.ghandriche@ricam.oeaw.ac.at. This author is supported by the Austrian Science Fund (FWF): P 30756-NBL}
\thanks{$^{\ddag}$ RICAM, Austrian Academy of Sciences, Altenbergerstrasse 69, A-4040, Linz, Austria. Email: mourad.sini@oeaw.ac.at. This author is partially supported by the Austrian Science Fund (FWF): P 30756-NBL}

\begin{document}

\date{\today}

\allowdisplaybreaks

\begin{abstract}
We analyze the inversion of the photo-acoustic imaging modality using electromagnetic plasmonic nano-particles as contrast agents. We prove that the generated pressure, before and after injecting the plasmonic nano-particles, measured at a single point, located away from the targeted inhomogeneity to image, and at a given band of incident frequencies is enough to reconstruct the (eventually complex valued) permittivity. Indeed, from these measurements, we define an indicator function which depends on the used incident frequency and the time of measurement. This indicator function has differentiating behaviors in terms of both time and frequency. First, from the behavior in terms of time, we can estimate the arrival time of the pressure from which we can localize the injected nano-particle. Second, we show that this indicator function has maximum picks at incident frequencies close to the plasmonic resonances. This allows us to estimate these resonances from which we construct the permittivity.
\bigskip

To justify these properties, we derive the dominant electric field created by the injected nano-particle when the incident frequency is close to plasmonic resonances. This is done for the full Maxwell system. To this end, we use a natural spectral decomposition of the vector space $(L^2(D))^3$ based on the spectra of the Newtonian and the Magnetization operators.  As another key argument, we provide the singularity analysis of the Green's tensor of the Maxwell problem with varying permittivity. Such singularity is unusual if compared to the ones of the elliptic case (as the acoustic or elastic models). In addition, we show how the derived approximation of the electric fields propagates, as a source, in the induced pressure with the time. 

\end{abstract}


\maketitle
\section{Introduction and statement of the results}
\subsection{Introduction}
The photo-acoustic experiment, in the general setting, applies to targets that are electrically conducting, in other words the imaginary part of the 'permittivity' is quite pronounce, and it goes as follows. Exciting the target, with laser, or by sending an incident electric field, will create heat in its surrounding. This heat, in its turn, creates fluctuations, i.e. a pressure field, that propagates along the body to image. This pressure can be collected in an accessible part of the boundary of the target. The photo-acoustic imaging is to trace back the pressure and reconstruct the permittivity that created it.       

In our settings, the source of the heat is given by the injected electromagnetic nano-particles. To describe the mathematical model behind this experiment, let us set $E$, $\bm{T}$ and $p$ to be respectively the electric field, the heat temperature and the acoustic pressure. Then, as described above, the photo-acoustic experiment is based on the following model coupling these three equations:  
\begin{equation*}\label{Photoacoustic-general-model}
\left \{
\begin{array}{llrr}

curl\; curl\; E -\omega^2\; \varepsilon\; \mu\; E=0,~~~ E:=E^s+E^i, \mbox{ in } \mathbb{R}^{3},\\

\rho_0 c_p\dfrac{\partial \bm{T}}{\partial t}-\nabla \cdot \kappa \nabla \bm{T} =\omega\; \Im(\varepsilon)\;\vert E \vert^2\; \delta_{0}(t),\; \mbox{ in } \mathbb{R}^{3}\times \mathbb{R}_+,\\

\dfrac{1}{c^2}\dfrac{\partial^2 p}{\partial t^2}-\Delta p= \rho_0\; \beta_0\; \dfrac{\partial^2 \bm{T}}{\partial t^2}, \mbox{ in } \mathbb{R}^{3}\times \mathbb{R}_+,
\end{array} \right.
\end{equation*}
 where $\rho_0$ is the mass density, $c_p$ the heat capacity, $\kappa$ is the heat conductivity, $c$ is the wave speed and $\beta_0$ the thermal expansion coefficient. 
To the last two equations, we supplement the homogeneous initial conditions $
 \bm{T} =p=\frac{\partial p}{\partial t}=0, \mbox{ at } t=0$ 
and the Silver-M\"{u}ller radiation condition to $E^s$. Under the condition that the heat conductivity is relatively small, the model above reduces to the following one:
\begin{equation}\label{pressurwaveequa}
\left\{
\begin{array}{rll}
    \partial^{2}_{t} p(x,t) - c_{s}^{2}(x) \underset{x}{\Delta} p(x,t) &=& 0 \qquad in \quad \mathbb{R}^{3} \times \mathbb{R}^{+},\\
    p(x,0) &=& \frac{\omega \, \beta_{0}}{c_{p}} \Im(\varepsilon)(x) \, \vert E \vert^{2}(x), \qquad in \quad \mathbb{R}^{3} \\ 
    \partial_{t}p(x,0) &=& 0 \qquad in \quad \mathbb{R}^{3} 
    \end{array}
\right.
\end{equation}
here $c_{s}$ is the velocity of sound in the medium that we assume to be a uniform constant. The constants $\beta_0$ and $c_p$ are known and $\omega$ is an incident frequency. The source $E$ is solution of the scattering problem
\begin{equation}
\label{eq:electromagnetic_scattering}
\left\{
\begin{array}{rll}
curl\; curl\; E -\omega^2\; \varepsilon\; \mu\; E=0,~~~ E:=E^s+E^i, \mbox{ in } \mathbb{R}^{3},\\
E^s(x) \mbox{ satisfies the Silver-M\"{u}ller radiation conditions},
\end{array}
\right.
\end{equation}
where $\varepsilon = \epsilon_p$ inside $D$, $\varepsilon = \epsilon_0(x)$ outside $D$ and $\epsilon_0(x)= \epsilon_{\infty} $ outside a bounded and smooth domain $\Omega$
 ($D \subset \Omega$ being the injected nano-particle with permittivity $\epsilon_p$ and permeability $\mu$). More details on the actual derivation of this model can be found in \cite{P-P-B:2015, Triki-Vauthrin:2017} and more references therein. The permittivity $\epsilon_0$ is variable and it is supposed to be smooth inside $\Omega$. The needed smoothness will be discussed later. 
\bigskip

From now on, we use the notation $u$ instead of $E$, i.e. $u:=E$.
\bigskip

 We have two classes of such nano-particles: dielectric and plasmonic nano-particles. The dielectric nano-particles enjoy the following features. They are highly localized as they are nano-scaled and they have high contrast permittivity.  Under these scales, we can choose the incident frequency so that we excite the dielectric resonances which are related to the eigenvalues of the Newtonian operator. The main feature of the plasmonic nano-particles is that they enjoy negative values of the real part of their permittivity if we choose incident frequencies close to the plasmonic frequencies of the nano-particle. With such negative permittivity, we can excite the plasmonic resonances which are related to the eigenvalues of the Magnetization operator. To describe this, we use the Lorentz model where the permeability $\mu$ is kept constant as the one of the homogeneous background while the permittivity has the form:
\begin{equation}\label{Lorentz-model}
\epsilon_{p} =\epsilon_\infty \left( 1+\frac{\omega^2_p}{\omega^2_0-\omega^2+i \gamma \omega} \right)
\end{equation} 
where $\omega_p$ is the electric plasma frequency, $\omega_0$ is the undamped frequency and $\gamma$ is the electric damping frequency. We observe that if we choose the incident frequency $\omega$ so that $\omega^2$ is larger than $\omega_0^2$, then the real part becomes negative. For such choices of the incident frequency, the nano-particle behaves as a plasmonic nano-particle.


\bigskip

The goal of the photo-acoustic imaging using nano-particles is to recover $\epsilon_{0}(\cdot)$ in $\Omega$ from the measure of the pressure $p(x, t),\; x \in \partial \Omega$ and $t \in (0, T)$ for large enough $T$. The decoupling of the original photo-acoustic mathematical model (\ref{Photoacoustic-general-model}) into (\ref{pressurwaveequa})-(\ref{eq:electromagnetic_scattering}) suggests that we split the inversion into the following two steps.

\begin{enumerate}
\item Acoustic Inversion: Recover the source term $\Im(\varepsilon)(x) \, \vert u \vert^{2}(x)$, $x \in \Omega$, from the measure of the pressure $p(x, t),\; x \in \partial \Omega$ and $t \in (0, T)$.
\bigskip

\item Electromagnetic Inversion: Recover the permittivity $\epsilon(x),\; x\in \Omega$ from  $\Im(\varepsilon)(x) \, \vert u \vert^{2}(x)$, $x \in \Omega$. 
\end{enumerate}
\bigskip

The pressure is collected on the boundary of $\Omega$ in the following situations:
\bigskip

\begin{itemize}

\item Before injecting any particle. The measured data is the pressure $p(x, t),\; x \in \partial \Omega$ and $t \in (0, T)$ without injecting any particle. There is a large literature based on such data. Without being exhaustive, we cite the following references \cite{Habib-book, B-E-K-S:2018, B-G-S:2016, B:2014, B-B-M-T, B-U:2010, C-A-B:2007, FHR, Kirsch-Scherzer, K-K:2010, KuchmentKunyansky, N-S:2014, Natterer, S:2010, S-U:2009} devoted to such inversions. The general approach is that, using the Radon transform, one can recover the initial pressure, i.e. $\Im(\varepsilon)(x) \, \vert E \vert^{2}(x)$, $x \in \Omega$. The next step is to use these internal values to recover the permittivity $\epsilon_0(\cdot)$. 
\bigskip

\item After injecting  nano-particles. The measured data is the pressure $p(x, t),\; x \in \partial \Omega$ and $t \in (0, T)$ after injecting a nano-particle. The first work in this direction is \cite{Triki-Vauthrin:2017} where plasmonic nano-particles  are used and an optimization method was proposed to invert the electric energy fields. There, the $2D$-model is stated and the magnetic field was used. Assuming the initial pressure to be already given, via one of the inversion methods as the Radon transform for instance, the authors propose a reconstruction method to recover the permittivity from the modulus of the electric (or the magnetic) field given in and around the particle $D$. For this, they use the contrasting behavior of the magnetic field across the interface of the particle. 
\bigskip

\item Before and after injecting nano-particles. The measured data is the pressure $p(x, t),\; x \in \partial \Omega$ and $t \in (0, T)$ before and then after injecting the nano-particle. In \cite{Ghandriche, Ahcene-Mourad-IICM}, we considered the $2D$-model using dielectric nano-particles. There, we did not split the problem into two steps. Rather, we derived direct formulas linking the measured pressure collected only on a {\it{single point}} $x$ on the accessible surface, to the internal values of the modulus of the electric field. In addition, using dimers (two close nano-particles), we showed that we can reconstruct, not only the electric field, but also the values of the (real part of the ) Green's functions on the centers of the dimer's nano-particles. From this Green function, we recover the permittivity. The main argument there is that under critical scales, on the size and the high values of the permittivity, we can choose the incident frequency so that we excite the dielectric resonances which are related to the eigenvalues of the Newtonian operator. Here, we propose to use plasmonic nano-particles. Measuring the induced pressure before and after injecting such a nano-particle, on a single point of $\partial \Omega $ but a band of frequencies, and taking their difference, we show that the generated curve has picks on incidence frequencies close to singular frequencies related to the eigenvalues of the Magnetization operator (that are the plasmonic resonances). With such behavior, we can construct those resonances. From these resonances, we extract the values of the permittivity. More details are given in section \ref{Inversion-method}.

\end{itemize}

\subsection{Statement of the results}
Let $\Omega$ be a $\mathcal{C}^2$-smooth and bounded domain. The nano-particle $D$ is taken of the form $D:=a\; B +z$ where $z$ models its location and $a$ its relative radius with $B$ as $\mathcal{C}^2$-smooth domain of maximum radius $1$. 
For later use, we introduce the integral operators of the volume potential $N^{k}(\cdot)$ and the Magnetization potential $\nabla M^{k}(\cdot)$, both acting on vector fields:
\begin{equation}\label{DefNDefMk}
N^{k}(f)(x):=\int_{B} \Phi_{k}(x,y) \, f(y)dy \quad \text{and} \quad  \nabla M^{k}(f)(x):=\nabla \int_{B}\underset{y}{\nabla}\Phi_{k}(x,y) \cdot f(y)dy,
\end{equation} 
where $\Phi_{k}(x,y) := \frac{e^{ik\left\vert x-y \right\vert}}{4\pi \vert x-y\vert}$ is the fundamental solution for Helmholtz equation in the entire space. Particularly, for $k = 0$ we obtain: 
\begin{equation}\label{DefNDefM}
N(f)(x):=\int_B\frac{1}{4\pi \vert x-y\vert} f(y)dy~~~~~~ \nabla M(f)(x):=\nabla \int_{B}\underset{y}{\nabla}\left( \frac{1}{4\pi \vert x-y\vert} \right) \cdot f(y)dy.
\end{equation} 
We recall the decomposition $(\mathbb{L}^2(B))^3=\mathbb{H}_{0}(\div = 0)(B) \oplus \mathbb{H}_{0}(Curl = 0)(B) \oplus \nabla \mathcal{H}armonic(B)$ where  
$\mathbb{H}_{0}(\div = 0)(B):=\{u \in \mathbb{H}(\div)(B); \nu \cdot u=0 \mbox{ on } \partial B\}$, $\mathbb{H}_{0}(Curl = 0) :=\{u \in \mathbb{H}(Curl)(B); \nu \times u=0 \mbox{ on } \partial B\}$ and 
$\nabla \mathcal{H}armonic(B):=\{u=\nabla \phi,~~ \Delta \phi =0 \mbox{ in } B\}$.

We can show, see later, that $N_{\displaystyle|_{\mathbb{H}_{0}\left(\div=0 \right)}}$ and $N_{\displaystyle|_{\mathbb{H}_{0}\left(Curl =0 \right)}}$ generate complete orthonormal bases $(\lambda_n^{1}, e^{1}_{n})_{n \in N}$ and $(\lambda_n^{2}, e^{2}_{n})_{n \in N}$ of $\mathbb{H}_{0}\left(\div=0 \right)$ and $\mathbb{H}_{0}\left(Curl =0 \right)$ respectively. In addition, it is known that $\nabla M: \nabla\mathcal{H}armonic \rightarrow \nabla\mathcal{H}armonic$ has a complete basis $(\lambda_n^{3}, e^{3}_{n})_{n \in N}$.

The permittivity $\varepsilon(\cdot)$ is defined as 
\begin{equation}\label{DefPermittivityfct}
\varepsilon(x) := \begin{cases} 
\epsilon_{\infty} & in \quad \mathbb{R}^{3} \setminus \Omega, \\
\epsilon_{0}(x) & in \quad \Omega \setminus D, \\
\epsilon_{p} & in \quad D,
\end{cases}
\end{equation}
where
\begin{equation*}\label{plasmonic}
\epsilon_p := \epsilon_{\infty}\left( 1+\frac{\omega^2_p}{\omega^2_0-\omega^2+i \omega \gamma} \right)
\end{equation*}
with $\omega_p$ as the electric plasma frequency, $\omega_0$ as the undamped frequency and $\gamma$ as the electric damping frequency.

Related to this, we set the index of refraction $\bm{n}$, in $\mathbb{R}^{3}$, given by\footnote{For $z \in \mathbb{C}$, given by $z = r \, e^{i \phi}$ with $- \pi < \phi \leq \pi$, the principal square root of $z$ is defined to be: $\sqrt{z} = \sqrt{r} \, e^{i \frac{\phi}{2}}$.} 
\begin{align}\label{Def-Index-Ref}
  \bm{n} := \begin{cases}
      \sqrt{\epsilon_{p} \, \mu} & \text{in $D$} \\
      \bm{n}_{0} & \text{in $\mathbb{R}^{3} \setminus D$}
    \end{cases} 
    \qquad \text{and} \qquad \bm{n}_{0} := \begin{cases}
      \sqrt{\epsilon_{0}(\cdot) \, \mu} & \text{in $\Omega$} \\
      \sqrt{\epsilon_{\infty} \, \mu} & \text{in $\mathbb{R}^{3} \setminus \Omega$}
    \end{cases}. 
\end{align}

We assume $\epsilon_0(\cdot)$ to be of class $\mathcal{C}^1$. Let $z\in \Omega$ and define
\begin{equation*}
f_n(\omega, \gamma):=\epsilon_0(z) - (\epsilon_0(z)-\epsilon_p)\lambda^{3}_n.
\end{equation*}
We show that in the square $\left( \omega_{0};~~\sqrt{ \omega^{2}_{p} + \omega^{2}_{0}} =: \omega_{max} \right) \times \left(0; ~~ \omega_{max} \, \Vert \frac{\Im \left(\epsilon_{0}(\cdot)\right) }{\Re \left(\epsilon_{0}(\cdot)\right)} \Vert_{L^{\infty}(\Omega)} =:\gamma_{max} \right) $, the dispersion equation $f_n(\omega, \gamma)=0$ has one and only one solution. For any $n_0$ fixed, we set $(\omega_{n_0}, \gamma_{n_0})$ to be the corresponding solution for $n=n_0$. 

\begin{theorem}\label{PrincipalTHM}
We assume $\Omega$ and $B$ (and hence $D$) to be of class $\mathcal{C}^2$. In addition, we assume $\epsilon_0(\cdot)$ to be of class $\mathcal{C}^1$ and satisfies the conditions\footnote{The first condition is a natural one in applications. The second condition is needed to derive and analyze (the singularity of the) the Green's kernel for the inhomogeneous Maxwell system.}
\begin{equation*}\label{condition-epsilon-0}\Re \epsilon_{0}(\cdot) >\epsilon_{\infty} ~~\mbox{ and}~~ \Vert \epsilon_{0}(\cdot) - \epsilon_\infty\Vert_{\mathbb{L}^{\infty}(\Omega)} \leq C
\end{equation*} 
where the positive constant $C:=C(\Omega)$ is given by $(\ref{Cdtalpha})$ and depends on $\Omega$ through the mapping property of the Newtonian operator. 
\bigskip

Let the used incident and damping frequencies $(\omega, \gamma)$ be such that 
\begin{equation}\label{freq-close-plasm}
\omega^2-\omega^2_{n_{0}} \sim a^h ~\mbox{ and }~ \gamma-\gamma_{n_0}\sim a^h \mbox{ for } h\in (0, 1). 
\end{equation}
1) We have the following approximation of the electric field 
\begin{equation}\label{THM-KMYHAYD}
\int_{D} \left\vert u_{1} \right\vert^{2}(x) \, dx = a^{3} \; \frac{\left\vert  \epsilon_{0}(z) \right\vert^{2} \, \left\vert u_{0}(z) \cdot \int_{B} e^{(3)}_{n_{0}}(x)dx \right\vert^{2}}{\left\vert  \epsilon_{0}(z) - \left(  \epsilon_{0}(z) - \epsilon_{p} \right) \, \lambda^{(3)}_{n_{0}}  \right\vert^{2}} + \mathcal{O}\left(a^{\min(3,4-3h)} \right).
\end{equation}
2) Let $x \in \partial \Omega$ and $s \geq \diam(D) + \dist(x,D)$. We have the following approximation of the average pressure:  
\begin{equation}\label{PrincipalFormula}
p^{\star}(x,s) - p^{\star}_{0}(x,s) = \frac{a^{3}}{4 \, \pi} \, \Im\left( \epsilon_{p} \right) \,  \frac{\left\vert \epsilon_{0}(z) \right\vert^{2} \, \left\vert \langle u_{0}(z) ; \int_{B} e^{(3)}_{n_{0}}(x) \, dx \rangle \right\vert^{2}}{\left\vert \epsilon_{0}(z) - \left(\epsilon_{0}(z) - \epsilon_{p} \right) \,   \lambda^{(3)}_{n_{0}} \right\vert^{2}} + \mathcal{O}\left(  a^{\min(3-h,4-3h)}  \right),  
\end{equation}
where 
\begin{equation*}
p^{\star}(x,s) := \int_{0}^{s} r \, \int_{0}^{r} p(x,t) \, dt \, dr \quad \text{and} \quad p^{\star}_{0}(x,s) := \int_{0}^{s} r \, \int_{0}^{r} p_0(x,t) \, dt \, dr
\end{equation*}
with $p_0(\cdot,\cdot)$ being the pressure generated by the medium in the absence of the nano-particle.
\end{theorem}

To justify these results, we need to derive the dominating fields in both the acoustic and the electromagnetic models that constitute the photo-acoustic model. The most difficult issue is in deriving the dominating electric field generated by the plasmonic resonances. These are related to the eigenvalues of the Magnetization operator, at zero wave number, restricted to the sub-space of grad-harmonic functions. To do this, we first use the natural decomposition of the $\mathbb{L}^{2}$-space as $\mathbb{L}^{2}=\mathbb{H}_{0}\left(\div = 0 \right) \oplus \mathbb{H}_{0}\left(Curl = 0 \right) \oplus \nabla\mathcal{H}armonic $ to which we correspond a spectral decomposition given by the eigenvalues-eigenfunctions of the vector Newtonian operator restricted to $\mathbb{H}_{0}\left(\div = 0 \right)$ and $\mathbb{H}_{0}\left(Curl = 0 \right)$ and the above mentioned eigenvalues-eigenfunctions of the Magnetization operator. As the electromagnetic background is inhomogeneous, unlike the elliptic models, as in acoustics or elasticity for instance, the corresponding Green's kernel has unusual singularities. We derive two main decompositions of this kernel in terms of the singularities. The first one is valid in $\mathbb{R}^3$ as a sum of the explicit kernel, of the Maxwell system with constant electromagnetic parameters, and a kernel\footnote{For $s>1$ we denote by $\mathbb{L}^{s-}$ the space of functions belonging to $\mathbb{L}^{s-\delta}$ for every $\delta$ such that $0<\delta<s-1$.} having $L^{{\frac{3}{2}}{-}}$ integrable singularity. This allows us to define and justify the equivalent Lippmann-Schwinger equation. The second one is valid locally, near any fixed source point, as a sum of three kernels. The first one is the kernel of the grad-harmonic operator, i.e. the kernel of the Magnetization operator, the second is of the form $\nabla K$, with $K$  having an $L^{3{-}}$ integrable singularity and the third one with an $L^{3{-}}$ integrable singularity as well. This last decomposition is key in deriving the dominating electric field generated by the plasmonic resonances. These decompositions of the Green's kernel, see for instance $(\ref{ExpansionofGkII})$, are more precise and general than the ones derived in the case of piecewise constant permittivity, see for instance \cite{cutzach1998existence} and \cite{Kirsch}.

\subsection{Inversion of the photo-acoustic imaging modality using plasmonic contrast agents}\label{Inversion-method}

Here, we discuss how to use the approximation formulas we derived in the previous section to the actual inversion procedure for the photo-acoustic modality using plasmonic nano-particles as contrast agents.
\bigskip

\begin{enumerate}
\item Let $x \in \partial \Omega$ be fixed. Let also $(\omega, \gamma)$ be any couple of incident and damping frequencies. 
\bigskip

\begin{enumerate}

\item If $s < dist(x, D)$, then $p^{\star}(x,s)-p_0^{\star}(x,s)= \mathcal{O}\left( a^{3-h} \right)$ for any  $(\omega, \gamma)$. This property, which is related to the finiteness of the speed of propagation, can be shown by combining (\ref{Finite-speed-propagation}) and {\bf{Lemma}} \ref{Anal-T}. 
\bigskip

\item \label{NVIM} If $s\geq \diam(D) +dist(x, D),$ then, \emph{under the condition of the existence of $n_0 \in \mathbb{N}$ such that $ \int_{B} e^{(3)}_{n_0}(x) \, dx \neq 0 $}, we have $p^{\star}(x,s)-p_0^{\star}(x,s) \sim a^{3-2h}$, for any $(\omega, \gamma)$ close to $(\omega_{n_0}, \gamma_{n_0})$ as in (\ref{freq-close-plasm}). This comes from
(\ref{PrincipalFormula}).
\end{enumerate}
\bigskip

From these formulas, we can estimate $dist(x, D)$ with an error of the order of $diam(D)\sim a$. Therefore, measuring $p^{\star}(x,s)-p_0^{\star}(x,s)$ for three different points $x_1, x_2, x_3$ on $\partial \Omega$, we can localize the injected nano-particle with an error of the order $a$.
\bigskip

\item Let $x \in \partial \Omega$ and $s\geq \diam(D) +dist(x, D)$ be fixed, we define 
\begin{equation*}
I_z(\omega, \gamma):=\vert p^{\star}(x,s, \omega)-p_0^{\star}(x,s, \omega)\vert
\end{equation*}
on the square $\left( \omega_{0};~~\sqrt{ \omega^{2}_{p} + \omega^{2}_{0}} := \omega_{max} \right) \times \left(0; ~~ \omega_{max} \, \Vert \frac{\Im \left(\epsilon_{0}(\cdot)\right) }{\Re \left(\epsilon_{0}(\cdot)\right)} \Vert_{L^{\infty}(\Omega)} := \gamma_{max} \right)$.
\bigskip

According to (\ref{PrincipalFormula}), this functional has a sequence of picks $(\omega_n, \gamma_n), n=1, 2, ...$
Observe that the index $n$ is related to one of the eigenvalue $\lambda^{(3)}_{n}$ of $\nabla M$. From \textbf{Lemma} $\ref{LemmaClaim}$, we know that 
\begin{equation}\label{monotonicity-introduction}
\lambda^{(3)}_{n} < \lambda^{(3)}_{m} \Rightarrow \omega^2_n <\omega^2_m. 
\end{equation}
Therefore from the picks $(\omega_n, \gamma_n)$ of the functional $I_z(\omega, \gamma)$ we can choose anyone of them, say $(\omega_{n_0}, \gamma_{n_0})$. From (\ref{monotonicity-introduction}), to $\omega_{n_0}$ we correspond a unique $\lambda^{(3)}_{n_0}$, via the ordering of $\lambda^{(3)}_{n}$'s. From $f_{n_0}(\omega_{n_0}, \gamma_{n_0})=0$, we obtain
\begin{equation*}
\epsilon_0(z)= - \, \frac{\epsilon_p \, \lambda^{(3)}_{n_0}}{1-\lambda^{(3)}_{n_0}}.
\end{equation*}
\end{enumerate}

\smallskip

Observe that the validity of the imaging procedure works for nano-particles for which $ \int_{B} e^{(3)}_{n}(x) \, dx \neq 0 $ for some $n$'s, see the condition mentioned in $(\ref{NVIM})$. This condition can be clarified for particular shapes. For nano-particles $B$ of ellipsoidal shape, with semi-axes given by $r_{1}, r_{2}$ and $r_{3}$ we use the fact that: 
\begin{equation}
N(1)(x) = \frac{r_{1} \, r_{2} \, r_{3}}{4} \, \int_{0}^{\infty} \left( 1 - \sum_{j=1}^{3} \frac{x^{2}_{j}}{\left( s + r^{2}_{j} \right)} \right) \, \frac{1}{\sqrt{\left( s + r^{2}_{1} \right) \, \left( s + r^{2}_{2} \right) \, \left( s + r^{2}_{3} \right)}}  \, ds, \quad x \in B,  
\end{equation}
see for instance Theorem 1.1 of \cite{shahgholian1991newtonian}. 
Therefore, by  straightforward computations, using the relation $\nabla M \left( I  \right) = - \nabla \div( N (I) ) =-\nabla \div ( N(1) I) $, we derive:
\begin{equation}\label{RD}
\nabla M \left( I \right)(x) = \frac{r_{1} \, r_{2} \, r_{3}}{2} \begin{pmatrix}
\mathcal{I}_{1}(r_{1}, r_{2}, r_{3}) & 0 & 0 \\
0 & \mathcal{I}_{2}(r_{1}, r_{2}, r_{3}) & 0 \\
0 & 0 & \mathcal{I}_{3}(r_{1}, r_{2}, r_{3})
\end{pmatrix}, \,\, x \in B,
\end{equation}
where, for $j=1,2,3$, we have: 
\begin{equation}\label{RDIj}
\mathcal{I}_{j}(r_{1}, r_{2}, r_{3}) := \int_{0}^{+\infty} \frac{1}{ \left(s + r^{2}_{j} \right)} \, \frac{1}{\sqrt{\left(s + r^{2}_{1} \right)\,\left(s + r^{2}_{2} \right)\,\left(s + r^{2}_{3} \right)}} \, ds.
\end{equation}  
Using the fact that $\lambda^{(3)}_{n}  \, e^{(3)}_{n} = \nabla M \left( e^{(3)}_{n} \right)$ we get, 
\begin{equation*}
\lambda^{(3)}_{n}  \, I \cdot \int_{B} e^{(3)}_{n}(x) \, dx = \int_{B} \lambda^{(3)}_{n}  \, e^{(3)}_{n}(x) \, dx =  \int_{B} \nabla M \left( e^{(3)}_{n} \right)(x) \, dx =  \int_{B} I \cdot \nabla M \left( e^{(3)}_{n} \right)(x) \, dx,
\end{equation*}
which, by using the self-adjointness of the Magnetization operator, becomes 
\begin{equation}\label{LambdaInt=}
\lambda^{(3)}_{n}  \, I \cdot \int_{B} e^{(3)}_{n}(x) \, dx =  \int_{B}  \nabla M \left( I \right)(x) \cdot e^{(3)}_{n}(x) \, dx.
\end{equation}
Now, thanks to the constancy of  $\nabla M \left( I \right)$ inside $B$, see $(\ref{RD})$, we deduce that:
\begin{equation*}
\left( \lambda^{(3)}_{n}  \, I - \frac{r_{1} \, r_{2} \, r_{3}}{2} \begin{pmatrix}
\mathcal{I}_{1}(r_{1}, r_{2}, r_{3}) & 0 & 0 \\
0 & \mathcal{I}_{2}(r_{1}, r_{2}, r_{3}) & 0 \\
0 & 0 & \mathcal{I}_{3}(r_{1}, r_{2}, r_{3})
\end{pmatrix} \right) \cdot \int_{B} e^{(3)}_{n}(x) \, dx = \begin{pmatrix}
0 \\
0 \\
0
\end{pmatrix}.
\end{equation*}
As the matrix on the left hand side is a diagonal one, we deduce that we can have at most three eigenvalues, $\lambda^{(3)}_{n}$, for which the corresponding eigenfunctions might have the property $ \int_{B} e^{(3)}_{n}(x) \, dx \neq 0 $. The computation of $\int_{B} e^{(3)}_{n}(x) \, dx$ in the case of ellipsoidal shape is a difficult task. We restrict our computations to the particular case of a unit ball, which corresponds to take $r_{1} = r_{2} = r_{3} = 1$. By straightforward computations, using the definition of $\nabla M \left( I \right)$, see $(\ref{RD})$, and the formula of $\mathcal{I}_{j}(r_{1}, r_{2}, r_{3})$, see $(\ref{RDIj})$, we obtain  
\begin{equation}\label{NablaMIBall}
\nabla M (I)(x) = \frac{1}{3} \, I, \quad x \in B. 
\end{equation}
Therefore, $\frac{1}{3}$ is the only eigenvalue for which the corresponding eigenfunctions might have non-zero average. In addition, using the fact that $e^{(3)}_{n} = \nabla SL(u_{n})$, where $SL$ is the Single Layer operator and $u_{n}$ are the eigenfunctions of the Double Layer operator, which are computed explicitly in several references, see for example \cite{Ritter}, we derive that: 
\begin{equation}\label{IntBen3}
\int_{B} e^{(3)}_{n}(x) \, dx = \frac{2}{9} \, \sqrt{\frac{\pi}{3}} \, \begin{pmatrix}
0 \\ 0 \\ 1 
\end{pmatrix} \delta_{n,1},
\end{equation}  
where $\delta_{n,1}$ is the Kronecker symbol.
\bigskip

Combining $(\ref{NablaMIBall}), (\ref{IntBen3})$ and $(\ref{LambdaInt=})$ we end up with $\lambda^{(3)}_{1} = \frac{1}{3}$, which is the first eigenvalue of the Magnetization operator $\nabla M (\cdot)$ in the unit ball, and the corresponding eigenfunctions have non-zero average. 
\bigskip

Hence, in the imaging procedure described above, using a nanoparticle with a spherical shape, the functional 
\begin{equation*}
I_z(\omega, \gamma):=\vert p^{\star}(x,s, \omega)-p_0^{\star}(x,s, \omega)\vert
\end{equation*}
{\it{has one and only one pick}} in the square $\left( \omega_{0};~~\sqrt{ \omega^{2}_{p} + \omega^{2}_{0}} := \omega_{max} \right) \times \left(0; ~~ \omega_{max} \, \Vert \frac{\Im \left(\epsilon_{0}(\cdot)\right) }{\Re \left(\epsilon_{0}(\cdot)\right)} \Vert_{L^{\infty}(\Omega)} := \gamma_{max} \right)$.
\bigskip

\begin{remark}\label{Drude-works-too}
It is worth noticing that the results above can also be derived using the Drude model for the permittivity, see for instance \cite{engheta2006metamaterials} formula (1.5), instead of the Lorentz model. 
\end{remark}
\bigskip

\begin{remark}\label{Other-options-regarding-gamma}
In our analysis, as we have seen, it is mandatory to vary both the incident frequencies $\omega$ and the damping frequencies $\gamma$. This is can turn out to be expensive from the point of view of applications as it would mean that one should change the nano-particles to change $\gamma$. In the following, we give two ways to overcome this eventual issue.

\begin{enumerate}
\item In the case where $\Im\left(\epsilon_{0}(z) \right)$ is very small, or mathematically zero, then we can take the damping frequency $ \gamma $ small as well but fixed. In this case, we only need to vary the incident frequency. Observe that the traditional photo-acoustic experiment applies only to electrically highly conducting tissues.  However, the photo-acoustic experiment based on using contrast agents can deal with non-conductive tissues as well (as for benign or early stage tumors). 
\bigskip

\item Instead of varying both the incident and damping frequencies, we allow the frequencies $\omega$ to be in the complex plan. In this case, we can derive similar reconstruction formulas. 

\end{enumerate}
More details can be found in Remark \ref{Other-options}.
\end{remark}
\bigskip

The remaining part of the manuscript is divided as follows. In Section \ref{Section-1}, we give the proof of Theorem \ref{PrincipalTHM} postponing the construction of the Green's tensor, the invertibility of the Lippmann-Schwinger equation and certain a priori estimates of the electric fields to the next sections. In Section \ref{Green's-function-and-LSE}, we construct the Green's tensor for our Maxwell model and provide its singularity analysis. These properties are used then to derive, and give sense to, the Lippmann-Schwinger equation. In Section \ref{proof-pro-2.1}, we prove the a priori estimates used in the proof of Theorem \ref{PrincipalTHM}. 
In Section \ref{Appendix}, we provide the spectral decomposition of the vector space $(L^2(D))^3$ based on the eigenvalues of the Newtonian and Magnetization operators. In addition, we analyze the dispersion equations $f_{n}(\omega, \gamma)=0$ used also in the proof of Theorem \ref{PrincipalTHM}.

\section{Proof of Theorem \ref{PrincipalTHM}}\label{Section-1}

\subsection{Approximation of the Lippmann Schwinger equation and proof of (\ref{THM-KMYHAYD})}\label{section2LSE}\
\\
As shown in Section \ref{Subsection-Distribution} the solution, in distributional sense, of $(\ref{eq:electromagnetic_scattering})$ can be written as solution of the following Lippmann-Schwinger equation
\begin{equation}\label{SK0}
u_{1}(x) + \omega^{2}  \, \int_{D} G_{k}(x,y) \cdot u_{1}(y) \, (\bm{n}_{0}^{2}(y) - \bm{n}^{2}(y) ) \, dy = u_{0}(x), \quad x \in \mathbb{R}^{3}, 
\end{equation} 
where $G_{k}(\cdot,\cdot)$ is the Green kernel for Maxwell's equation for the inhomogeneous background, defined as solution of: 
\begin{equation*}
\underset{y}{\nabla} \times \underset{y}{\nabla} \times G_{k}(x,y) - \omega^{2} \, \bm{n}_{0}^{2}(y) \, G_{k}(x,y) = \underset{x}{\delta}(y) \, \bm{I}, \;\, x, y \in \mathbb{R}^{3}, 
\end{equation*}
such that each column of $G_{k}(\cdot,\cdot)$ satisfies the outgoing radiation condition 
\begin{equation*}
\lim_{\left\vert x \right\vert \rightarrow +\infty} \;\; \left\vert x \right\vert \;\; \left( \underset{y}{\nabla} \times G_{k}(x,y) \times \frac{x}{\left\vert x \right\vert} - i \, k \, G_{k}(x,y) \right) = 0.
\end{equation*}
Recalling the definition of the index of refraction $\bm{n}(\cdot)$, see $(\ref{Def-Index-Ref})$, we get
\begin{equation}\label{SK}
u_{1}(x) + \omega^{2} \, \mu \, \int_{D} G_{k}(x,y) \cdot u_{1}(y) \, (\epsilon_{0}(y)-\epsilon_{p}) \, dy = u_{0}(x), \quad x \in \mathbb{R}^{3}, 
\end{equation} 
The formal representation of the convolution part is justified in Section \ref{Subsection-Distribution} as well. The following theorem, on the singularity analysis of the Green's function, is of importance to reduce the complexities of the integral equation $(\ref{SK})$, and then invert it.  
\begin{theorem}\label{DH} 
The Green kernel $G_{k}(\cdot, \cdot)$ admit the following decomposition:
\begin{equation}\label{DecompositionGreenKernel}
G_{k}(x,z) = \Upsilon(x,z) + \Gamma(x,z), \quad x \neq z,  
\end{equation}
where $\Upsilon(\cdot,\cdot)$ is the kernel defined by
\begin{equation}\label{AI-Homogeneous}
\Upsilon(x,z) := \frac{1}{\omega^{2} \, \mu \, \epsilon_{0}(z)} \, \underset{x}{\nabla} \, \underset{x}{\div} \left(  \Phi_{0}(x,z) \, \bm{I} \right), \;\; x \neq z, 
\end{equation}
and the remainder part $\Gamma(\cdot, \cdot)$, for all element $x$ near $z$, is given by: 
\begin{equation}\label{AY}
\Gamma(x,z) := \frac{-1}{\omega^{2} \, \mu \, \left(\epsilon_{0}(z)\right)^{2}} \, \underset{x}{\nabla} \, \underset{x}{\nabla} M \left(  \Phi_{0}(\cdot,z) \, \nabla \epsilon_{0}(z) \right)(x) + W_{4}(x,z), \;\; x \neq z, 
\end{equation}
where, for an arbitrarily and sufficiently small positive $\delta$, the term $W_{4}(\cdot,z)$ is an element in $\mathbb{L}^{\frac{3(3-2\delta)}{(3+2\delta)}}(D)$.
\end{theorem}
\begin{proof}
The proof and details about the decomposition of the kernel $G_{k}(\cdot,\cdot)$ are given in Section $\ref{AppendixGreenKernel}$.
\end{proof}
We restrict the study of the equation $(\ref{SK})$ to the domain $D$ and we use the decomposition $(\ref{DecompositionGreenKernel})$ to rewrite it  as:
\begin{equation}\label{Ups=Hess+I}
u_{1}(x) + \omega^{2} \, \mu \, \int_{D} \Upsilon(x,y) \cdot u_{1}(y) \, (\epsilon_{0}(y)-\epsilon_{p}) \, dy = u_{0}(x) + Err_{\Gamma}(x),
\end{equation}
where $Err_{\Gamma}(x)$ is the vector field given by  
\begin{equation}\label{Err-Gamma}
Err_{\Gamma}(x) := - \omega^{2} \, \mu \, \int_{D} \Gamma(x,y) \cdot u_{1}(y) \, (\epsilon_{0}(y) - \epsilon_{p}) \, dy, \quad x \in D.
\end{equation}
Now, using the expression $(\ref{AI-Homogeneous})$, we  reformulate $(\ref{Ups=Hess+I})$ as 
\begin{equation*}
u_{1}(x) -  \underset{x}{\nabla} \int_{D} \underset{y}{\nabla} \Phi_{0}(x,y) \cdot u_{1}(y) \, \frac{\left( \epsilon_{0}(y) - \epsilon_{p} \right)}{\epsilon_{0}(y)} \, dy = u_{0}(x) + Err_{\Gamma}(x).
\end{equation*}
We set 
\begin{equation}\label{Defeta}
\eta(\cdot) := \frac{\left( \epsilon_{0}(\cdot) - \epsilon_{p} \right)}{\epsilon_{0}(\cdot)}
\end{equation}
and use the definition of the Magnetization operator, see  $(\ref{DefNDefM})$, to rewrite the previous equation as: 
\begin{equation*}
u_{1}(x) - \nabla M \left(u_{1} \, \eta \right)(x)   \overset{}{=}  u_{0}(x) + Err_{\Gamma}(x), 
\end{equation*}
and then, by Taylor expansion for the function $\eta(\cdot)$ near the center $z$, we get
\begin{equation}\label{sppo}
u_{1}(x) - \eta(z)  \,  \,  \nabla M(u_{1})(x) = u_{0}(x) + Err_{0}(x) + Err_{\Gamma}(x), 
\end{equation}
where 
\begin{equation}\label{WN}
Err_{0}(x) := \nabla \, M \left(  u_{1}(\cdot) \, \int_{0}^{1} \, \nabla \, \eta(z+t(\cdot -z)) \cdot (\cdot -z) \, dt  \right)(x).
\end{equation} 
Set $W(\cdot)$ to be the scattering matrix defined by 
\begin{equation}\label{SAH}
W(\cdot) =  \left[ I - \, \overline{\eta(z)} \,   \nabla M \right]^{-1}\left( \bm{I} \right)(\cdot).
\end{equation}
Then, successively, taking the inverse of $\left[ I - \, \eta(z) \,   \nabla M \right]$, on both sides of $(\ref{sppo})$, integrating over $D$ the obtained equation and using the definition of the matrix $W(\cdot)$, we get 
\begin{equation}\label{intu1=intw+err2}
\int_{D} u_{1}(x) dx \, = \int_{D} W(x) \cdot \left[ u_{0}(x)  + Err_{0}(x) + Err_{\Gamma}(x) \right] \, dx = \int_{D} W(x) \, dx \cdot  u_{0}(z) + Err_{1},
\end{equation}
where 
\begin{equation*}
Err_{1} := \int_{D} W(x) \cdot \left[ \int_{0}^{1} \nabla u_{0}(z+t(x-z)) \cdot (x-z) dt  + Err_{0}(x) + Err_{\Gamma}(x) \right] \, dx.
\end{equation*}
Next, we estimate $Err_{1}$. For this, we split it as
\begin{equation}\label{BG}
Err_{1} := S_{1} + S_{2} + S_{3},
\end{equation}
then, we define each term and estimate it. More precisely, we have:
\begin{enumerate}
\item Estimation of: 
\begin{eqnarray}\label{TermS1}
\nonumber
S_{1} &:=& \int_{D} W(x) \cdot  \int_{0}^{1} \nabla u_{0}(z+t(x-z)) \cdot (x-z) \, dt \, dx \\ 
\left\vert S_{1} \right\vert & \leq & \left\Vert W \right\Vert_{\mathbb{L}^{2}(D)} \; \left\Vert \int_{0}^{1} \nabla u_{0}(z+t(\cdot -z)) \cdot (\cdot -z) \, dt \right\Vert_{\mathbb{L}^{2}(D)} = \mathcal{O}\left(a^{\frac{5}{2}} \, \left\Vert W \right\Vert_{\mathbb{L}^{2}(D)} \right).
\end{eqnarray}
\item Estimation of: 
\begin{eqnarray*}
S_{2} &:=&  \int_{D} W(x) \cdot Err_{0}(x) \, dx \\
S_{2} & \overset{(\ref{WN})}{=} &  \int_{D} W(x) \cdot \nabla M \left( u_{1}(\cdot) \, \int_{0}^{1} \, \nabla \, \eta(z+t(\cdot -z)) \cdot (\cdot - z) \, dt \right)(x) \, dx.
\end{eqnarray*}
We apply the Cauchy-Schwartz inequality to obtain: 
\begin{eqnarray*}
\left\vert S_{2} \right\vert & \lesssim & \left\Vert W \right\Vert_{\mathbb{L}^{2}(D)} \; \left\Vert \nabla M \left( u_{1}(\cdot) \, \int_{0}^{1} \, \nabla \, \eta(z+t(\cdot -z)) \cdot (\cdot - z) \, dt \right) \right\Vert_{\mathbb{L}^{2}(D)} \\
 & \overset{(\ref{NormMagnetization})}{\leq} & \left\Vert W \right\Vert_{\mathbb{L}^{2}(D)} \; \left\Vert  u_{1}(\cdot) \, \int_{0}^{1} \, \nabla \, \eta(z+t(\cdot -z)) \cdot (\cdot - z) \, dt  \right\Vert_{\mathbb{L}^{2}(D)}.
\end{eqnarray*}
Then, 
\begin{equation}\label{AILFB}
S_{2} =  \mathcal{O}\left(\left\Vert W \right\Vert_{\mathbb{L}^{2}(D)} \; a \; \left\Vert u_{1} \right\Vert_{\mathbb{L}^{2}(D)} \right).
\end{equation}
\item Estimation of: 
\begin{eqnarray}
\nonumber
S_{3} &:=& \int_{D} W(x) \cdot Err_{\Gamma}(x) \, dx \overset{(\ref{Err-Gamma})}{\simeq}  \int_{D} W(x) \cdot \int_{D} \Gamma(x,y) \cdot u_{1}(y) \, (\epsilon_{0}(y) - \epsilon_{p}) \, dy \, dx.
\end{eqnarray}
With the help of $(\ref{AY})$ we rewrite the previous formula as 
\begin{eqnarray*}
S_{3} & \simeq & \int_{D} W(x) \cdot \int_{D}  \underset{x}{\nabla} \left[   \underset{x}{\nabla} M\left(\Phi_{0}(\cdot ,y)  \nabla \epsilon_{0}(y) \right) \right](x) \cdot u_{1}(y) \,\frac{(\epsilon_{0}(y) - \epsilon_{p})}{\epsilon^{2}_{0}(y)}  \, dy \, dx \\ &+& \int_{D} W(x) \cdot \int_{D} W_{4}(x,y) \cdot u_{1}(y) \, (\epsilon_{0}(y) - \epsilon_{p}) \, dy \, dx. 
\end{eqnarray*}  
We split the previous formula as $S_{3} = S_{3,1} + S_{3,2}$, we define and we estimate each term. 
\begin{enumerate}
\item Estimation of:
\begin{eqnarray*}
S_{3,1} & := & \int_{D} W(x) \cdot \int_{D}  \underset{x}{\nabla} \left[   \underset{x}{\nabla} M\left(\Phi_{0}(\cdot ,y)  \nabla \epsilon_{0}(y) \right) \right](x) \cdot u_{1}(y) \,\frac{(\epsilon_{0}(y) - \epsilon_{p})}{\epsilon^{2}_{0}(y)}  \, dy \, dx \\ 
 & = &  \int_{D} W(x) \cdot \int_{D} \underset{x}{\nabla} \, \underset{x}{\nabla} \int_{D} \underset{t}{\nabla}\Phi_{0}(t,x) \cdot \nabla \epsilon_{0}(y) \Phi_{0}(t,y) \, dt  u_{1}(y)  \,\frac{(\epsilon_{0}(y) - \epsilon_{p})}{\epsilon^{2}_{0}(y)}  \, dy \, dx \\
& = & \int_{D} W(x) \cdot \underset{x}{\nabla} \, \underset{x}{\nabla} \int_{D}  \int_{D} \underset{t}{\nabla}\Phi_{0}(t,x)  \cdot \nabla \epsilon_{0}(y)  \Phi_{0}(t,y) \,  u_{1}(y)  \,\frac{(\epsilon_{0}(y) - \epsilon_{p})}{\epsilon^{2}_{0}(y)}  \, dt  \, dy \, dx  \\
& = & \int_{D} W(x) \cdot \underset{x}{\nabla} \, \underset{x}{\nabla} \int_{D}   \underset{t}{\nabla}\Phi_{0}(t,x)  \cdot \int_{D}  \Phi_{0}(t,y) \, \nabla \epsilon_{0}(y)   \otimes  u_{1}(y)  \,\frac{(\epsilon_{0}(y) - \epsilon_{p})}{\epsilon^{2}_{0}(y)} \, dy\, dt \, dx.
\end{eqnarray*}
For shortness, for every $y \in D$, we set 
\begin{equation}\label{defU1star}
u^{\star}_{1}(y) := \nabla \epsilon_{0}(y)   \otimes  u_{1}(y)  \,\frac{(\epsilon_{0}(y) - \epsilon_{p})}{\epsilon^{2}_{0}(y)} \, \in \mathbb{L}^{2}(D)  , 
\end{equation}
hence, 
\begin{eqnarray*}
S_{3,1} & = & \int_{D} W(x) \cdot \underset{x}{\nabla} \, \underset{x}{\nabla} \int_{D}   \underset{t}{\nabla}\Phi_{0}(t,x)  \cdot \int_{D}  \Phi_{0}(t,y) \, u^{\star}_{1}(y)   \, dy\, dt \, dx\\
& = & \int_{D} W(x) \cdot \underset{x}{\nabla} \, \underset{x}{\nabla} \int_{D}   \underset{t}{\nabla}\Phi_{0}(t,x)  \cdot \, N\left( u^{\star}_{1}\right)(t)  \, dt \, dx\\
& = & \int_{D} W(x) \cdot \underset{x}{\nabla} \, \underset{x}{\nabla} M\left( u^{\star}_{1}\right)(x)   \, dx.
\end{eqnarray*}
By an integration by parts, we rewrite the last formula as: 
\begin{equation*}\label{LMTT}
S_{3,1} = - \int_{D} W(x) \cdot  \nabla \, \nabla N\left( \div N \left( u^{\star}_{1} \right) \right)(x) \, dx + \int_{D} W(x) \cdot \nabla \, \nabla SL\left(\nu \cdot N \left( u^{\star}_{1} \right) \right)(x) \, dx.
\end{equation*}
Again, we split $S_{3,1}$ as $S_{3,1} = S_{3,1,1} + S_{3,1,2}$, we define each term and we estimate it.\\ As $u^{\star}_{1} \in \mathbb{L}^{2}(D)$ the function $\nabla \, \nabla N\left( \div N\left(u^{\star}_{1}\right)\right) \in  \mathbb{L}^{2}(D)$. This implies,  
\begin{equation*}
\left\vert S_{3,1,1} \right\vert  :=  \left\vert \int_{D} W(x) \cdot \nabla \, \nabla N\left( \div N\left(u^{\star}_{1}\right)\right)(x)  \, dx \right\vert  \leq  \left\Vert W \right\Vert_{\mathbb{L}^{2}(D)} \; \left\Vert \nabla \, \nabla N\left( \div N\left(u^{\star}_{1}\right)\right) \right\Vert_{\mathbb{L}^{2}(D)},
\end{equation*}
and, from Calderon-Zygmund inequality, see \cite{gilbarg2001elliptic}, page 242, we reduce the previous inequality to 
\begin{eqnarray*}
\left\vert S_{3,1,1} \right\vert & \leq & \left\Vert W \right\Vert_{\mathbb{L}^{2}(D)} \; \left\Vert  \div N\left(u^{\star}_{1}\right) \right\Vert_{\mathbb{L}^{2}(D)}.
\end{eqnarray*}
Just as divergence operator $ \div $ and gradient operator $\nabla$ are both differential operator of order one, we use $(\ref{NormgradN})$ to finish the estimation of $S_{3,1,1}$. We have,  
\begin{equation}\label{AILFBS2}
S_{3,1,1}  = \mathcal{O}\left( a \, \left\Vert u^{\star}_{1} \right\Vert_{\mathbb{L}^{2}(D)} \,  \, \left\Vert W \right\Vert_{\mathbb{L}^{2}(D)} \right).
\end{equation}
Also, we have, 
\begin{equation*}
\left\vert S_{3,1,2} \right\vert := \left\vert \int_{D} W(x) \cdot \nabla \, \nabla SL\left(\nu \cdot N \left( u^{\star}_{1} \right) \right)(x) \, dx \right\vert  \leq  \left\Vert  W \right\Vert_{\mathbb{L}^{2}(D)} \, \left\Vert  \nabla \, \nabla SL\left(\nu \cdot N \left( u^{\star}_{1} \right) \right) \right\Vert_{\mathbb{L}^{2}(D)},
\end{equation*}
and, after scaling, we obtain:  
\begin{equation*}
\left\vert S_{3,1,2} \right\vert  \leq  \left\Vert  W \right\Vert_{\mathbb{L}^{2}(D)} \, a^{\frac{3}{2}} \, a^{2} \,\left\Vert  \nabla \, \nabla SL\left(\nu \cdot N \left( \widetilde{u^{\star}_{1}} \right) \right) \right\Vert_{\mathbb{L}^{2}(B)}.
\end{equation*}
Using the continuity of the operator $\nabla \, \nabla SL: \mathbb{H}^{\frac{1}{2}}(\partial B) \rightarrow \mathbb{L}^{2}(B)$, see for instance \cite{McLean}, corollary 6.14, page 210, we reduce the previous equation to 
\begin{equation*}
\left\vert S_{3,1,2} \right\vert  \lesssim  \left\Vert  W \right\Vert_{\mathbb{L}^{2}(D)} \, a^{\frac{3}{2}} \, a^{2} \,\left\Vert  \nu \cdot N \left( \widetilde{u^{\star}_{1}}  \right) \right\Vert_{\mathbb{H}^{1/2}(\partial B)}.
\end{equation*}
Now, from the continuity of the trace operator we deduce that 
\begin{equation*}
\left\vert S_{3,1,2} \right\vert  \lesssim  \left\Vert  W \right\Vert_{\mathbb{L}^{2}(D)} \, a^{\frac{3}{2}} \, a^{2} \,\left\Vert  N \left( \widetilde{u^{\star}_{1}}  \right) \right\Vert_{\mathbb{H}^{1}(B)}.
\end{equation*}
In addition, using the continuity of the Newtonian operator and scaling back to end up with 
\begin{equation}\label{SNQ}
 S_{3,1,2} = \mathcal{O}\left( \left\Vert  W \right\Vert_{\mathbb{L}^{2}(D)} \, a^{2} \,\left\Vert  u^{\star}_{1}  \right\Vert_{\mathbb{L}^{2}(D)}\right).
\end{equation}
By gathering $(\ref{AILFBS2})$ with $(\ref{SNQ})$ we deduce that:
\begin{equation*}
 S_{3,1} = \mathcal{O}\left( \left\Vert  W \right\Vert_{\mathbb{L}^{2}(D)} \, a \,\left\Vert  u^{\star}_{1}  \right\Vert_{\mathbb{L}^{2}(D)}\right).
\end{equation*}
Viewing the definition of $u^{\star}_{1}(\cdot)$, see for instance $(\ref{defU1star})$, and using the smoothness of $ \dfrac{\eta(\cdot)}{\epsilon_{0}(\cdot)} \, \nabla \varepsilon (\cdot)$ we deduce that 
\begin{equation}\label{SNQS31}
 S_{3,1} = \mathcal{O}\left( \left\Vert  W \right\Vert_{\mathbb{L}^{2}(D)} \, a \,\left\Vert  u_{1}  \right\Vert_{\mathbb{L}^{2}(D)}\right).
\end{equation}
\item Estimation of:
\begin{equation*}
S_{3,2} := \int_{D} W(x) \cdot \int_{D} W_{4}(x,y) \cdot u_{1}(y) \, (\epsilon_{0}(y) - \epsilon_{p}) \, dy \, dx. 
\end{equation*} 
Then, by Holder inequality, we obtain: 
\begin{eqnarray*}
\left\vert S_{3,2} \right\vert & \leq & \left\Vert W \right\Vert_{\mathbb{L}^{2}(D)} \, \left\Vert \int_{D} W_{4}(\cdot ,y) \cdot u_{1}(y) \, (\epsilon_{0}(y) - \epsilon_{p}) \, dy \right\Vert_{\mathbb{L}^{2}(D)} \\
& \leq & \left\Vert W \right\Vert_{\mathbb{L}^{2}(D)} \, \left\Vert  u_{1} \right\Vert_{\mathbb{L}^{2}(D)} \, \left[\int_{D} \, \int_{D} \left\vert W_{4}(x ,y)  \, (\epsilon_{0}(y) - \epsilon_{p}) \right\vert^{2} dy \, dx \right]^{\frac{1}{2}} \\
& \leq & \left\Vert W \right\Vert_{\mathbb{L}^{2}(D)} \, \left\Vert  u_{1} \right\Vert_{\mathbb{L}^{2}(D)} \, \underset{D}{Sup}\left\vert \epsilon_{0}(\cdot) - \epsilon_{p} \right\vert  \, \left[\int_{D} \, \int_{D} \left\vert W_{4}(x ,y)  \,  \right\vert^{2} dy \, dx \right]^{\frac{1}{2}}.
\end{eqnarray*}
We know that $\left( \epsilon_{0}(\cdot)-\epsilon_{p} \right) \sim 1$, with respect to the size $a$, we deduce that 
\begin{equation}\label{MLePen}
\left\vert S_{3,2} \right\vert  \lesssim  \left\Vert W \right\Vert_{\mathbb{L}^{2}(D)} \, \left\Vert  u_{1} \right\Vert_{\mathbb{L}^{2}(D)}  \, \left[\int_{D}  \int_{D} \left\vert W_{4}(x ,y) \right\vert^{2} dy  dx \right]^{\frac{1}{2}}.
\end{equation}
Now, we recall from $(\ref{AY})$ that $W_{4}(\cdot,y) \in \mathbb{L}^{\frac{3(3-2\delta)}{(3+2\delta)}}(D)$ and, then, we approximate it's singularity as 
\begin{equation*}
W_{4}(\cdot,y) \simeq \left\vert \cdot - y \right\vert^{-\frac{(3+2\delta)(3-\delta)}{3(3-2\delta)}}
\end{equation*}
and, then, by scaling the double integral in $(\ref{MLePen})$ we get 
\begin{equation}\label{MLePenI}
S_{3,2} = \mathcal{O}\left(  \left\Vert W \right\Vert_{\mathbb{L}^{2}(D)} \, \left\Vert  u_{1} \right\Vert_{\mathbb{L}^{2}(D)}  \, \, a^{\frac{2\delta^{2}-21\delta + 18}{3(3-\delta)}} \right).
\end{equation}
\end{enumerate}
Finally, by summing $(\ref{SNQS31})$ and $(\ref{MLePenI})$ we get: 
\begin{equation}\label{MQUSMA}
S_{3} = \mathcal{O}\left(  \left\Vert W \right\Vert_{\mathbb{L}^{2}(D)} \, \left\Vert  u_{1} \right\Vert_{\mathbb{L}^{2}(D)} \, a \right).
\end{equation}
Using $(\ref{TermS1}), (\ref{AILFB}), (\ref{MQUSMA})$ and the definition of $Err_{1}$, see $(\ref{BG})$, we get
\begin{equation*}
Err_{1} = \mathcal{O}\left( a^{\frac{5}{2}} \, \left\Vert W \right\Vert_{\mathbb{L}^{2}(D)}  + a \, \left\Vert u_{1} \right\Vert_{\mathbb{L}^{2}(D)} \; \left\Vert W \right\Vert_{\mathbb{L}^{2}(D)} \right).
\end{equation*} 
\end{enumerate}
Going back to $(\ref{intu1=intw+err2})$ and plugging the expression of $Err_{1}$ to obtain
\begin{equation}\label{Intu1=IntWu0+Err}
\int_{D} u_{1}(x) dx = \int_{D} W(x) \, dx \cdot  u_{0}(z) + \mathcal{O}\left( a^{\frac{5}{2}} \, \left\Vert W \right\Vert_{\mathbb{L}^{2}(D)}  + a \, \left\Vert u_{1} \right\Vert_{\mathbb{L}^{2}(D)} \; \left\Vert W \right\Vert_{\mathbb{L}^{2}(D)} \right).
\end{equation}
The next proposition show the estimates of the terms appearing in the right hand-side.  
\begin{proposition}\label{Proposition2.1}
We  have:
\begin{equation}\label{aprioriestimateregimemoderate}
\left\Vert u_{1} \right\Vert_{\mathbb{L}^{2}(D)} \leq a^{-h} \;\left\Vert u_{0} \right\Vert_{\mathbb{L}^{2}(D)}, \; \; \left\Vert W \right\Vert_{\mathbb{L}^{2}(D)} = \mathcal{O}\left( a^{\frac{3}{2}-h} \right)
\end{equation}
and
\begin{equation*}\label{IntWD2.32}
\int_{D} W(x) \, dx = \frac{a^{3} \, \epsilon_{0}(z)}{\left( \epsilon_{0}(z) - \left(\epsilon_{0}(z) - \epsilon_{p} \right) \,   \lambda^{(3)}_{n_{0}} \right)} \, \int_{B} e^{(3)}_{n_{0}}(x) \, dx \, \otimes \, \int_{B} e^{(3)}_{n_{0}}(x) \, dx + \mathcal{O}\left( a^{3} \right).
\end{equation*}
\end{proposition}
\begin{proof}
See Subsection \ref{Section4} and Subsection \ref{ADZ}. 
\end{proof}
Thanks to Proposition $\ref{Proposition2.1}$, the equation $(\ref{Intu1=IntWu0+Err})$ takes the following form: 
\begin{equation*}\label{BRTVMM}
 \int_{D} u_{1}(x) dx = \frac{a^{3} \; \epsilon_{0}(z) \, \langle u_{0}(z) ; \int_{B} e^{(3)}_{n_{0}}(x) \, dx \rangle}{\left(\epsilon_{0}(z) -  \left(\epsilon_{0}(z) - \epsilon_{p} \right) \,   \lambda^{(3)}_{n_{0}} \right)} \, \int_{B} e^{(3)}_{n_{0}}(x) \, dx  + \mathcal{O}\left(a^{\min(3, 4-2h)} \right).
\end{equation*}
Consequently, 
\begin{equation*}\label{ZAKD}
\left\vert \int_{D} u_{1}(x) dx \right\vert^{2} = \frac{a^{6} \; \left\vert \epsilon_{0}(z) \right\vert^{2} \,\left\vert \langle u_{0}(z) ; \int_{B} e^{(3)}_{n_{0}}(x) \, dx \rangle \right\vert^{2}}{\left\vert \epsilon_{0}(z) - \left(\epsilon_{0}(z) - \epsilon_{p} \right) \,   \lambda^{(3)}_{n_{0}} \right\vert^{2}} \, \left\vert \int_{B} e^{(3)}_{n_{0}}(x) \, dx \right\vert^{2} + \mathcal{O}\left(a^{\min(6-h, 7-3h)} \right).
\end{equation*}
\newline
In the sequel, we derive a relation between the given data  $\left\Vert u_{1} \right\Vert_{\mathbb{L}^{2}(D)}$ and $\left\vert \int_{D} u_{1}(x) \, dx \right\vert$. For this, we recall from  $(\ref{sppo})$ that, in the domain $D$, we have
\begin{equation*}
u_{1} =  \left[ I - \eta(z)  \,   \nabla M \right]^{-1}\left[ u_{0} + Err_{0} + Err_{\Gamma} \right], 
\end{equation*}
and after scaling to the domain $B$ we obtain 
\begin{equation*}
\tilde{u}_{1} =  \left[ I - \eta(z) \,   \nabla M \right]^{-1}\left[ \tilde{u}_{0}  + \widetilde{Err}_{0} + \widetilde{Err}_{\Gamma} \right]. 
\end{equation*}
Recalling the decomposition of $\mathbb{L}^{2}(B)$, see $(\ref{L2-decomposition})$,
\begin{equation*}
\mathbb{L}^{2}(B)  = \mathbb{H}_{0}\left(\div=0 \right) \overset{\perp}{\oplus} \mathbb{H}_{0}\left(Curl=0 \right) \overset{\perp}{\oplus} \nabla \mathcal{H}armonic,
\end{equation*}
we project the previous equation into three subspaces. 
\begin{enumerate}
\item Taking the inner product with respect to $e^{(1)}_{n}(\cdot)$: 
\begin{equation}\label{HAS}
\langle \tilde{u}_{1}; e^{(1)}_{n} \rangle = \langle e^{(1)}_{n}; \left[ I - \eta(z) \,\nabla M \right]^{-1}\left[ \tilde{u}_{0} + \widetilde{Err}_{0} + \widetilde{Err}_{\Gamma} \right] \rangle, 
\end{equation}
using the self-adjointness of $\nabla M (\cdot)$ and the fact that  $\nabla M \left( e^{(1)}_{n} \right) = 0$, see  Lemma \ref{MagnetizationOperator}, we deduce that $\left[ I - \eta(z) \,\nabla M \right]^{-1} \left( e^{(1)}_{n} \right) = e^{(1)}_{n}
$
and we reduce the equation $(\ref{HAS})$ to: 
\begin{equation*}
\langle \tilde{u}_{1}; e^{(1)}_{n} \rangle = \langle e^{(1)}_{n}; \left[ \tilde{u}_{0} + \widetilde{Err}_{0} + \widetilde{Err}_{\Gamma} \right] \rangle = \langle e^{(1)}_{n};  \tilde{u}_{0} \rangle + Err_{3,n}, 
\end{equation*}
where $Err_{3,n}$ is the term given by: 
\begin{eqnarray*}
Err_{3,n} &:=& \langle e^{(1)}_{n}; \widetilde{Err}_{0}  \rangle + \langle e^{(1)}_{n};  \widetilde{Err}_{\Gamma} \rangle \\
& \overset{(\ref{WN})}{=} & - \omega^{2} \, \mu \, a  \langle e^{(1)}_{n}; \nabla M \left(\tilde{u}_{1}(\cdot) \int_{0}^{1} \widetilde{\nabla \eta}(z+ta \cdot)  \cdot (\cdot) dt \right) \rangle + \langle e^{(1)}_{n};  \widetilde{Err}_{\Gamma} \rangle.
\end{eqnarray*}
We use the fact that $\nabla M \left(e^{(1)}_{n} \right) = 0$ to reduce the previous expression to $Err_{3,n} =  \langle e^{(1)}_{n}; \widetilde{Err}_{\Gamma} \rangle.$
Hence, 
\begin{equation*}
\langle e^{(1)}_{n}; \tilde{u}_{1} \rangle = \langle e^{(1)}_{n}; \tilde{u}_{0} \rangle + \langle e^{(1)}_{n}; \widetilde{Err}_{\Gamma} \rangle.
\end{equation*}
By taking the square modulus of the previous equality and then the series with respect to $n$, we obtain: 
\begin{eqnarray}\label{11MM1}
\nonumber
\sum_{n} \, \left\vert \langle \tilde{u}_{1}; e^{(1)}_{n} \rangle \right\vert^{2} & = & \sum_{n} \left\vert \langle \tilde{u}_{0}; e^{(1)}_{n} \rangle \right\vert^{2} + \sum_{n} \left\vert \langle \widetilde{Err}_{\Gamma}; e^{(1)}_{n} \rangle \right\vert^{2} \\ \nonumber &+&  \mathcal{O}\left( \left(\sum_{n} \left\vert \langle \tilde{u}_{0}; e^{(1)}_{n} \rangle \right\vert^{2} \right)^{\frac{1}{2}} \left(\sum_{n} \left\vert \langle \widetilde{Err}_{\Gamma}; e^{(1)}_{n} \rangle \right\vert^{2} \right)^{\frac{1}{2}} \right),  
\end{eqnarray}
and, using $(\ref{Norm-P12U0}), (\ref{EstimationErrGamma1})$ and $(\ref{aprioriestimateregimemoderate})$, we get: 
\begin{equation}\label{11MM1}
\sum_{n} \, \left\vert \langle \tilde{u}_{1}; e^{(1)}_{n} \rangle \right\vert^{2}  =  \sum_{n} \left\vert \langle \tilde{u}_{0}; e^{(1)}_{n} \rangle \right\vert^{2} + \mathcal{O}\left( a^{\frac{(9 - 7 \delta -2 \delta^{2})}{(3 - 2 \delta)} - h} \right) = \mathcal{O}\left( a^{\min\left( 2 ;\frac{(9 - 7 \delta -2 \delta^{2})}{(3 - 2 \delta)} - h \right)} \right).  
\end{equation}
\item Taking the inner product with respect to $e^{(2)}_{n}(\cdot)$.
\begin{equation*}
\langle \tilde{u}_{1}; e^{(2)}_{n} \rangle = \langle e^{(2)}_{n}; \left[ I - \eta(z) \,   \nabla M \right]^{-1}\left[ \tilde{u}_{0} + \widetilde{Err}_{0} + \widetilde{Err}_{\Gamma} \right] \rangle, 
\end{equation*}
since $\nabla M \left( e^{(2)}_{n} \right) = e^{(2)}_{n}$, see Lemma $(\ref{MagnetizationOperator})$, then after taking the adjoint operator of $\left[ I - \eta(z)   \,   \nabla M \right]$ the previous equation will be reduced to 
\begin{eqnarray}\label{equa2.40LZ}
\langle \tilde{u}_{1}; e^{(2)}_{n} \rangle = \frac{\overline{\epsilon_{0}(z)} \, \left[ \langle e^{(2)}_{n};  \tilde{u}_{0} \rangle + \langle e^{(2)}_{n}; \widetilde{Err}_{0}\rangle + \langle e^{(2)}_{n}; \widetilde{Err}_{\Gamma}\rangle \right]}{\overline{\epsilon}_{p}} = \frac{\overline{\epsilon_{0}(z)} \,  \langle e^{(2)}_{n};  \tilde{u}_{0} \rangle}{ \overline{\epsilon}_{p}} + Err_{4,n},
\end{eqnarray}
where, obviously, the term $Err_{4,n}$ is given by
\begin{equation*}
Err_{4,n} := \frac{\overline{\epsilon_{0}(z)} \, \left[ \langle e^{(2)}_{n}; \widetilde{Err}_{0}\rangle + \langle e^{(2)}_{n}; \widetilde{Err}_{\Gamma}\rangle \right]}{\overline{\epsilon}_{p} }. 
\end{equation*}
Now, as $\dfrac{\epsilon_{0}(z)}{\epsilon_{p}} \sim 1$, with respect to the size $a$, we approximate $Err_{4,n}$ by: 
\begin{equation*}
Err_{4,n} \simeq \langle e^{(2)}_{n}; \widetilde{Err}_{0}\rangle + \langle e^{(2)}_{n}; \widetilde{Err}_{\Gamma}\rangle. 
\end{equation*}
Using the definition of $Err_{0}$, see for instance $(\ref{WN})$, and the fact that $\nabla M \left( e^{(2)}_{n} \right) = e^{(2)}_{n}$ to get: 
\begin{equation*}
Err_{4,n} \simeq a\, \langle e^{(2)}_{n}; \tilde{u}_{1}(\cdot) \, \int_{0}^{1} \widetilde{\nabla \, \eta}(z+t a \, \cdot) \cdot (\cdot) \, dt \rangle + \langle e^{(2)}_{n}; \widetilde{Err}_{\Gamma}\rangle. 
\end{equation*}
Consequently, 
\begin{equation*}
\sum_{n} \left\vert Err_{4,n} \right\vert^{2}  \lesssim  a^{2} \; \left\Vert \tilde{u}_{1} \right\Vert^{2}_{\mathbb{L}^{2}(B)}  + \sum_{n} \left\vert \langle e^{(2)}_{n}; \widetilde{Err}_{\Gamma} \rangle \right\vert^{2} \overset{(\ref{RDSN})}{=} \mathcal{O}\left( a^{2} \; \left\Vert \tilde{u}_{1} \right\Vert^{2}_{\mathbb{L}^{2}(B)} \right),
\end{equation*}
and, using the a priori estimation $(\ref{aprioriestimateregimemoderate})$, we obtain: 
\begin{equation}\label{RDSNKAS}
\sum_{n} \left\vert Err_{4,n} \right\vert^{2} =  \mathcal{O}\left(a^{2-2h}\right).
\end{equation}
Hence, using $(\ref{RDSNKAS})$ in $(\ref{equa2.40LZ})$ we obtain:  
\begin{equation}\label{MM2}
\sum_{n} \, \left\vert \langle \tilde{u}_{1}; e^{(2)}_{n} \rangle \right\vert^{2} = \frac{\left\vert \epsilon_{0}(z)  \right\vert^{2}}{\left\vert  \epsilon_{p}  \right\vert^{2}} \underset{n}{\sum} \left\vert \langle \tilde{u}_{0}; e^{(2)}_{n} \rangle \right\vert^{2}  + \mathcal{O}\left( a^{2-2h} \right) \overset{(\ref{Norm-P12U0})}{=} \mathcal{O}\left( a^{2-2h} \right).
\end{equation}

\item Taking the inner product with respect to $e^{(3)}_{n}(\cdot)$, we get:
\begin{equation*}
\langle \tilde{u}_{1}; e^{(3)}_{n} \rangle = \langle e^{(3)}_{n}; \left[ I - \eta(z) \nabla M \right]^{-1}\left[ \tilde{u}_{0} + \widetilde{Err}_{0} + \widetilde{Err}_{\Gamma} \right] \rangle. 
\end{equation*}
Since $\nabla M \left( e^{(3)}_{n} \right) = \lambda^{(3)}_{n} \, e^{(3)}_{n}$, see $(\ref{b3})$, and the definition of the function $\eta(\cdot)$, see $(\ref{Defeta})$, the previous equation will be reduced to: 
\begin{eqnarray}\label{equa2.43AS}
\langle \tilde{u}_{1}; e^{(3)}_{n} \rangle &=& \frac{\overline{\epsilon_{0}(z)} \, \langle e^{(3)}_{n}; \tilde{u}_{0} \rangle}{\left( \overline{\epsilon_{0}(z)} - \left( \overline{ \epsilon_{0}(z)} - \overline{\epsilon}_{p} \right) \, \lambda^{(3)}_{n} \right)} + Err_{5,n}, 
\end{eqnarray}
where $Err_{5,n}$ is the term given by
\begin{eqnarray*}
Err_{5,n} &:=& \frac{\overline{\epsilon_{0}(z)}}{\left( \overline{\epsilon_{0}(z)} - \left( \overline{ \epsilon_{0}(z)} - \overline{\epsilon}_{p} \right) \, \lambda^{(3)}_{n} \right)} \left[ \langle e^{(3)}_{n}; \widetilde{Err}_{0} \rangle + \langle e^{(3)}_{n}; \widetilde{Err}_{\Gamma} \rangle \right] \\
Err_{5,n} & \overset{(\ref{WN})}{\simeq} & \frac{\langle e^{(3)}_{n}; \widetilde{Err}_{\Gamma} \rangle + a \,  \langle e^{(3)}_{n};   
\tilde{u}_{1}(\cdot) \, \int_{0}^{1} \widetilde{\nabla \,\eta}(z+ta\cdot)) \cdot (\cdot) \, dt  \, dy  \rangle}{\left( \overline{\epsilon_{0}(z)} - \left( \overline{ \epsilon_{0}(z)} - \overline{\epsilon}_{p} \right) \, \lambda^{(3)}_{n} \right)}.
\end{eqnarray*}
Taking the squared modulus and the series with respect to $n$, we obtain  
\begin{equation*}
\sum_{n} \left\vert Err_{5,n} \right\vert^{2}  \lesssim  \sum_{n} \frac{\left\vert \langle e^{(3)}_{n}; \widetilde{Err}_{\Gamma} \rangle \right\vert^{2}+a^{2} \, \left\vert \langle e^{(3)}_{n};  \tilde{u}_{1}(\cdot) \, \int_{0}^{1} \widetilde{\nabla \eta}(z+a t \cdot ) \cdot (\cdot) dt  \rangle \right\vert^{2}}{\left\vert  \epsilon_{0}(z) - \left(  \epsilon_{0}(z) - \epsilon_{p} \right) \, \lambda^{(3)}_{n}  \right\vert^{2}}.
\end{equation*}
Recall, from $(\ref{approximationoflambdan0})$, that we have
\begin{equation}\label{cases}
\left\vert  \epsilon_{0}(z) - \left(  \epsilon_{0}(z) - \epsilon_{p} \right) \, \lambda^{(3)}_{n}  \right\vert \sim \begin{cases}
a^{h} & \text{if} \quad  n=n_{0} \\
1 & \text{if}  \quad n \neq n_{0}
\end{cases}.
\end{equation} 
Then 
\begin{equation*}
\sum_{n} \left\vert Err_{5,n} \right\vert^{2}  \lesssim   a^{-2h} \sum_{n} \left\vert \langle e^{(3)}_{n}; \widetilde{Err}_{\Gamma} \rangle \right\vert^{2}  + a^{2-2h} \; \left\Vert  \tilde{u}_{1} \right\Vert^{2}_{\mathbb{L}^{2}(B)}.
\end{equation*}
Using the relation $(\ref{TMCAA})$ and the a priori estimate given by $(\ref{aprioriestimateregimemoderate})$, we get:
\begin{equation*}\label{OF2021}
\sum_{n} \left\vert Err_{5,n} \right\vert^{2} = \mathcal{O}\left(a^{2-4h} \right).
\end{equation*}
Now, taking the squared modulus and the series with respect to $n$ in $(\ref{equa2.43AS})$, we obtain  
\begin{equation*}
\sum_{n} \left\vert \langle e^{(3)}_{n}; \tilde{u}_{1} \rangle \right\vert^{2}  =  \sum_{n} \frac{\left\vert  \epsilon_{0}(z) \right\vert^{2} \, \left\vert \langle e^{(3)}_{n}; \tilde{u}_{0} \rangle \right\vert^{2}}{\left\vert  \epsilon_{0}(z) - \left(  \epsilon_{0}(z) - \epsilon_{p} \right) \, \lambda^{(3)}_{n}  \right\vert^{2}} + \mathcal{O}\left(a^{1-3h} \right).
\end{equation*}
The relation $(\ref{cases})$ suggests us to split the series appearing on the right hand side as follows: 
\begin{equation*}
\sum_{n} \left\vert \langle e^{(3)}_{n}; \tilde{u}_{1} \rangle \right\vert^{2} =  \frac{\left\vert  \epsilon_{0}(z) \right\vert^{2} \, \left\vert \langle e^{(3)}_{n_{0}}; \tilde{u}_{0} \rangle \right\vert^{2}}{\left\vert  \epsilon_{0}(z) - \left(  \epsilon_{0}(z) - \epsilon_{p} \right) \, \lambda^{(3)}_{n_{0}}  \right\vert^{2}} + \sum_{n \neq n_{0}} \frac{\left\vert  \epsilon_{0}(z) \right\vert^{2} \, \left\vert \langle e^{(3)}_{n}; \tilde{u}_{0} \rangle \right\vert^{2}}{\left\vert  \epsilon_{0}(z) - \left(  \epsilon_{0}(z) - \epsilon_{p} \right) \, \lambda^{(3)}_{n}  \right\vert^{2}} + \mathcal{O}\left( a^{1-3h}  \right),
\end{equation*}
and, obviously, we have  
\begin{equation*}
\sum_{n \neq n_{0}} \frac{\left\vert  \epsilon_{0}(z) \right\vert^{2} \, \left\vert \langle e^{(3)}_{n}; \tilde{u}_{0} \rangle \right\vert^{2}}{\left\vert  \epsilon_{0}(z) - \left(  \epsilon_{0}(z) - \epsilon_{p} \right) \, \lambda^{(3)}_{n}  \right\vert^{2}} \sim 1. 
\end{equation*}
Then, 
\begin{equation*}
\sum_{n} \left\vert \langle e^{(3)}_{n}; \tilde{u}_{1} \rangle \right\vert^{2} = \frac{\left\vert  \epsilon_{0}(z) \right\vert^{2} \, \left\vert \langle e^{(3)}_{n_{0}}; \tilde{u}_{0} \rangle \right\vert^{2}}{\left\vert  \epsilon_{0}(z) - \left(  \epsilon_{0}(z) - \epsilon_{p} \right) \, \lambda^{(3)}_{n_{0}}  \right\vert^{2}} + \mathcal{O}\left( a^{\min(1-3h,0)}  \right).
\end{equation*}
By Taylor expansion we can prove that 
\begin{equation*}
\langle e^{(3)}_{n_{0}}; \tilde{u}_{0} \rangle = \int_{B} e^{(3)}_{n_{0}}(x) \, dx \cdot u_{0}(z) + \mathcal{O}\left( a \right), 
\end{equation*}
and, consequently, we get 
\begin{equation}\label{MM3}
\sum_{n} \left\vert \langle e^{(3)}_{n}; \tilde{u}_{1} \rangle \right\vert^{2} = \frac{\left\vert  \epsilon_{0}(z) \right\vert^{2} \, \left\vert u_{0}(z) \cdot \int_{B} e^{(3)}_{n_{0}}(x)dx \right\vert^{2}}{\left\vert  \epsilon_{0}(z) - \left(  \epsilon_{0}(z) - \epsilon_{p} \right) \, \lambda^{(3)}_{n_{0}}  \right\vert^{2}} +  \mathcal{O}\left( a^{\min(1-3h, 0)} \right). 
\end{equation}
\end{enumerate}
At present, by combining $(\ref{11MM1}), (\ref{MM2})$ and $(\ref{MM3})$, we obtain an estimation of $\left\Vert \tilde{u}_{1} \right\Vert^{2}_{\mathbb{L}^{2}(B)}$. More precisely,  
\begin{eqnarray*}
\left\Vert \tilde{u}_{1} \right\Vert^{2}_{\mathbb{L}^{2}(B)} &:=& \sum_{n} \left\vert \langle e^{(1)}_{n}; \tilde{u}_{1} \rangle \right\vert^{2} + \sum_{n} \left\vert \langle e^{(2)}_{n}; \tilde{u}_{1} \rangle \right\vert^{2} + \sum_{n} \left\vert \langle e^{(3)}_{n}; \tilde{u}_{1} \rangle \right\vert^{2} \\
 &=&\frac{\left\vert  \epsilon_{0}(z) \right\vert^{2} \, \left\vert u_{0}(z) \cdot \int_{B} e^{(3)}_{n_{0}}(x)dx \right\vert^{2}}{\left\vert  \epsilon_{0}(z) - \left(  \epsilon_{0}(z) - \epsilon_{p} \right) \, \lambda^{(3)}_{n_{0}}  \right\vert^{2}} + \mathcal{O}\left(a^{\min(0,1-3h)} \right),
\end{eqnarray*}
or, after scaling back, 
\begin{equation}\label{KMYHAYD}
\int_{D} \left\vert u_{1} \right\vert^{2}(x) \, dx = a^{3} \; \frac{\left\vert  \epsilon_{0}(z) \right\vert^{2} \, \left\vert u_{0}(z) \cdot \int_{B} e^{(3)}_{n_{0}}(x)dx \right\vert^{2}}{\left\vert  \epsilon_{0}(z) - \left(  \epsilon_{0}(z) - \epsilon_{p} \right) \, \lambda^{(3)}_{n_{0}}  \right\vert^{2}} + \mathcal{O}\left(a^{\min(3,4-3h)} \right).
\end{equation}
This proves $(\ref{THM-KMYHAYD})$. 

\subsection{Photo-acoustic model in the presence of one particle and proof of (\ref{PrincipalFormula})}
Let us start by recalling the model problem of the photo-acoustic imaging:
\begin{equation}\label{pa2nd}
\begin{cases} 
\partial_{t}^{\prime\prime} p(x,t) - \underset{x}{\Delta} \, p(x,t) \, = 0, \quad \text{in} \quad \mathbb{R}^{3} \times \mathbb{R}^{+} \\ 
\qquad \qquad \qquad p(x,0) = \Im\left( \varepsilon \right)(x) \, \left\Vert E \right\Vert^{2}(x) \, \chi_{\Omega} \quad \text{in} \quad \mathbb{R}^{3} \\
\qquad \qquad \, \quad \partial_{t} p(x,0) = 0 \quad \text{in} \quad \mathbb{R}^{3}.
\end{cases}
\end{equation}
From \cite{Neil}, Corollary 4.1, page 180, we know that the solution of $(\ref{pa2nd})$, can be represented as:
\begin{equation*}
p(x,t) = \frac{1}{4 \, \pi} \frac{\partial}{\partial t} \left[ \frac{1}{t} \; \int_{\partial B(x,t)} \Im\left( \varepsilon \right)(y) \, \left\Vert E \right\Vert^{2}(y) \, \chi_{\Omega}(y)  \; d\sigma(y) \, \right], \quad x \in \partial \Omega, \, t >0,
\end{equation*}
where $B(x,t)$ is the ball of center $x$ and radius $t$. Remark that  
\begin{eqnarray*}
\bm{J} := \int_{\partial B(x,t)} \Im\left( \varepsilon \right)(y) \, \left\Vert E \right\Vert^{2}(y) \, \chi_{\Omega}(y)  \; d\sigma(y) & \overset{(\ref{diff-int})}{=} & \partial_{t} \int_{B(x,t)} \Im\left( \varepsilon \right)(y) \, \left\Vert E \right\Vert^{2}(y) \, \chi_{\Omega}(y) \; dy \\
& = & \partial_{t} \int_{B(x,t) \cap \Omega} \Im\left( \varepsilon \right)(y) \, \left\Vert E \right\Vert^{2}(y) \; dy 
\end{eqnarray*}
and for $t > diam(\Omega)$, we know that $B(x,t) \cap \Omega = \Omega$ and this implies $\bm{J} = 0$ which translates the Huygens principle. To fix notations, in the presence of one particle, we set $E := u_{1}$ and we rewrite the previous representation of $p(\cdot, \cdot)$ as 
\begin{equation*}
p(x,t) = \frac{1}{4 \, \pi} \frac{\partial}{\partial t} \left[ \frac{1}{t} \; \int_{\partial B(x,t)} \Im\left( \varepsilon \right)(y) \, \left\vert u_{1} \right\vert^{2}(y) \, \chi_{\Omega}(y)  \; d\sigma(y) \, \right].
\end{equation*}
Now, taking the integral from $0$ to $r$, $r \leq diam(\Omega) $, in both sides of the previous equation to get  
\begin{eqnarray*}
\int_{0}^{r} p(x,t) \, dt &=& \int_{0}^{r} \frac{1}{4 \, \pi} \frac{\partial}{\partial t} \left[ \frac{1}{t} \; \int_{\partial B(x,t)} \Im\left( \varepsilon \right)(y) \, \left\vert u_{1} \right\vert^{2}(y) \, \chi_{\Omega}(y)  \; d\sigma(y) \, \right] \, dt \\
\int_{0}^{r} p(x,t) \, dt &=& \frac{1}{4 \, \pi \, r} \; \int_{\partial B(x,r)} \Im\left( \varepsilon \right)(y) \, \left\vert u_{1} \right\vert^{2}(y) \, \chi_{\Omega}(y)  \; d\sigma(y),
\end{eqnarray*}
or
\begin{equation}\label{OW}
r \; \int_{0}^{r} p(x,t) \, dt = \frac{1}{4 \, \pi} \; \int_{\partial B(x,r)} \Im\left( \varepsilon \right)(y) \, \left\vert u_{1} \right\vert^{2}(y) \, \chi_{\Omega}(y)  \; d\sigma(y).
\end{equation}
On the right hand side, for technical reasons, we need to make a volume integral appear instead of surface one. For this, having in mind that on the ball centered at the origin of radius $\rho$ we have\footnote{In differential form we obtain: 
\begin{equation}\label{diff-int}
\partial_{\rho} \, \int_{B(0,\rho)} (\cdots) \, d \mu = \int_{\partial B(0,\rho)} (\cdots) \, d\sigma, \quad \rho >0.
\end{equation}
}: 
\begin{equation*}
\int_{B(0,\rho)} (\cdots) \, d \mu = \int_{0}^{\rho} \int_{\partial B(0,s)} (\cdots) \, d\sigma \, ds, \quad \rho >0, \; s>0,
\end{equation*}
we integrate $(\ref{OW})$ to obtain  
\begin{eqnarray*}
\int_{0}^{s} r \; \int_{0}^{r} p(x,t) \, dt \, dr &=& \frac{1}{4 \, \pi} \int_{0}^{s} \; \int_{\partial B(x,r)} \Im\left( \varepsilon \right)(y) \, \left\vert u_{1} \right\vert^{2}(y) \, \chi_{\Omega}(y)  \; d\sigma(y) \, dr \\
&=& \frac{1}{4 \, \pi}  \int_{B(x,s)} \Im\left( \varepsilon \right)(y) \, \left\vert u_{1} \right\vert^{2}(y) \, \chi_{\Omega}(y)  \; dy. 
\end{eqnarray*}
For shortness, we set 
\begin{equation*}
p^{\star}(x,s) := \int_{0}^{s} r \; \int_{0}^{r} p(x,t) \, dt \, dr
\end{equation*}
and, then, we end up with the following equation 
\begin{equation*}
p^{\star}(x,s) = \frac{1}{4 \, \pi}  \int_{B(x,s)} \Im\left( \varepsilon \right)(y) \, \left\vert u_{1} \right\vert^{2}(y) \, \chi_{\Omega}(y)  \; dy.
\end{equation*}
Next, we split the domain of integration into two parts as follows 
\begin{equation}\label{Finite-speed-propagation}
p^{\star}(x,s) = \frac{1}{4 \, \pi}  \int_{B(x,s) \cap D} \Im\left( \varepsilon \right)(y) \, \left\vert u_{1} \right\vert^{2}(y)  \; dy + \frac{1}{4 \, \pi}  \int_{B(x,s) \cap \left(\Omega \setminus D \right)} \Im\left( \varepsilon \right)(y) \, \left\vert u_{1} \right\vert^{2}(y)   \; dy,
\end{equation}
and for
\begin{equation}\label{diam+dist}
diam(D) + dist(x,D) \leq s,
\end{equation}
we reduce the previous equation to 
\begin{equation*}
p^{\star}(x,s) = \frac{1}{4 \, \pi}  \int_{D} \Im\left( \varepsilon \right)(y) \, \left\vert u_{1} \right\vert^{2}(y)  \; dy + \frac{1}{4 \, \pi}  \int_{B(x,s) \cap \left(\Omega \setminus D \right)} \Im\left( \varepsilon \right)(y) \, \left\vert u_{1} \right\vert^{2}(y)   \; dy,
\end{equation*}
and regarding the definition of the permittivity function $\varepsilon(\cdot)$, see $(\ref{DefPermittivityfct})$, we obtain 
\begin{eqnarray}\label{pressure-equation}
\nonumber
p^{\star}(x,s) &=& \frac{1}{4 \, \pi} \, \Im\left( \epsilon_{p} \right) \, \int_{D} \left\vert u_{1} \right\vert^{2}(y)  \; dy + \frac{1}{4 \, \pi}  \int_{B(x,s) \cap \left(\Omega \setminus D \right)} \Im\left( \epsilon_{0} \right)(y) \, \left\vert u_{1} \right\vert^{2}(y)   \; dy \\ \nonumber 
p^{\star}(x,s) &\overset{(\ref{KMYHAYD})}{=}& \frac{a^{3}}{4 \, \pi} \, \Im\left( \epsilon_{p} \right) \,  \frac{\left\vert \epsilon_{0}(z) \right\vert^{2} \, \left\vert \langle u_{0}(z) ; \int_{B} e^{(3)}_{n_{0}}(x) \, dx \rangle \right\vert^{2}}{\left\vert \epsilon_{0}(z) - \left(\epsilon_{0}(z) - \epsilon_{p} \right) \,   \lambda^{(3)}_{n_{0}} \right\vert^{2}}     + \frac{1}{4 \, \pi}  \underset{B(x,s) \cap \left(\Omega \setminus D \right)}{\int} \Im\left( \epsilon_{0} \right)(y) \, \left\vert u_{1} \right\vert^{2}(y)   \; dy \\
&+& \mathcal{O}\left(a^{\min(3,4-3h)}\right).
\end{eqnarray}
In the sequel we analyze the term 
\begin{equation}\label{def T(x,s)}
\bm{T}(x,s) := \frac{1}{4 \, \pi}  \underset{B(x,s) \cap \left(\Omega \setminus D \right)}{\int} \Im\left( \epsilon_{0} \right)(y) \, \left\vert u_{1} \right\vert^{2}(y)   \; dy.
\end{equation}
We derive the following approximation
\begin{lemma}\label{Anal-T}
We have, 
\begin{equation*}
\bm{T}(x,s)= \frac{1}{4 \, \pi}  \underset{B(x,s) \cap \Omega}{\int} \Im\left( \epsilon_{0} \right)(y) \, \left\vert u_{0} \right\vert^{2}(y)   \; dy+ \mathcal{O}\left(  a^{3-h}  \right).
\end{equation*}
\end{lemma}
We set $p^{\star}_{0}(\cdot, \cdot)$ to be 
\begin{equation}\label{defpre0}
p^{\star}_{0}(x,s) :=  \frac{1}{4 \, \pi}  \int_{B(x,s) \cap \Omega} \Im\left( \epsilon_{0} \right)(y) \, \left\vert u_{0} \right\vert^{2}(y)   \; dy,
\end{equation}
which is nothing but the average pressure in the absence of the particle, inside $\Omega$, at point $x \in \partial \Omega$.
\smallskip

\begin{proof} 
To prove Lemma \ref{Anal-T}, we recall from $(\ref{sppo})$ that $u_{1}(\cdot)$ satisfies the following equation
\begin{equation}\label{BSK}
u_{1} = u_{0} + \eta(z) \, \nabla M(u_{1}) + Err_{0} + Err_{\Gamma},
\end{equation}
where $Err_{0}$ is given by $(\ref{WN})$ and $Err_{\Gamma}$ is given by $(\ref{Err-Gamma})$. Now, we use the representation $(\ref{BSK})$, to compute the square Euclidean norm of $u_{1}(\cdot)$.
\begin{eqnarray}\label{SquareEuclideanNorm}
\nonumber
\left\vert u_{1} \right\vert^{2} = u_{1} \cdot \overline{u}_{1}  &=& \left[ u_{0} + \eta(z) \nabla M(u_{1}) + Err_{0} + Err_{\Gamma}\right]  \cdot  \left[ \overline{u}_{0} + \overline{\eta(z)}  \nabla M(\overline{u}_{1}) + \overline{Err_{0}} + \overline{Err_{\Gamma}} \right] \\ \nonumber
 &=& \left\vert u_{0} \right\vert^{2} + 2 \, \Re\left[ \eta(z) \nabla M(u_{1}) \cdot \overline{u}_{0} \right]  + \left\vert \eta(z)  \right\vert^{2} \, \left\vert \nabla M(u_{1}) \right\vert^{2} + R_{1},
\end{eqnarray}
where $R_{1}$, the remainder part, is be given by:
\begin{equation}\label{Remainder1}
R_{1} := 2 \Re\left[\left( \overline{Err_{0}} + \overline{ Err_{\Gamma}} \right) \,  u_{0}   \right] + \left\vert Err_{0} +  Err_{\Gamma} \right\vert^{2} + 2 \,  \Re\left[\left( \overline{Err_{0}} + \overline{ Err_{\Gamma}} \right) \, \eta(z) \, \nabla M(u_{1})  \right].
\end{equation}
Also, set: 
\begin{equation}\label{Remainder2}
R_{2}(x,s) :=  \frac{1}{4 \, \pi}  \int_{B(x,s) \cap \left( \Omega \setminus D \right)} \Im\left( \epsilon_{0} \right)(y) \, R_{1}(y)   \; dy.
\end{equation}
Going back to the definition of $\bm{T}(\cdot,\cdot)$, see $(\ref{def T(x,s)})$, and using the representation $(\ref{SquareEuclideanNorm})$, for $\left\vert u_{1} \right\vert^{2}$, and the definition of $R_{2}(\cdot,\cdot)$, see for instance $(\ref{Remainder2})$, we rewrite $\bm{T}(\cdot,\cdot)$ as: 
\begin{eqnarray}\label{bfT=p0+R2+T2+...+T6}
\nonumber
\bm{T}(x,s) &=& \frac{1}{4 \, \pi}  \underset{B(x,s) \cap \left(\Omega \setminus D \right)}{\int} \Im\left( \epsilon_{0} \right)(y) \,  \left\vert u_{0} \right\vert^{2}(y) dy + R_{2}(x,s) \\ \nonumber
&+&  \Re\left[ \frac{\eta(z)}{2 \, \pi }  \underset{B(x,s) \cap \left(\Omega \setminus D \right)}{\int} \Im\left( \epsilon_{0} \right)(y) \, \nabla M(u_{1})(y) \cdot \overline{u}_{0}(y) \; dy \; \right]\\ 
&+& \frac{\left\vert \eta(z)  \right\vert^{2}}{4 \, \pi }  \underset{B(x,s) \cap \left(\Omega \setminus D \right)}{\int} \Im\left( \epsilon_{0} \right)(y) \,   \, \left\vert \nabla M(u_{1}) \right\vert^{2}(y) \; dy. 
\end{eqnarray}
Now we estimate the terms appearing on the right hand side equation. 
\begin{enumerate}
\item Computation of
\begin{eqnarray*}
T_{1}(x,s) &:=& \frac{1}{4 \, \pi}  \underset{B(x,s) \cap \left(\Omega \setminus D \right)}{\int} \Im\left( \epsilon_{0} \right)(y) \,  \left\vert u_{0} \right\vert^{2}(y) dy \\
&=& \frac{1}{4 \, \pi}  \underset{B(x,s) \cap \Omega}{\int} \Im\left( \epsilon_{0} \right)(y) \,  \left\vert u_{0} \right\vert^{2}(y) dy -  \frac{1}{4 \, \pi}  \underset{B(x,s) \cap D}{\int} \Im\left( \epsilon_{0} \right)(y) \,  \left\vert u_{0} \right\vert^{2}(y) dy \\ 
& \overset{(\ref{defpre0})}{=} & p^{\star}_{0}(x,s) -  \frac{1}{4 \, \pi}  \underset{B(x,s) \cap D}{\int} \Im\left( \epsilon_{0} \right)(y) \,  \left\vert u_{0} \right\vert^{2}(y) dy \\ 
& \overset{(\ref{diam+dist})}{=} & p^{\star}_{0}(x,s) -  \frac{1}{4 \, \pi}  \underset{D}{\int} \Im\left( \epsilon_{0} \right)(y) \,  \left\vert u_{0} \right\vert^{2}(y) dy. 
\end{eqnarray*} 
As $\Im\left( \epsilon_{0} \right)(\cdot)$ is smooth function we derive the following estimation 
\begin{equation*}
\left\vert -  \frac{1}{4 \, \pi}  \underset{D}{\int} \Im\left( \epsilon_{0} \right)(y) \,  \left\vert u_{0} \right\vert^{2}(y) dy \right\vert \lesssim \left\Vert u_{0} \right\Vert^{2}_{\mathbb{L}^{2}(D)} = \mathcal{O}\left(a^{3}\right). 
\end{equation*} 
Then, 
\begin{equation}\label{T1=pstar}
T_{1}(x,s) = p^{\star}_{0}(x,s) + \mathcal{O}\left(a^{3}\right).
\end{equation}
\item Estimation of  
\begin{eqnarray*}
T_{2}(x,s) &:=& \Re\left[ \frac{\eta(z)}{2 \, \pi} \, \underset{B(x,s) \cap \left(\Omega \setminus D \right)}{\int} \Im\left( \epsilon_{0} \right)(y) \, \nabla M(u_{1})(y) \cdot \overline{u}_{0}(y) \; dy \; \right]\\ 
\left\vert T_{2}(x,s) \right\vert & \lesssim & \left\vert  \, \underset{B(x,s) \cap \left(\Omega \setminus D \right)}{\int} \Im\left( \epsilon_{0} \right)(y) \, \nabla M(u_{1})(y) \cdot \overline{u}_{0}(y) \; dy \; \right\vert . 
\end{eqnarray*}
Since $\Im\left( \epsilon_{0} \right)(\cdot)$ is smooth function in $\Omega$, we obtain 
\begin{eqnarray*} 
\left\vert T_{2}(x,s) \right\vert & \lesssim &   \underset{B(x,s) \cap \left(\Omega \setminus D \right)}{\int} \left\vert  \nabla M(u_{1})(y) \cdot \overline{u}_{0}(y) \;\right\vert \,  dy \\ 
 & \leq &    \left\Vert  \nabla M(u_{1}) \right\Vert_{\mathbb{L}^{2}\left( B(x,s) \cap \left(\Omega \setminus D \right)\right)} \; \left\Vert u_{0} \right\Vert_{\mathbb{L}^{2}\left( B(x,s) \cap \left(\Omega \setminus D \right)\right)}.
\end{eqnarray*}
Certainly, 
\begin{equation}\label{u0inOmegastar}
\left\Vert u_{0} \right\Vert_{\mathbb{L}^{2}\left( B(x,s) \cap \left(\Omega \setminus D \right)\right)} \sim 1.
\end{equation}
Then,  
\begin{equation}\label{ChinaVsUsa} 
\left\vert T_{2}(x,s) \right\vert  \lesssim       \left\Vert  \nabla M(u_{1}) \right\Vert_{\mathbb{L}^{2}\left( B(x,s) \cap \left(\Omega \setminus D \right)\right)}   :=   \left[ \underset{B(x,s) \cap \left(\Omega \setminus D \right)}{\int} \left\vert  \underset{\eta}{\nabla} \int_{D} \underset{\xi}{\nabla} \Phi_{0}(\eta,\xi) \cdot u_{1} (\xi) \, d\xi \right\vert^{2} \, d \eta \right]^{\frac{1}{2}}.
\end{equation}
Given that $\eta$ and $\xi$ are in two distinct domains we can 
exchange the operator $\underset{\eta}{\nabla}$ and the integral over $D$, with respect to the variable $\xi$, to obtain, inside the integral, the Hessian operator. Afterwards, we use the fact that $Hess \, \Phi_{0}$ has the same singularity as $\, \Phi_{0}^{3} \, \bm{I} \,$ to obtain: 
\begin{equation*} 
\left\vert T_{2}(x,s) \right\vert  \leq   \left[ \underset{B(x,s) \cap \left(\Omega \setminus D \right)}{\int} \left\vert  \underset{D}{\int} \Phi^{3}_{0}(\eta,\xi)  u_{1} (\xi) \, d\xi \right\vert^{2} \, d \eta \right]^{\frac{1}{2}}  \leq  \left\Vert u_{1} \right\Vert_{\mathbb{L}^{2}(D)} \; \left[ \underset{D}{\int} \underset{B(x,s) \cap \left(\Omega \setminus D \right)}{\int} \, \Phi^{6}_{0}(\eta,\xi) \, d\eta \, d\xi   \right]^{\frac{1}{2}}.
\end{equation*}
As $\eta$ and $\xi$ are in two disjoint domains, the function $\vartheta_{6}(\cdot)$, defined in the domain $D$, by:
\begin{equation}\label{seqfct}
\vartheta_{6}(\cdot) := \underset{B(x,s) \cap \left(\Omega \setminus D \right)}{\int} \, \Phi^{6}_{0}(\eta,\cdot) \, d\eta
\end{equation}
is a smooth one. Consequently, 
\begin{equation}\label{T2xs}
\left\vert T_{2}(x,s) \right\vert  =  \mathcal{O}\left( \left\Vert u_{1} \right\Vert_{\mathbb{L}^{2}(D)} \; a^{\frac{3}{2}} \right).
\end{equation}
\item Estimation of 
\begin{eqnarray}\label{T4XS}
\nonumber
T_{3}(x,s) &:=& \frac{\left\vert \eta(z) \right\vert^{2}}{4 \, \pi}  \underset{B(x,s) \cap \left(\Omega \setminus D \right)}{\int} \Im\left( \epsilon_{0} \right)(y) \,   \, \left\vert \nabla M(u_{1}) \right\vert^{2}(y) \; dy \\
\left\vert T_{3}(x,s) \right\vert & \lesssim &   \underset{B(x,s) \cap \left(\Omega \setminus D \right)}{\int}  \,   \, \left\vert \nabla M(u_{1}) \right\vert^{2}(y) \; dy \\ \nonumber
& = &  \underset{B(x,s) \cap \left(\Omega \setminus D \right)}{\int}    \, \left\vert  \int_{D} \underset{\eta}{Hess} \, \Phi_{0}(\eta,y) \cdot u_{1}(\eta) \, d\eta \right\vert^{2} \; dy \\ \nonumber
& \simeq &  \underset{B(x,s) \cap \left(\Omega \setminus D \right)}{\int}    \, \left\vert  \int_{D} \,\Phi^{3}_{0}(\eta,y) \, u_{1}(\eta) \, d\eta \right\vert^{2} \; dy \leq  \left\Vert u_{1} \right\Vert^{2}_{\mathbb{L}^{2}(D)} \; \int_{D} \vartheta_{6}(\eta) \, d\eta.
\end{eqnarray}
Lastly, 
\begin{equation}\label{456}
\left\vert T_{3}(x,s) \right\vert  =  \mathcal{O}\left( \left\Vert u_{1} \right\Vert^{2}_{\mathbb{L}^{2}(D)} \; a^{3} \right).
\end{equation}
\end{enumerate}
We move back to $(\ref{bfT=p0+R2+T2+...+T6})$, using $(\ref{T1=pstar}), (\ref{T2xs})$ and $(\ref{456})$, to obtain 
\begin{eqnarray}\label{bmT=pstar+R2+Err}
\nonumber
\bm{T}(x,s) &=&  p^{\star}_{0}(x,s) + R_{2}(x,s) + \mathcal{O}\left(a^{3}\right)  + \mathcal{O}\left( \left\Vert u_{1} \right\Vert_{\mathbb{L}^{2}(D)} \; a^{\frac{3}{2}} \right) + \mathcal{O}\left( \left\Vert u_{1} \right\Vert^{2}_{\mathbb{L}^{2}(D)} \; a^{3} \right) \\
& \overset{(\ref{aprioriestimateregimemoderate})}{=} &  p^{\star}_{0}(x,s) + R_{2}(x,s) + \mathcal{O}\left(a^{3-h}\right).
\end{eqnarray}
Currently, to finish with the estimation of $\bm{T}(\cdot , \cdot)$, we need to estimate $R_{2}(\cdot , \cdot)$. Therefore, we start by recalling, from $(\ref{Remainder2})$, the definition of $R_{2}(\cdot, \cdot)$,  
\begin{eqnarray}\label{R2(x,s)-expression}
\nonumber
R_{2}(x,s) &:=& \frac{1}{4 \, \pi}  \underset{B(x,s) \cap \left( \Omega \setminus D \right)}{\int} \Im\left( \epsilon_{0} \right)(y) \, R_{1}(y) \; dy \\ \nonumber
&\overset{(\ref{Remainder1})}{=}&  \Re\left[ \frac{1}{2 \, \pi}  \underset{B(x,s) \cap \left( \Omega \setminus D \right)}{\int} \Im\left( \epsilon_{0} \right)(y) \,  Err_{0}(y) \cdot   \overline{u}_{0} (y) \; dy \right] \\ \nonumber
&+&  \Re\left[ \frac{1}{2 \, \pi}  \underset{B(x,s) \cap \left( \Omega \setminus D \right)}{\int} \Im\left( \epsilon_{0} \right)(y) \,  Err_{\Gamma}(y) \cdot   \overline{u}_{0}(y) \; dy \right] \\ \nonumber
&+& \frac{1}{4 \, \pi}  \underset{B(x,s) \cap \left( \Omega \setminus D \right)}{\int} \Im\left( \epsilon_{0} \right)(y) \, \left\vert Err_{0} +  Err_{\Gamma} \right\vert^{2}(y) \; dy \\ \nonumber
&+& \Re\left[ \underset{B(x,s) \cap \left( \Omega \setminus D \right)}{\int} \frac{\eta(z)}{2 \, \pi }  \Im\left( \epsilon_{0} \right)(y) \, Err_{0}(y) \cdot \nabla M(\overline{u}_{1}) (y) \; dy \right] \\ 
&+& \Re\left[ \underset{B(x,s) \cap \left( \Omega \setminus D \right)}{\int} \frac{\eta(z)}{2 \, \pi }  \Im\left( \epsilon_{0} \right)(y) \, Err_{\Gamma}(y) \cdot \nabla M(\overline{u}_{1})(y) \; dy \right].
\end{eqnarray}
Now, as before, we need to estimate all terms appearing on the right hand side. 
\begin{enumerate}
\item Estimation of 
\begin{equation*}
T_{4}(x,s) := \Re\left[ \frac{1}{2 \, \pi}  \underset{B(x,s) \cap \left( \Omega \setminus D \right)}{\int} \Im\left( \epsilon_{0} \right)(y) \,  Err_{0}(y) \cdot   \overline{u}_{0} (y) \; dy \right].
\end{equation*}
Using the definition of the term  $Err_{0}(\cdot)$, see $(\ref{WN})$, we rewrite the previous equation as 
\begin{equation*}
T_{4}(x,s) \simeq \underset{B(x,s) \cap \left( \Omega \setminus D \right)}{\int} \Im\left( \epsilon_{0} \right)(y) \nabla M \left( u_{1}(\cdot)  \int_{0}^{1}  \nabla \eta(z+t(\cdot -z)) \cdot (\cdot -z) dt \right)(y) \cdot   \overline{u}_{0} (y) dy,
\end{equation*}
and after applying the Cauchy-Schwartz inequality we get 
\begin{eqnarray}\label{AST7}
\nonumber
\left\vert T_{4}(x,s) \right\vert & \lesssim & \left\Vert \Im\left( \epsilon_{0} \right) \, u_{0} \right\Vert_{\mathbb{L}^{2}\left( B(x,s) \cap \left( \Omega \setminus D \right)\right)} \\ \nonumber &&\; \left\Vert \nabla M\left( u_{1}(\cdot)  \int_{0}^{1}  \nabla \eta(z+t(\cdot -z)) \cdot (\cdot -z) dt \right) \right\Vert_{\mathbb{L}^{2}\left( B(x,s) \cap \left( \Omega \setminus D \right)\right)} \\
& \overset{(\ref{u0inOmegastar})}{\simeq} & \left\Vert \nabla M\left( u_{1}(\cdot)  \int_{0}^{1}  \nabla \eta(z+t(\cdot -z)) \cdot (\cdot -z) dt \right) \right\Vert_{\mathbb{L}^{2}\left( B(x,s) \cap \left( \Omega \setminus D \right)\right)}. 
\end{eqnarray}
As done in $(\ref{T4XS})$ and $(\ref{456})$ we deduce that 
\begin{eqnarray*}
\left\vert T_{4}(x,s) \right\vert & \simeq & a^{\frac{3}{2}} \; \left\Vert  u_{1}(\cdot)  \int_{0}^{1}  \nabla \eta(z+t(\cdot -z)) \cdot (\cdot -z) dt  \right\Vert_{\mathbb{L}^{2}\left(  D \right)} \leq a^{\frac{5}{2}} \; \left\Vert  u_{1}  \right\Vert_{\mathbb{L}^{2}\left(  D \right)}. 
\end{eqnarray*}
Then, 
\begin{equation}\label{5.57}
\left\vert T_{4}(x,s) \right\vert = \mathcal{O}\left( a^{\frac{5}{2}} \; \left\Vert  u_{1}  \right\Vert_{\mathbb{L}^{2}\left(  D \right)} \right).
\end{equation}
\item Estimation of 
\begin{eqnarray*}
T_{5}(x,s) &:=& \Re\left[ \frac{1}{2 \, \pi}  \underset{B(x,s) \cap \left( \Omega \setminus D \right)}{\int} \Im\left( \epsilon_{0} \right)(y) \,  Err_{\Gamma}(y) \cdot   \overline{u}_{0}(y) \; dy \right] \\
& \overset{(\ref{Err-Gamma})
}{=} & \Re\left[ \frac{- \omega^{2} \, \mu}{2 \, \pi}  \underset{B(x,s) \cap \left( \Omega \setminus D \right)}{\int} \Im\left( \epsilon_{0} \right)(y) \,  \int_{D} \Gamma(y,\eta) \cdot u_{1}(\eta) \, (\epsilon_{0} - \epsilon_{p})(\eta) \, d\eta \cdot   \overline{u}_{0} (y) \; dy \right].
\end{eqnarray*}
Taking the modulus in both sides, using Cauchy-Schwartz inequality and the estimation $(\ref{u0inOmegastar})$ we obtain 
\begin{eqnarray}\label{AST8}
\left\vert T_{5}(x,s) \right\vert & \lesssim & \left\Vert \int_{D} \Gamma(\cdot ,\eta) \cdot u_{1}(\eta) \, (\epsilon_{0} - \epsilon_{p})(\eta) \, d\eta \right\Vert_{\mathbb{L}^{2}\left( B(x,s) \cap \left( \Omega \setminus D \right)\right)} \\ \nonumber
& \lesssim & \left\Vert \int_{D} \Gamma(\cdot ,\eta) \cdot u_{1}(\eta)\, d\eta \right\Vert_{\mathbb{L}^{2}\left( B(x,s) \cap \left( \Omega \setminus D \right)\right)} \\ \label{RLCR}
& \lesssim & \left\Vert  u_{1}  \right\Vert_{\mathbb{L}^{2}\left(  D \right)} \, \left[ \int_{D} \, \underset{B(x,s) \cap \left( \Omega \setminus D \right)}{\int}  \left\vert \Gamma(\eta , y) \right\vert^{2} \, dy \, d\eta \right]^{\frac{1}{2}}.
\end{eqnarray}
Next, set $\vartheta^{\star}(\cdot)$ to be 
\begin{equation*}
\vartheta^{\star}(\cdot) := \underset{B(x,s) \cap \left( \Omega \setminus D \right)}{\int}  \left\vert \Gamma(\cdot , y) \right\vert^{2} \, dy.
\end{equation*}
Obviously, when the point $y$ is away from the particle $D$ the function $\vartheta^{\star}(\cdot)$ is smooth one, and in the other case, i.e. when $y$ is close to $D$, we use the representation $(\ref{AY})$ of $\Gamma(\cdot,\cdot)$ to rewrite $\vartheta^{\star}(\cdot)$ as:
\begin{eqnarray*}
\vartheta^{\star}(\cdot) &=& \underset{B(x,s) \cap \left( \Omega \setminus D \right)}{\int} \left\vert \frac{-1}{\omega^{2} \, \mu \, \left(\epsilon_{0}(y)\right)^{2}} \, \nabla \, \nabla M \left(  \Phi_{0}(\cdot,y) \, \nabla \epsilon_{0}(y) \right)(\cdot) + W_{4}(\cdot ,y) \right\vert^{2} \, dy \\
\vartheta^{\star}(\cdot) & \lesssim & \left\Vert \epsilon^{-2}_{0}(\cdot) \right\Vert_{\mathbb{L}^{\infty}(B(x,s) \cap \left( \Omega \setminus D \right))} \underset{B(x,s) \cap \left( \Omega \setminus D \right)}{\int} \left\vert \nabla \, \nabla M \left(  \Phi_{0}(\cdot,y)  \, \nabla \epsilon_{0}(y) \right)(\cdot)  \right\vert^{2} \, dy  \\ &+& \underset{B(x,s) \cap \left( \Omega \setminus D \right)}{\int} \left\vert W_{4}(\cdot ,y) \right\vert^{2} \, dy. 
\end{eqnarray*}
We assume that $ \left\Vert \epsilon^{-2}_{0}(\cdot) \right\Vert_{\mathbb{L}^{\infty}(B(x,s) \cap \left( \Omega \setminus D \right))} = \mathcal{O}(1)$ and, by an integration by parts, we rewrite the previous inequality as:
\begin{eqnarray*}
\vartheta^{\star}(\cdot) & \lesssim & \underset{B(x,s) \cap \left( \Omega \setminus D \right)}{\int} \left\vert - \nabla \, \nabla N \div N \left( \Phi_{0}(\cdot,y)  \nabla \epsilon_{0}(y) \right)(\cdot)  +  \nabla \, \nabla SL \left(\nu \cdot N \left( \Phi_{0} (\cdot, y) \nabla \epsilon_{0}(y) \right)(\cdot) \right) \right\vert^{2} \, dy\\ &+& \underset{B(x,s) \cap \left( \Omega \setminus D \right)}{\int} \left\vert W_{4}(\cdot ,y) \right\vert^{2} \, dy,
\end{eqnarray*}
where $SL(\cdot)$ denote the Single Layer operator defined by:
\begin{eqnarray*}
SL : \quad \mathbb{H}^{\frac{1}{2}}\left(\partial \left( B(x,s) \cap \left( \Omega \setminus D \right)\right) \right) & \rightarrow & \mathbb{H}^{2}\left( B(x,s) \cap \left( \Omega \setminus D \right) \right) \\
f(\cdot) & \rightarrow & SL \left( f \right)(\cdot) := \underset{\partial \left( B(x ,s) \cap \left( \Omega \setminus D \right)\right)}{\int} \Phi_{0}(\cdot ,y) \, f(y) \, d\sigma(y).  
\end{eqnarray*}
At this stage, for the first term, using the Calderon-Zygmund inequality and the continuity of the operator $\nabla \, \nabla SL: \mathbb{H}^{\frac{1}{2}}\left(\partial \left( B(x,s) \cap \left( \Omega \setminus D \right)\right) \right) \rightarrow \mathbb{L}^{2}\left( B(x,s) \cap \left( \Omega \setminus D \right) \right)$ we get:\footnote{We recall that:
\begin{equation*}
\left\Vert f \right\Vert^{2}_{\mathbb{H}^{\frac{1}{2}}(\partial \Omega)}
= \left\Vert f \right\Vert^{2}_{\mathbb{L}^{2}(\partial \Omega)} + \int_{\partial \Omega} \, \int_{\partial \Omega} \, \frac{\left\vert f(x) - f(y) \right\vert^{2}}{\left\vert x - y \right\vert^{3}} \, d\sigma(x) \, d\sigma(y).  
\end{equation*}

 } 
\begin{eqnarray*}
\vartheta^{\star}(\cdot) & \lesssim & \underset{B(x,s) \cap \left( \Omega \setminus D \right)}{\int} \left\vert  \div N \left( \Phi_{0}(\cdot,y)  \nabla \epsilon_{0}(y) \right)(\cdot)  \right\vert^{2} \, dy \\ &+& \left\Vert \nu \cdot N \left( \Phi_{0} (\cdot, \cdot) \nabla \epsilon_{0}(\cdot) \right) \right\Vert^{2}_{\mathbb{H}^{\frac{1}{2}} \left( \partial \left( B(x,s) \cap \left( \Omega \setminus D \right)\right)  \right) } + \underset{B(x,s) \cap \left( \Omega \setminus D \right)}{\int} \left\vert W_{4}(\cdot ,y) \right\vert^{2} \, dy.
\end{eqnarray*}
We use the continuity of the operator $\div N : \mathbb{L}^{2}\left( B(x,s) \cap \left( \Omega \setminus D \right) \right) \rightarrow \mathbb{L}^{2}\left( B(x,s) \cap \left( \Omega \setminus D \right) \right) $ and the continuity of the trace operator to obtain: 
\begin{eqnarray*}
\vartheta^{\star}(\cdot) & \lesssim & \underset{B(x,s) \cap \left( \Omega \setminus D \right)}{\int}  \Phi^{2}_{0}(\cdot,y) \, \left\vert \nabla \epsilon_{0}(y) \right\vert^{2} \, dy + \left\Vert  N \left( \Phi_{0} (\cdot, \cdot) \nabla \epsilon_{0}(\cdot) \right) \right\Vert^{2}_{\mathbb{H}^{1}  \left( B(x,s) \cap \left( \Omega \setminus D \right)  \right) } \\ &+& \underset{B(x,s) \cap \left( \Omega \setminus D \right)}{\int} \left\vert W_{4}(\cdot ,y) \right\vert^{2} \, dy. 
\end{eqnarray*}
We use the continuity of the Newtonian operator to reduce the previous inequality to: 
\begin{equation*}
\vartheta^{\star}(\cdot)  \lesssim  \left\Vert \nabla \epsilon_{0} \right\Vert^{2}_{\mathbb{L}^{\infty}(B(x,s) \cap \left( \Omega \setminus D \right))} \, \underset{B(x,s) \cap \left( \Omega \setminus D \right)}{\int}  \Phi^{2}_{0}(\cdot,y)  \, dy  + \underset{B(x,s) \cap \left( \Omega \setminus D \right)}{\int} \left\vert W_{4}(\cdot ,y) \right\vert^{2} \, dy,
\end{equation*}
and we assume that $ \left\Vert \nabla \epsilon_{0}(\cdot) \right\Vert_{\mathbb{L}^{\infty}(B(x,s) \cap \left( \Omega \setminus D \right))} = \mathcal{O}(1)$ to get: 
\begin{equation*}
\vartheta^{\star}(\cdot)  \lesssim  \underset{B(x,s) \cap \left( \Omega \setminus D \right)}{\int}  \Phi^{2}_{0}(\cdot,y)  \, dy  + \underset{B(x,s) \cap \left( \Omega \setminus D \right)}{\int} \left\vert W_{4}(\cdot ,y) \right\vert^{2} \, dy  \overset{(\ref{seqfct})}{=}  \vartheta_{2}(\cdot) + \underset{B(x,s) \cap \left( \Omega \setminus D \right)}{\int} \left\vert W_{4}(\cdot ,y) \right\vert^{2} \, dy.
\end{equation*}

Keeping the dominant term of $\vartheta^{\star}(\cdot)$  which is, as we have proved, $\vartheta_{2}(\cdot)$  to get from $(\ref{RLCR})$ the following estimation: 
\begin{equation*}
\left\vert T_{5}(x,s) \right\vert  \lesssim  \left\Vert  u_{1}  \right\Vert_{\mathbb{L}^{2}\left(  D \right)} \, \left[ \int_{D} \, \vartheta_{2}(\eta) \, d\eta \right]^{\frac{1}{2}} 
\end{equation*}
and we have seen that $\vartheta_{2}(\cdot)$ is a smooth function. Lastly, 
\begin{equation}\label{5.59}
\left\vert T_{5}(x,s) \right\vert = \mathcal{O}\left( a^{\frac{3}{2}} \, \left\Vert  u_{1}  \right\Vert_{\mathbb{L}^{2}\left(  D \right)} \right). 
\end{equation}
\item Estimation of 
\begin{eqnarray*}
T_{6}(x,s) &:=& \frac{1}{4 \, \pi}  \underset{B(x,s) \cap \left( \Omega \setminus D \right)}{\int} \Im\left( \epsilon_{0} \right)(y) \, \left\vert Err_{0} +  Err_{\Gamma} \right\vert^{2}(y) \; dy \\
\left\vert T_{6}(x,s) \right\vert & \lesssim &   \underset{B(x,s) \cap \left( \Omega \setminus D \right)}{\int} \left\vert Err_{0} \right\vert^{2}(y) \, dy +  \underset{B(x,s) \cap \left( \Omega \setminus D \right)}{\int} \left\vert Err_{\Gamma} \right\vert^{2}(y) \; dy.
\end{eqnarray*}
Using the definition of $Err_{0}(\cdot)$, see $(\ref{WN})$, we remark that the first term on the right hand side is nothing but $\left\vert T_{4}(x,s) \right\vert^{2}$, see $(\ref{AST7})$. In the same manner, the second term can be seen as $\left\vert T_{5}(x,s) \right\vert^{2}$, see $(\ref{AST8})$. Therefore, 
\begin{equation*}
\left\vert T_{6}(x,s) \right\vert  \lesssim  \left\vert T_{4}(x,s) \right\vert^{2} + \left\vert T_{5}(x,s) \right\vert^{2},
\end{equation*}
and, from $(\ref{5.57})$ and $(\ref{5.59})$, we get 
\begin{equation}\label{T9xs}
\left\vert T_{6}(x,s) \right\vert = \mathcal{O}\left(a^{3} \, \left\Vert u_{1} \right\Vert^{2}_{\mathbb{L}^{2}(D)} \right).
\end{equation}
\item Estimation of 
\begin{eqnarray*}
T_{7}(x,s) &:=&  \Re\left[ \underset{B(x,s) \cap \left( \Omega \setminus D \right)}{\int} \frac{  \eta(z)}{2 \, \pi}  \Im\left( \epsilon_{0} \right)(y) \, Err_{0}(y) \cdot \nabla M(\overline{u}_{1}) (y) \; dy \right] \\
\left\vert T_{7}(x,s) \right\vert & \lesssim &  \underset{B(x,s) \cap \left( \Omega \setminus D \right)}{\int} \left\vert \, Err_{0}(y) \cdot \nabla M(\overline{u}_{1}) (y) \right\vert \, dy \\
& \leq &  \left\Vert Err_{0} \right\Vert_{\mathbb{L}^{2}\left(B(x,s) \cap \left( \Omega \setminus D \right)\right)} \; \left\Vert \nabla M(u_{1}) \right\Vert_{\mathbb{L}^{2}\left(B(x,s) \cap \left( \Omega \setminus D \right)\right)}.
\end{eqnarray*}
Similar computations as in the estimation of $T_{2}(\cdot,\cdot)$, see $(\ref{ChinaVsUsa})$, we can estimate 
\begin{equation}\label{1843}
\left\Vert \nabla M(u_{1}) \right\Vert_{\mathbb{L}^{2}\left(B(x,s) \cap \left( \Omega \setminus D \right)\right)} = \mathcal{O}\left(a^{\frac{3}{2}} \, \left\Vert u_{1} \right\Vert_{\mathbb{L}^{2}(D)} \right)
\end{equation}
and, as we have seen in the previous computations, we can estimate  
\begin{equation*}
\left\Vert Err_{0} \right\Vert_{\mathbb{L}^{2}\left(B(x,s) \cap \left( \Omega \setminus D \right)\right)} \overset{(\ref{AST7})}{\sim} \left\vert T_{4}(x,s) \right\vert \overset{(\ref{5.57})}{=} \mathcal{O}\left( a^{\frac{5}{2}} \; \left\Vert  u_{1}  \right\Vert_{\mathbb{L}^{2}\left(  D \right)} \right).
\end{equation*}
This implies, 
\begin{equation}\label{T10xs}
\left\vert T_{7}(x,s) \right\vert = \mathcal{O}\left( a^{4} \; \left\Vert  u_{1}  \right\Vert^{2}_{\mathbb{L}^{2}\left(  D \right)} \right).
\end{equation}
\item Estimation of 
\begin{eqnarray*}
T_{8}(x,s) &:=&  \Re\left[ \underset{B(x,s) \cap \left( \Omega \setminus D \right)}{\int} \frac{  \eta(z)}{2 \, \pi}  \Im\left( \epsilon_{0} \right)(y) \, Err_{\Gamma}(y) \cdot \nabla M(\overline{u}_{1})(y) \; dy \right] \\
\left\vert T_{8}(x,s) \right\vert & \lesssim &  \underset{B(x,s) \cap \left( \Omega \setminus D \right)}{\int} \left\vert Err_{\Gamma}(y) \cdot \nabla M(\overline{u}_{1})(y) \, \right\vert \; dy \\
& \leq &  \left\Vert Err_{\Gamma} \right\Vert_{\mathbb{L}^{2}\left(B(x,s) \cap \left( \Omega \setminus D \right)\right)} \; \left\Vert \nabla M(u_{1}) \right\Vert_{\mathbb{L}^{2}\left(B(x,s) \cap \left( \Omega \setminus D \right)\right)} \\
& \overset{(\ref{1843})}{\lesssim} & \left\Vert Err_{\Gamma} \right\Vert_{\mathbb{L}^{2}\left(B(x,s) \cap \left( \Omega \setminus D \right)\right)} \; a^{\frac{3}{2}} \, \left\Vert u_{1} \right\Vert_{\mathbb{L}^{2}(D)}
\end{eqnarray*}
and, as we have seen in the previous computations, we can estimate  
\begin{equation*}
\left\Vert Err_{\Gamma} \right\Vert_{\mathbb{L}^{2}\left(B(x,s) \cap \left( \Omega \setminus D \right)\right)} \overset{(\ref{AST8})}{\simeq} \left\vert T_{5}(x,s) \right\vert \overset{(\ref{5.59})}{=} \mathcal{O}\left( a^{\frac{3}{2}} \; \left\Vert  u_{1}  \right\Vert_{\mathbb{L}^{2}\left(  D \right)} \right).
\end{equation*}
Finally,  
\begin{equation}\label{T12xs}
\left\vert T_{8}(x,s) \right\vert = \mathcal{O}\left( a^{3} \; \left\Vert  u_{1}  \right\Vert^{2}_{\mathbb{L}^{2}\left(  D \right)} \right).
\end{equation}
\end{enumerate}
Now we are able to estimate, by summing $(\ref{5.57}), (\ref{5.59}), (\ref{T9xs}), (\ref{T10xs})$ and $(\ref{T12xs})$, the term $R_{2}(x,s)$, given by  $(\ref{R2(x,s)-expression})$, and we obtain: 
\begin{equation}\label{R2estimation}
\left\vert R_{2}(x,s) \right\vert = \mathcal{O}\left( a^{\frac{3}{2}} \; \left\Vert  u_{1}  \right\Vert_{\mathbb{L}^{2}\left(  D \right)} \right).
\end{equation}
With help of $(\ref{R2estimation})$ the equation $(\ref{bmT=pstar+R2+Err})$ becomes, 
\begin{equation}\label{final-formula-bmT}
\bm{T}(x,s) =  p^{\star}_{0}(x,s) + \mathcal{O}\left(a^{3-h}\right)  + \mathcal{O}\left( \left\Vert u_{1} \right\Vert_{\mathbb{L}^{2}(D)} \; a^{\frac{3}{2}} \right) \overset{(\ref{aprioriestimateregimemoderate})}{=} p^{\star}_{0}(x,s) + \mathcal{O}\left(a^{3-h}\right).
\end{equation}

\end{proof}

We use the representation $(\ref{final-formula-bmT})$, of $\bm{T}(x,s)$, to rewrite the equation of the pressure, given by $(\ref{pressure-equation})$, in the following form
\begin{equation*}
p^{\star}(x,s) - p^{\star}_{0}(x,s) = \frac{a^{3}}{4 \, \pi} \, \Im\left( \epsilon_{p} \right) \,  \frac{\left\vert \epsilon_{0}(z)  \right\vert^{2} \,  \left\vert \langle u_{0}(z) ; \int_{B} e^{(3)}_{n_{0}}(x) \, dx \rangle \right\vert^{2}}{\left\vert \epsilon_{0}(z) - \left(\epsilon_{0}(z) - \epsilon_{p} \right) \,   \lambda^{(3)}_{n_{0}} \right\vert^{2}} + \mathcal{O}\left(  a^{\min(3-h,4-3h)}  \right). 
\end{equation*}
This proves the formula $(\ref{PrincipalFormula})$ and ends the proof of \textbf{Theorem} $\ref{PrincipalTHM}$.

\section{Construction of the Green's kernel and the Lippmann-Schwinger equation}\label{Green's-function-and-LSE}

\subsection{Construction and regularity of Green's kernel.}\label{AppendixGreenKernel}
This section is dedicated to the proof of \textbf{Theorem $\ref{DH}$}. We want to construct tensor $G_{k}(\cdot,\cdot)
$ which is solution, in the distribution sense, of 
\begin{equation}\label{equa1}    
      \left( Curl \circ Curl - \alpha(\cdot) \right) G_{k}(\cdot,z) =  \underset{z}{\delta}(\cdot) \, I \quad \text{in} \quad \mathbb{R}^{3}, 
\end{equation}
where $\alpha(\cdot) := \omega^{2}\, \mu \, \varepsilon(\cdot)$ and $\varepsilon(\cdot)$ is given by $(\ref{DefPermittivityfct})$, satisfying the following radiation condition at infinity:
\begin{equation*}
\lim_{\left\vert x \right\vert \rightarrow +\infty} \;\; \left\vert x \right\vert \;\; \left( \underset{y}{\nabla} \times G_{k}(x,y) \times \frac{x}{\left\vert x \right\vert} - i \, k \, G_{k}(x,y) \right) = 0.
\end{equation*}
When we deal with the Maxwell system, in contrast to elliptic systems as in the Helmholtz and elastic cases, we need to be more careful because of the strong singularity of its Green's kernel. For this we decompose the kernel $G_{k}(\cdot,\cdot)$ as
\begin{equation}\label{defW}
G_{k}(\cdot,z) = \Upsilon_{k}(\cdot,z) + \Gamma(\cdot,z),
\end{equation}
where $\Upsilon_{k}(\cdot,z)$, the Green's tensor of Maxwell equations for the free space with wave number $k := \omega \, \sqrt{\mu \epsilon_{\infty}} $, is solution of 
\begin{equation}\label{equa2}
\left( Curl \circ Curl - \alpha_{\infty} \right) \Upsilon_{k}(\cdot,z) =  \underset{z}{\delta}(\cdot) \, I,
\end{equation} 
and is given by 
\begin{equation}\label{equa3}
\Upsilon_{k}(\cdot,z) := \Phi_{k}(\cdot,z) \, I + \frac{1}{\alpha_{\infty}} \; \nabla \, \div \left( \Phi_{k}(\cdot,z) \, I \right). 
\end{equation}
Combining $(\ref{equa1}), (\ref{defW})$ and $(\ref{equa2})$, we obtain   
\begin{equation}\label{LZB}
\left( Curl \circ Curl - \alpha(\cdot) \right) \Gamma(\cdot, z) = \left( \alpha(\cdot) - \alpha_{\infty} \right) \, \Upsilon_{k}(\cdot,z) := F(\cdot, z) \quad \text{in} \quad  \mathbb{R}^{3}.
\end{equation}
From the definition of the permittivity $\varepsilon(\cdot)$, see $ (\ref{DefPermittivityfct})$, we remark that $F(\cdot,z)$ is of compact support that we note, in the sequel, by $\Omega$. Now, as 
\begin{equation*}
\Upsilon_{k}(\cdot,z) = \frac{1}{\alpha_{\infty}} Curl \circ Curl \left( \Phi_{k} \, I \right)(\cdot,z) - \frac{1}{\alpha_{\infty}}  \underset{z}{\delta}(\cdot) \, I , 
\end{equation*}
we rewrite $F(\cdot , z)$ as 
\begin{eqnarray*}\label{ADZZK}
\nonumber
F(\cdot , z) & = & \frac{\left( \alpha(\cdot) - \alpha_{\infty} \right)}{\alpha_{\infty}} \; Curl \circ Curl \left( \Phi_{k} \, I \right)(\cdot, z) + \frac{\left( \alpha_{\infty} - \alpha(\cdot) \right)}{\alpha_{\infty}}  \underset{z}{\delta}(\cdot) \, I  \\ \nonumber
& = & \frac{\left( \alpha(\cdot) - \alpha_{\infty} \right)}{\alpha_{\infty}} \; Curl \circ Curl \left( \Phi_{k} \, I \right)(\cdot, z) + \frac{\left( \epsilon_{\infty} - \epsilon_{0} (z) \right)}{\epsilon_{\infty}}  \underset{z}{\delta}(\cdot) \, I  \\
& = & \frac{1}{\alpha_{\infty}} \; \left[ Curl \left( \left( \alpha(\cdot) - \alpha_{\infty} \right) Curl \left( \Phi_{k} \, I \right)(\cdot, z)  \right) - \nabla \alpha(\cdot)  \times Curl\left( \Phi_{k} \, I \right)(\cdot, z) \right] + \frac{\left( \epsilon_{\infty} - \epsilon_{0} (z) \right)}{\epsilon_{\infty}}  \underset{z}{\delta}(\cdot) \, I.
\end{eqnarray*}
This allows to write $F(\cdot,z)$ as 
\begin{equation}\label{F=Curlf+g}
F(\cdot,z) = Curl(f)(\cdot,z) + g(\cdot,z) + \frac{\left( \epsilon_{\infty} - \epsilon_{0} (z) \right)}{\epsilon_{\infty}}  \underset{z}{\delta}(\cdot) \, I,
\end{equation}
where $f(\cdot,z)$ is given by 
\begin{equation*}\label{def-f}
f(\cdot,z) = \frac{\left( \alpha(\cdot) - \alpha_{\infty} \right)}{\alpha_{\infty}} \; Curl \left( \Phi_{k} \, I \right)(\cdot, z)    \quad \text{in} \quad \mathbb{R}^{3}\setminus \{ z \}
\end{equation*}
and $g(\cdot,z)$ is given by
\begin{equation}\label{FW}
g(\cdot,z)  = \frac{-1}{\alpha_{\infty}} \;\nabla \alpha(\cdot)  \times Curl\left( \Phi_{k} \, I \right)(\cdot, z) \quad \text{in} \quad \mathbb{R}^{3}\setminus \{ z \}.
\end{equation}
Remark that even if the two functions $f$ and $g$ are defined in the entire space $\mathbb{R}^{3}$, except the point $z$, they still have a compact support, given by $\Omega$, and this is due to the definition of the permittivity function $\varepsilon(\cdot)$. Using the decomposition $(\ref{F=Curlf+g})$, of $F(\cdot,z)$, we rewrite $(\ref{LZB})$ as: 
\begin{equation}\label{NLCS}
      \left( Curl \circ Curl - \alpha(\cdot) \, I \right) \Gamma(\cdot,z) = Curl(f)(\cdot,z) + g(\cdot,z) + \frac{\left( \epsilon_{\infty} - \epsilon_{0} (z) \right)}{\epsilon_{\infty}}  \underset{z}{\delta}(\cdot) \, I. 
\end{equation}
Clearly, the regularity of $\Gamma$ depends on the regularity of the data sources $f$ and $g$.  For this, as first step, we need
the next lemma to get  precisions about the integrability of $f$ and $g$.
\begin{lemma}\label{SE}
For the functions $g$ and $f$, we have: 
\begin{equation*}
g \in \mathbb{L}^{p}(\Omega) \quad \text{and} \quad Curl(f) \in \left( \mathbb{H}_{0}\left(Curl, q \right) \right)^{\prime},
\end{equation*} 
where 
\begin{equation*}
p = \frac{3}{2} - \delta \quad \text{and q its conjugate number} \quad q := \frac{p}{p-1}=\frac{3 - 2 \delta}{1 - 2 \delta}.  
\end{equation*}
\end{lemma}
\begin{proof}
From the definition of the function $g$, see $(\ref{FW})$, we obtain:
\begin{equation*}
\left\Vert g \right\Vert^{p}_{\mathbb{L}^{p}(\Omega)} \leq \epsilon^{-p}_{\infty} \, \left\Vert \nabla \epsilon_{0}(\cdot) \right\Vert^{p}_{\mathbb{L}^{\infty}(\Omega)}  \, \int_{\Omega}  \left\vert Curl\left( \Phi_{\omega} \, \bm{I} \right)(x,z) \right\vert^{p} \, dx \leq \epsilon^{-p}_{\infty} \, \left\Vert \nabla \epsilon_{0}(\cdot) \right\Vert^{p}_{\mathbb{L}^{\infty}(\Omega)}  \, 2^{1/2p} \, \int_{\Omega}  \left\vert \nabla \Phi_{\omega}(x,z) \right\vert^{p} \, dx.
\end{equation*}
As $\epsilon_{0}(\cdot)$ is smooth and knowing that $\nabla \Phi_{\omega} \sim \Phi^{2}_{0}$, in singularity analysis point of view, we obtain 
\begin{equation*}
\left\Vert g \right\Vert^{p}_{\mathbb{L}^{p}(\Omega)} \lesssim \, \int_{\Omega}  \left\vert  \Phi_{0}(x,z) \right\vert^{2p} \, dx = \int_{\Omega} \frac{1}{ \left\vert x - z \right\vert^{2p}} \, dx.
\end{equation*}
This last integral is finite if $p < \frac{3}{2}$, which is equivalent to take, $
p = \frac{3}{2} - \delta, \,\, \forall\; \delta>0.$
We denote by $q$ the conjugate number of $p$ given by
$q = \frac{3 - 2 \delta}{1 - 2 \delta}.$
Similarly $f$ is in $\mathbb{L}^{p}(\Omega)$, therefore $Curl(f)$ is a distribution that extends as an element of 
$\left(\mathbb{H}_{0}\left(Curl, q \right) \right)^{\prime}$ since $\mathcal{C}^{\infty}_{0}(\Omega)$ is dense in $\mathbb{H}_{0}\left(Curl, q \right)$. 

\end{proof}
Now that we know the regularity of the second term of the problem $(\ref{NLCS})$, we investigate its solvability and the integrability of its solution. To do this, we split the problem $(\ref{NLCS})$ into three sub-problems:  
\begin{eqnarray}
\nonumber
      \left( Curl \circ Curl - \alpha(\cdot) \, I \right) W_{1} &=& Curl(f), \quad \quad \quad \quad \quad \,\, \text{in} \quad  \mathbb{R}^{3}\\
\label{NLCSW2}
      \left( Curl \circ Curl - \alpha(\cdot) \, I \right) W_{2} &=& g, \quad \quad \quad \quad \quad \quad \quad \quad \, \,  \text{in} \quad \mathbb{R}^{3} \\
\label{NLCSGammadelta}
      \left( Curl \circ Curl - \alpha(\cdot) \, I \right) \Gamma^{\delta} &=&  \frac{\left( \epsilon_{\infty} - \epsilon_{0} (z) \right)}{\epsilon_{\infty}}  \underset{z}{\delta}(\cdot) \, I, \quad \text{in} \quad \mathbb{R}^{3}. 
\end{eqnarray}
For the first sub-problem, to analyze the regularity of $W_{1}$, we start by the following lemma.  
\begin{lemma}\label{IYLASS}
There exists one and only one solution $W_{1}$ of: 
\begin{equation}\label{KdrJap}
Curl \circ Curl (W_{1}) - W_{1} = Curl(f), \quad \text{in} \quad \mathbb{R}^{3},
\end{equation}
satisfying the Silver-M\"{u}ller radiation condition
\begin{equation}\label{SMRC}
\underset{\left\vert x \right\vert \rightarrow + \infty}{\lim} \,  \left\vert x \right\vert \, \left( \nabla \times W_{1}(x) \times \frac{x}{\left\vert x \right\vert} - i  \, W_{1}(x) \right) = 0.
\end{equation}
In addition it is in $\mathbb{L}_{Loc}^{\frac{3(3-2\delta)}{(3+2\delta)}}(\mathbb{R}^{3}).$
\end{lemma}
\begin{proof}  
By taking the divergence on both sides of $(\ref{KdrJap})$ we deduce that $\div\left( W_{1} \right) = 0$, in $\mathbb{R}^{3}$, and then $(\ref{KdrJap})$ will be reduced to the following vectorial Helmholtz system:
\begin{equation*}\label{VHE}
\Delta W_{1} + W_{1}  = - Curl\left( f \right), \qquad \text{in} \quad \mathbb{R}^{3}, 
\end{equation*}  
and each component of $W_1$ satisfies the Sommerfeld radiation condition.
The solution to this problem is unique and it is given by:
\begin{equation}\label{equa448}
W_{1} =  Curl N^{1}\left( f \right),
\end{equation}
where $N^{1}(\cdot)$ is the operator defined by $(\ref{DefNDefMk})$. Using the continuity of the operator $Curl N^1(\cdot)$, see Theorem 1 of \cite{CZ}, 
we get:
\begin{equation}\label{normWLpR3}
\left\Vert W_{1} \right\Vert_{\mathbb{L}_{Loc}^{p}(\mathbb{R}^{3})} \leq C(1,p) \,\, \left\Vert f \right\Vert_{\mathbb{L}^{p}(\Omega)}.
\end{equation}
This proves that $W_{1} \in \mathbb{L}_{Loc}^{p}(\mathbb{R}^{3}), \, p=\frac{3}{2} - \delta$. Now taking the Curl operator on both sides of $(\ref{equa448})$, we get 
\begin{equation*}
Curl (W_{1}) =  Curl \circ Curl N^{1}\left( f \right) =(- \Delta + \nabla \div) N^{1}\left( f \right)  =  N^{1}\left( f \right) + f + \nabla \div N^{1}\left( f \right).
\end{equation*}
Then,  
\begin{equation*}
\left\Vert Curl (W_{1}) \right\Vert_{\mathbb{L}_{Loc}^{p}(\mathbb{R}^{3})}  \leq  \left\Vert N^{1}\left( f \right) \right\Vert_{\mathbb{L}_{Loc}^{p}(\mathbb{R}^{3})} + \left\Vert f \right\Vert_{\mathbb{L}^{p}(\mathbb{R}^{3})} + \left\Vert \nabla \div N^{1}\left( f \right)\right\Vert_{\mathbb{L}_{Loc}^{p}(\mathbb{R}^{3})}.
\end{equation*}
Now, using the Calderon-Zygmund inequality, see Theorem 9.9 of \cite{gilbarg2001elliptic}, and the continuity of the Newtonian potential operator $N(\cdot)$, see for instance Theorem 1, page 132, of \cite{Stein}, we obtain
\begin{equation*}
\left\Vert Curl (W_{1}) \right\Vert_{\mathbb{L}_{Loc}^{p}(\mathbb{R}^{3})} \leq  (C(0,p)+1+C(2,p)) \, \left\Vert f \right\Vert_{\mathbb{L}^{p}(\Omega)}.
\end{equation*}
This proves that $Curl \left( W_{1} \right) \in \mathbb{L}_{Loc}^{p}(\mathbb{R}^{3}), \, p=\frac{3}{2} - \delta$. 
More regularity for $W_{1}$ can be obtained by taking the gradient operator on both sides of $(\ref{equa448})$ to get:
\begin{equation*}
\nabla W_{1} = \nabla Curl N^{1}\left( f \right),
\end{equation*}
and then
\begin{equation}\label{G2}
\left\Vert \nabla W_{1} \right\Vert_{\mathbb{L}_{Loc}^{p}(\mathbb{R}^{3})} = \left\Vert \nabla Curl N^{1}\left( f \right)\right\Vert_{\mathbb{L}_{Loc}^{p}(\mathbb{R}^{3})} < \left\Vert Hess N^{1}\left( f \right)\right\Vert_{\mathbb{L}_{Loc}^{p}(\mathbb{R}^{3})} \leq C(2,p) \,\, \left\Vert  f \right\Vert_{\mathbb{L}^{p}(\Omega)},
\end{equation}
which proves, by gathering  $(\ref{normWLpR3})$ with $(\ref{G2})$, that $W_{1} \in \mathbb{W}_{Loc}^{1,p}(\mathbb{R}^{3}), \, p=\frac{3}{2}-\delta$. At this stage we use Sobolev embedding theorem to get the following injection 
\begin{equation*}
\mathbb{W}_{Loc}^{1,\frac{3}{2}-\delta}(\mathbb{R}^{3}) \hookrightarrow \mathbb{L}_{Loc}^{\frac{3(3-2\delta)}{(3+2\delta)}}(\mathbb{R}^{3}).
\end{equation*}  
Finally, 
\begin{equation*}
W_{1} \in \mathbb{L}_{Loc}^{\frac{3(3-2\delta)}{(3+2\delta)}}(\mathbb{R}^{3}).
\end{equation*}
As we have equivalence between the Silver-M\"{u}ller radiation condition for $W_{1}$ and the Sommerfeld radiation condition for the Cartesian component of $W_{1}$, see for instance \cite{colton2019inverse}, Theorem 6.8, then $W_1$ solves the problem (\ref{KdrJap})-(\ref{SMRC}). Finally, we have uniqueness of the problem (\ref{KdrJap})-(\ref{SMRC}) with standard arguments, see \cite{colton2019inverse}.
\end{proof} 
The next lemma extends the previous result to the perturbed problem case. 
\begin{lemma}\label{DjHCV19}
For $\alpha(\cdot) \in \mathcal{C}^{1}(\mathbb{R}^{3})$,  constant outside a bounded domain and such that $\Im\left(\alpha(\cdot) \right)>0$ the solution $W_{1}$ to 
\begin{equation}\label{KdrJapCS}
Curl \circ Curl (W_{1}) - \alpha(\cdot) \, W_{1} = Curl(f), \quad \text{in} \quad \mathbb{R}^{3},
\end{equation}
satisfying the Silver-M\"{u}ller radiation condition
\begin{equation*}
\underset{\left\vert x \right\vert \rightarrow + \infty}{\lim} \,  \left\vert x \right\vert \, \left( \nabla \times W_{1}(x) \times \frac{x}{\left\vert x \right\vert} - i \, \sqrt{\alpha_{\infty}} \, W_{1}(x) \right) = 0,
\end{equation*}
is  in $ \mathbb{L}^{\frac{3(3-2\delta)}{(3+2\delta)}}(\Omega)$, where $\Omega$ is the support of the function $f$. 
\end{lemma} 
\begin{proof} 
Without loss of generalities, we set:
\begin{equation}\label{alphaCte}
\alpha_{\infty} := \alpha(\cdot)_{\displaystyle|_{\mathbb{R}^{3} \setminus \Omega}},
\end{equation}
where $\Omega$ is the support of the function $f$ and we rewrite $(\ref{KdrJap})$ as 
\begin{equation*}
       Curl \circ Curl \left( W_{1} \right) - \alpha_{\infty} W_{1}  =  Curl\left( f \right) + \left(\alpha(\cdot) - \alpha_{\infty}  \right) W_{1}.
\end{equation*}
As already done in the unperturbed problem case, see Lemma \ref{IYLASS}, we represent the solution $W_{1}$ as 
\begin{equation}\label{PMMS}
W_{1} = Curl N^{\sqrt{\alpha_{\infty}}}(f) - \frac{1}{\alpha_{\infty}} \nabla M^{\sqrt{\alpha_{\infty}}}\left( \left( \alpha(\cdot) - \alpha_{\infty} \right) W_{1} \right) + N^{\sqrt{\alpha_{\infty}}}\left( \left(  \alpha(\cdot)-\alpha_{\infty} \right) W_{1} \right), \quad \text{in} \quad \mathbb{R}^{3}.
\end{equation}
As the support of $\left( \alpha(\cdot) - \alpha_{\infty} \right)$ is $\Omega$, we can see that the value of $W_{1}$ in $\mathbb{R}^{3}$ is given by its value in $\Omega$. So we need to prove that, when restricted to $\Omega$, the solution $W_{1}$ is well defined. It is natural to look for solutions of $(\ref{PMMS})$ in the $\mathbb{L}^p(\Omega)$-spaces. For this, taking the $\mathbb{L}^{p}(\Omega)$-norm on both sides of (\ref{PMMS}) and using the triangular inequality, we obtain  
\begin{eqnarray*}
\left\Vert W_{1} \right\Vert_{\mathbb{L}^{p}(\Omega)} & \leq & \left\Vert Curl N^{\sqrt{\alpha_{\infty}}}(f) \right\Vert_{\mathbb{L}^{p}(\Omega)} + \frac{1}{\left\vert \alpha_{\infty} \right\vert} \left\Vert \nabla \nabla N^{\sqrt{\alpha_{\infty}}}\left( \left( \alpha_{\infty} -\alpha(\cdot) \right) W_{1} \right) \right\Vert_{\mathbb{L}^{p}(\Omega)} \\ &+& \left\Vert N^{\sqrt{\alpha_{\infty}}}\left( \left( \alpha_{\infty} -\alpha(\cdot) \right) W_{1} \right)\right\Vert_{\mathbb{L}^{p}(\Omega)} \\
\left\Vert W_{1} \right\Vert_{\mathbb{L}^{p}(\Omega)} & \leq & C(1,p) \, \left\Vert f \right\Vert_{\mathbb{L}^{p}(\Omega)} + \left( \frac{C(2,p)}{\left\vert \alpha_{\infty} \right\vert} + C(0,p) \right) \left\Vert  \left( \alpha_{\infty} -\alpha(\cdot) \right)   \right\Vert_{\mathbb{L}^{\infty}(\Omega)} \, \left\Vert  W_{1}  \right\Vert_{\mathbb{L}^{p}(\Omega)}.
\end{eqnarray*}
Now, for $p = \frac{3}{2} - \delta$, we assume  that: 
\begin{equation}\label{Cdtalpha}
1 > \left( \frac{C(2,p)}{\left\vert \alpha_{\infty} \right\vert}  + C(0,p) \right) \left\Vert  \left( \alpha_{\infty} -\alpha(\cdot) \right)  \right\Vert_{\mathbb{L}^{\infty}(\Omega)}.
\end{equation}
Then, under this condition,  we end up with:
\begin{equation}\label{Wfctf}
\left\Vert W_{1} \right\Vert_{\mathbb{L}^{p}(\Omega)}  \leq  \frac{C(1,p)}{\left[1 - \left( \frac{C(2,p)}{\left\vert \alpha_{\infty} \right\vert}  + C(0,p) \right) \left\Vert  \left( \alpha_{\infty} -\alpha(\cdot) \right)  \right\Vert_{\mathbb{L}^{\infty}(\Omega)} \right]} \, \left\Vert f \right\Vert_{\mathbb{L}^{p}(\Omega)}.
\end{equation}
This proves that, under the condition $(\ref{Cdtalpha})$, the equation $(\ref{PMMS})$ is invertible in $\mathbb{L}^p(\Omega)$. Therefore, again from $(\ref{PMMS})$, we see that $W_1$ is well defined in $\mathbb{R}^{3}$ and satisfies the Silver-M\"{u}ller radiation condition. Hence $W_1$ in $(\ref{PMMS})$ solves $(\ref{KdrJapCS})$. The uniqueness of the solution of the problem $(\ref{KdrJapCS})$ is known under the condition $\Im\left( \alpha(\cdot) \right) > 0$, see for instance Theorem 9.4 in \cite{colton2019inverse}. 
\bigskip

Let us now examine the regularity of this solution. For this, taking the $Curl$ operator on both sides of $(\ref{PMMS})$ to obtain: 
\begin{equation*}
Curl \left( W_{1} \right) = Curl \circ Curl N^{\sqrt{\alpha_{\infty}}}(f) - Curl N^{\sqrt{\alpha_{\infty}}}\left( \left( \alpha_{\infty} -\alpha(\cdot) \right) W_{1} \right) \end{equation*}
and by the use of the relation $Curl \circ Curl (\cdot) = - \Delta (\cdot) + \nabla \nabla \cdot(\cdot)$ we obtain: 
\begin{equation*}
Curl \left( W_{1} \right) = \alpha_{\infty} N^{\sqrt{\alpha_{\infty}}}(f)  + f + \nabla \div N^{\sqrt{\alpha_{\infty}}}(f) - Curl N^{\sqrt{\alpha_{\infty}}}\left( \left( \alpha_{\infty} -\alpha(\cdot) \right) W_{1} \right).
\end{equation*}
Hence, 
\begin{eqnarray*}\label{CurlPMMS}
\left\Vert Curl \left( W_{1} \right) \right\Vert_{\mathbb{L}^{p}(\Omega)} & \leq & \left\vert \alpha_{\infty} \right\vert \,\, \left\Vert N^{\sqrt{\alpha_{\infty}}}(f) \right\Vert_{\mathbb{L}^{p}(\Omega)} + \left\Vert f \right\Vert_{\mathbb{L}^{p}(\Omega)} + \left\Vert \nabla \div N^{\sqrt{\alpha_{\infty}}}(f) \right\Vert_{\mathbb{L}^{p}(\Omega)} \\ &+& \left\Vert Curl N^{\sqrt{\alpha_{\infty}}}\left( \left( \alpha_{\infty} -\alpha(\cdot) \right) W_{1} \right) \right\Vert_{\mathbb{L}^{p}(\Omega)} \\
\left\Vert Curl \left( W_{1} \right) \right\Vert_{\mathbb{L}^{p}(\Omega)} & \leq & \left[ \left\vert \alpha_{\infty} \right\vert \,C(0,p)  + 1 + C(2,p) \right] \, \left\Vert f \right\Vert_{\mathbb{L}^{p}(\Omega)} + C(1,p) \left\Vert  \left( \alpha_{\infty} -\alpha(\cdot) \right) \right\Vert_{\mathbb{L}^{\infty}(\Omega)} \, \left\Vert W_{1}  \right\Vert_{\mathbb{L}^{p}(\Omega)}.
\end{eqnarray*}
This last inequality combined with $(\ref{Wfctf})$ give us: 
\begin{eqnarray}\label{CurlWfctf}
\nonumber
\left\Vert Curl \left( W_{1} \right) \right\Vert_{\mathbb{L}^{p}(\Omega)} & \leq & \Bigg[ \left\vert \alpha_{\infty} \right\vert \,C(0,p)  + 1 + C(2,p) + \\ &&   \frac{(C(1,p))^{2} \, \left\Vert  \left( \alpha_{\infty} -\alpha(\cdot) \right) \right\Vert_{\mathbb{L}^{\infty}(\Omega)}}{\left[1 - \left( \frac{C(2,p)}{\left\vert \alpha_{\infty} \right\vert}  + C(0,p) \right) \left\Vert  \left( \alpha_{\infty} -\alpha(\cdot) \right)  \right\Vert_{\mathbb{L}^{\infty}(\Omega)} \right]} \Bigg] \left\Vert f \right\Vert_{\mathbb{L}^{p}(\Omega)}.
\end{eqnarray}
This proves that $Curl \left( W_{1} \right) \in \mathbb{L}^{p}(\Omega), \, p = \frac{3}{2} - \delta$. In addition, by taking the divergence operator on both sides of $(\ref{PMMS})$ we get $\div\left( \alpha(\cdot) \, W_{1} \right) = 0,$ in $\mathbb{R}^{3}$,
which, under the condition $\alpha(\cdot)_{\displaystyle|_{\Omega}} \neq 0$, is equivalent to 
\begin{equation*}
\div\left( W_{1} \right) = - \, \alpha(\cdot)^{-1} \, W_{1} \cdot \nabla \alpha(\cdot).
\end{equation*} 
From the last equality and the constancy of $\alpha(\cdot)$ outside $\Omega$, see $(\ref{alphaCte})$, we deduce that:
\begin{equation}\label{LRRA} 
\div\left( W_{1} \right) \in \mathbb{L}^{p}(\mathbb{R}^{3}),\,\, p = \frac{3}{2}-\delta \quad \text{and} \quad Supp\left( \div (W_{1}) \right) \subseteq \Omega.
\end{equation}
Now, set $\phi \in \mathcal{C}^\infty(\mathbb{R}^{3})$ such that: 
\begin{align*}
  \begin{cases}
      \phi = 1 & \text{in $\Omega$,} \\
      0 \leq \phi \leq 1 & \text{in $B(0,R) \setminus \Omega$,} \\
      \phi = 0 & \text{in $\mathbb{R}^{3} \setminus B(0,R)$,}
    \end{cases}
\end{align*}
where $B(0,R)$ is a ball of center $0$ and a large  radius $R$ containing the domain $\Omega$. Clearly, from $(\ref{Wfctf})$, the vector field $\phi \, W_{1} \in \mathbb{L}^{p}(B(0,R))$ with $p = \frac{3}{2} - \delta$. Also, we have 
\begin{eqnarray*}
Curl\left( \phi \, W_{1} \right) &=& \nabla \phi \times W_{1} + \phi \, Curl\left( W_{1} \right) \\
\left\Vert Curl\left( \phi \, W_{1} \right) \right\Vert_{\mathbb{L}^{p}(B(0,R))}  & \leq & \sqrt{2} \, \left\Vert \nabla \phi \right\Vert_{\mathbb{L}^{\infty}(B(0,R))} \, \left\Vert W_{1} \right\Vert_{\mathbb{L}^{p}(B(0,R))} +  \left\Vert Curl\left( W_{1} \right) \right\Vert_{\mathbb{L}^{p}(B(0,R))},
\end{eqnarray*}  
then, from $(\ref{Wfctf})$ and $(\ref{CurlWfctf})$   , we deduce that $Curl\left( \phi \, W_{1} \right) \in \mathbb{L}^{p}(B(0,R))$ with $p = \frac{3}{2} - \delta$. In a similar manner, we have: 
\begin{eqnarray*}
\div\left( \phi \, W_{1} \right) &=& W_{1} \cdot \nabla \phi  + \phi \, \div\left( W_{1} \right) \\
\left\Vert \div\left( \phi \, W_{1} \right) \right\Vert_{\mathbb{L}^{p}(B(0,R))}  & \leq &  \left\Vert \nabla \phi \right\Vert_{\mathbb{L}^{\infty}(B(0,R))} \, \left\Vert W_{1} \right\Vert_{\mathbb{L}^{p}(B(0,R))} +  \left\Vert \div\left( W_{1} \right) \right\Vert_{\mathbb{L}^{p}(B(0,R))},
\end{eqnarray*}
then, from $(\ref{Wfctf})$ and $(\ref{LRRA})$, we deduce that $\div\left( \phi \, W_{1} \right) \in \mathbb{L}^{p}(B(0,R))$ with $p = \frac{3}{2} - \delta$. \\
At this stage we have $\phi \, W_{1}, Curl(\phi \, W_{1})$ and $\div(\phi \, W_{1})$ are elements in the space $\mathbb{L}^{p}\left(B(0,R) \right),\, p=\frac{3}{2}-\delta,$ and $\nu \times \phi \, W_{1} = 0$ on the boundary $\partial B(0,R)$. This is sufficient to justify, for reference see \cite{Amrouche-Houda}, that $\phi \, W_{1} \in \mathbb{W}^{1,p}\left(B(0,R) \right),\, p=\frac{3}{2}-\delta$. Using embedding results for Sobolev spaces, we obtain: 
\begin{equation*}
\phi \, W_{1} \in \mathbb{W}^{1,\frac{3}{2}-\delta}\left(B(0,R) \right) \lhook\joinrel\xrightarrow{\text{continously}} \mathbb{L}^{\frac{3(3-2\delta)}{(3+2\delta)}}\left(B(0,R) \right). 
\end{equation*}
Finally, by restriction to the domain $\Omega$, we get $W_{1} \in \mathbb{L}^{\frac{3(3-2\delta)}{(3+2\delta)}}(\Omega).$
\end{proof}
Let us now study  the existence and regularity of $W_{2}$, solution of $(\ref{NLCSW2})$. 
\begin{lemma}
For $\alpha(\cdot) \in \mathcal{C}(\mathbb{R}^{3})$,  constant outside a bounded domain and such that $\Im\left(\alpha(\cdot) \right)>0$ there exists one and only one solution $W_{2}$ of 
\begin{equation*}
 \left( Curl \circ Curl - \alpha(\cdot) \, I \right) W_{2} = g,  \,\,  \text{in} \quad \mathbb{R}^{3},
\end{equation*}
satisfying the Silver-M\"{u}ller radiation condition
\begin{equation*}
\underset{\left\vert x \right\vert \rightarrow + \infty}{\lim} \,  \left\vert x \right\vert \, \left( \nabla \times W_{2}(x) \times \frac{x}{\left\vert x \right\vert} - i \, \sqrt{\alpha_{\infty}} \, W_{2}(x) \right) = 0.
\end{equation*}
In addition, it is in $\mathbb{H}(Curl,p,\Omega)$, with $p = \frac{3}{2} - \delta$. 
\end{lemma}
\begin{proof}
Similarly as in the previous lemma, we begin by the integral equation representation of the solution $W_{2}$, which will be given by: 
\begin{equation}\label{ASVS}
W_{2} = N^{\sqrt{\alpha_{\infty}}}\left( \left(  \alpha(\cdot) - \alpha_{\infty}  \right)  W_{2} + g \right) - \frac{1}{\alpha_{\infty}} \nabla M^{\sqrt{\alpha_{\infty}}}\left( \left(  \alpha(\cdot) - \alpha_{\infty}  \right)  W_{2} + g \right), \quad \text{in} \quad \mathbb{R}^{3}.
\end{equation}
Again as in the previous lemma, restricted to $\Omega$, we see at (\ref{ASVS}) as an integral equation in the $\mathbb{L}^p(\Omega)$-spaces. Taking the $\mathbb{L}^p(\Omega)$-norm on both sides of the previous equation and using the condition, on the function $\alpha(\cdot)$, given by $(\ref{Cdtalpha})$, we obtain: 
\begin{equation}\label{EstimationW2}
\left\Vert W_{2} \right\Vert_{\mathbb{L}^{p}(\Omega)} \leq  \frac{\left( C(0,p) + \frac{C(2,p)}{\left\vert \alpha_{\infty} \right\vert} \right)}{\left[1 - \left( \frac{C(2,p)}{\left\vert \alpha_{\infty} \right\vert}  + C(0,p) \right) \left\Vert  \left( \alpha_{\infty} -\alpha(\cdot) \right)  \right\Vert_{\mathbb{L}^{\infty}(\Omega)} \right]} \,\, \left\Vert g \right\Vert_{\mathbb{L}^{p}(\Omega)}, \quad \text{where} \quad p = \frac{3}{2} - \delta. 
\end{equation}
This proves that $W_{2} \in \mathbb{L}^{p}(\Omega)$, with $p = \frac{3}{2} - \delta$, and then $W_{2}$ is well defined. Now, by taking the Curl operator of $(\ref{ASVS})$, the $\mathbb{L}^{p}(\mathbb{R}^{3})$-norm and using the estimation $(\ref{EstimationW2})$ we get: 
\begin{equation}\label{CurlEstimationW2}
\left\Vert Curl\left( W_{2} \right) \right\Vert_{\mathbb{L}^{p}(\Omega)} \leq  \frac{ C(1,p)  }{\left[1 - \left( \frac{C(2,p)}{\left\vert \alpha_{\infty} \right\vert}  + C(0,p) \right) \left\Vert  \left( \alpha_{\infty} -\alpha(\cdot) \right)  \right\Vert_{\mathbb{L}^{\infty}(\Omega)} \right]} \,\, \left\Vert g \right\Vert_{\mathbb{L}^{p}(\Omega)}, \, \text{where} \,\, p = \frac{3}{2} - \delta. 
\end{equation}
By gathering $(\ref{EstimationW2})$ with $(\ref{CurlEstimationW2})$ we deduce that $W_{2}$, solution of $(\ref{NLCSW2})$, will be in $\mathbb{H}(Curl,p,\Omega)$ with $p = \frac{3}{2} - \delta$.
\end{proof}
 
The previous lemma give us the global regularity of  $W_{2}$, solution of $(\ref{NLCSW2})$. The goal of the coming lemma is to provide a more explicit expression of the dominant term of $W_{2}$ near the fixed point $z$.
\begin{lemma}\label{ExpressionW2} 
For $x$ in $D^{\star}$, where $D^{\star}$ is the ball of center $z$ and radius $a^{\star}$ such that $a^{\star} > a$, we have: 
\begin{equation*}
W_{2}(x,z) = \frac{-1}{\alpha(z) \, \epsilon_{\infty}} \, \underset{x}{\nabla}   \underset{x}{\nabla} M\left(\Phi_{0}(\cdot,z)  \nabla \epsilon_{0}(z) \right)(x)  + W_{3}(x,z),
\end{equation*}
where $W_{3}(\cdot,z) \in  \mathbb{L}^{\frac{3(3-2\delta)}{(3+2\delta)}}\left(D^{\star}\right)$.
\end{lemma}
\begin{proof}
We begin by setting, in the ball $D^{\star}$, the PDE satisfied by $W_{2}(\cdot,z)$, which is: 
\begin{equation}\label{044prime}
  \underset{x}{Curl} \circ \underset{x}{Curl} \left( W_{2} \right)(x,z) - \alpha(x)  W_{2}(x,z) = g(x,z), \quad x \in D^{\star}.
\end{equation}
As it was shown previously, see $(\ref{EstimationW2})$, the solution $W_{2}(\cdot,z) \in \mathbb{L}^{\frac{3}{2}-\delta}(D^{\star})$. Thanks to \cite{Amrouche-Houda}, Lemma 5.4, we can rewrite $W_{2}(\cdot,z)$ as: 
\begin{equation}\label{decompositionW4}
W_{2}(x,z) = V(x,z) + \underset{x}{\nabla} \theta(x,z), 
\end{equation}
where 
\begin{equation*}
V(\cdot,z) \in \mathbb{L}^{\frac{3}{2}-\delta}(D^{\star}), \; \underset{x}{\div}(V(x,z)) = 0 \quad \text{and} \quad \theta(\cdot,z) \in \mathbb{W}^{1,\frac{3}{2}-\delta}_{0}(D^{\star}). 
\end{equation*}
Plugging the new expression of $W_{2}(x,z)$, see $(\ref{decompositionW4})$, into $(\ref{044prime})$ to obtain:   
\begin{align}\label{W4=V+gradtheta}
  \begin{cases}
      \left( \underset{x}{\Delta} + \alpha(x) \, I \right) V(x,z)  = \;- g(x,z) - \alpha(x) \; \underset{x}{\nabla} \theta(x,z)  & \text{in \,\, $D^{\star}$} \\
 \qquad  \quad \,    \underset{x}{\div}(V(x,z)) \, \, \, \, =\;  0 & \text{in \,\, $D^{\star}$} \\
 \qquad  \quad \, \qquad  \quad \,\, \theta(x,z) =\;  0 & \text{on \,\, $\partial D^{\star}$}
    \end{cases}.
\end{align}
Since $\theta(\cdot,z) \in \mathbb{W}^{1,\frac{3}{2}-\delta}_{0}(D^{\star})$ and $g(\cdot,z) \in \mathbb{L}^{\frac{3}{2}-\delta}(D^{\star})$, we deduce that the right hand side of $(\ref{W4=V+gradtheta})$ is in $\mathbb{L}^{\frac{3}{2}-\delta}(D^{\star})$. From the interior regularity results for the Helmholtz equation, we deduce that $V(\cdot,z) \in \mathbb{W}^{2,\frac{3}{2}-\delta}(D^{\star})$. In addition, using the embedding of Sobolev spaces,   $
V(\cdot,z) \in \mathbb{W}^{2,\frac{3}{2}-\delta}\left(D^{\star} \right) \lhook\joinrel\xrightarrow{\text{continously}} \mathbb{L}^{\frac{3(3-2\delta)}{4\delta}}\left(D^{\star}\right),$
we deduce that:
\begin{equation}\label{V-Regularity}
V(\cdot,z) \in  \mathbb{L}^{\frac{3\, (3 - 2 \delta)}{4 \, \delta}}(D^{\star}). 
\end{equation}
Next, assuming that $\alpha(\cdot) \neq 0$ in $D^{\star}$, we rewrite $(\ref{W4=V+gradtheta})$ as: 
\begin{equation*}
\underset{x}{\nabla} \theta(x,z) \; = - \alpha(x)^{-1} \, \left( g(x,z) +  \underset{x}{\Delta} V(x,z) \right)- V(x,z) \;\; \text{in}  \;\, D^{\star} \,\; \text{and} \;\;\; \theta(\cdot,z)_{|_{\partial D^{\star}}} =\; 0.
\end{equation*}
Near the center $z$, by Taylor expansion, we get: 
\begin{eqnarray*}
\underset{x}{\nabla} \theta(x,z) &=& - \alpha^{-1}(z) \, \left( g(x,z) +  \underset{x}{\Delta} V(x,z) \right) - V(x,z) \\ &-& \int_{0}^{1} \nabla \alpha^{-1}(z+t(x - z)) \cdot (x - z) dt \, \left( g(x,z)  + \underset{x}{\Delta} V(x,z) \right). 
\end{eqnarray*}
Taking the divergence operator, with respect to $x$, and using the fact that $\underset{x}{\div}(V(x,z)) = 0$, we obtain: 
\begin{eqnarray*}
\underset{x}{\Delta} \theta(x,z) &=& - \alpha^{-1}(z) \, \underset{x}{\div}\left( g(x,z) \right) \\ &-& \underset{x}{\div}\left( \int_{0}^{1} \nabla \alpha^{-1}(z+t(x - z)) \cdot (x - z) dt \, \left( g(x,z)  + \underset{x}{\Delta} V(x,z) \right) \right)\\
\underset{x}{\Delta} \theta(x,z) & \overset{(\ref{FW})}{=} & - \frac{\alpha^{-1}(z)}{\epsilon_{\infty}} \, \left(  \omega^{2} \, \mu \, \epsilon_{\infty} \underset{x}{\nabla} \epsilon_{0}(x) \, \Phi_{k}(x,z) \, + \, \underset{x}{\nabla} \underset{x}{\nabla} (\Phi_{k}(x,z)) \cdot \underset{x}{\nabla} \epsilon_{0}(x) \right) \\
&-& \underset{x}{\div}\left( \int_{0}^{1} \nabla \alpha^{-1}(z+t(x - z)) \cdot (x - z) dt \, \left( g(x,z)  + \underset{x}{\Delta} V(x,z) \right) \right).
\end{eqnarray*}
Again, by Taylor expansion for the function $\nabla \epsilon_{0}(\cdot)$ near the point $z$, we rewrite the previous equation as: 
\begin{eqnarray*}
\underset{x}{\Delta} \theta(x,z) & = & - \frac{\alpha^{-1}(z)}{\epsilon_{\infty}} \, \left(  \omega^{2} \, \mu \, \epsilon_{\infty} \underset{x}{\nabla} \epsilon_{0}(x) \, \Phi_{k}(x,z) \, + \, \underset{x}{\nabla} \underset{x}{\nabla} (\Phi_{k}(x,z)) \cdot \nabla \epsilon_{0}(z) \right) \\
& - & \frac{\alpha^{-1}(z)}{\epsilon_{\infty}} \,   \underset{x}{\nabla} \underset{x}{\nabla} (\Phi_{k}(x,z)) \cdot \int_{0}^{1} \underset{x}{\nabla} \left( \underset{x}{\nabla} \epsilon_{0} \right) (z + t (x - z)) \cdot (x - z) \, dt  \\
&-& \underset{x}{\div}\left( \int_{0}^{1} \nabla \alpha^{-1}(z+t(x - z)) \cdot (x - z) dt \, \left( g(x,z)  + \underset{x}{\Delta} V(x,z) \right) \right).
\end{eqnarray*}
To solve the previous equation, or to get an expression for the dominant term of its solution, we start by split $\theta(x,z)$ as $\theta(x,z) = \underset{j=1}{\overset{3}{\sum}} \theta_{j}(x,z)$, where: 
\begin{align}\label{EquaTheta1}
    \begin{cases}
      - \underset{x}{\Delta} \theta_{1}(x,z)  \, \, = \alpha^{-1}(z) \, \omega^{2} \, \mu  \, \underset{x}{\nabla} \epsilon_{0}(x) \, \Phi_{k}(x,z) & \text{in $D^{\star}$} \\
    \qquad   \theta_{1}(x,z) = 0 & \text{on $\partial D^{\star}$}
    \end{cases},
\end{align}
\begin{align}\label{EquaTheta2}
    \begin{cases}
      - \underset{x}{\Delta} \theta_{2}(x,z) \,\, = \frac{\alpha^{-1}(z)}{\epsilon_{\infty}} \,  \underset{x}{\nabla} \underset{x}{\nabla} (\Phi_{k}(x,z)) \cdot \nabla \epsilon_{0}(z) & \text{in $D^{\star}$} \\
  \qquad    \theta_{2}(x,z) = 0 & \text{on $\partial D^{\star}$}
    \end{cases}
\end{align}
and
\begin{align}\label{MLNL}
    \begin{cases}
    - \underset{x}{\Delta} \theta_{3}(x,z)  =   \frac{\alpha^{-1}(z)}{\epsilon_{\infty}} \,   \underset{x}{\nabla} \underset{x}{\nabla} (\Phi_{k}(x,z)) \cdot \int_{0}^{1} \underset{x}{\nabla} \left( \underset{x}{\nabla} \epsilon_{0} \right) (z + t (x - z)) \cdot (x - z) \, dt  & \text{} \\
 \qquad \qquad \quad    + \,\, \underset{x}{\div}\left( \int_{0}^{1} \nabla \alpha^{-1}(z+t(x - z)) \cdot (x - z) dt \, \left( g(x,z)  + \underset{x}{\Delta} V(x,z) \right) \right)  & \text{in $D^{\star}$} \\
   \quad \,\,   \theta_{3}(x,z) = 0 & \text{on $\partial D^{\star}$}
    \end{cases}.
\end{align} 
From potential theory results we can check that $\theta_{1}(x,z)$, solution of $(\ref{EquaTheta1})$, and $\theta_{2}(x,z)$, solution of $(\ref{EquaTheta2})$, can be represented in the following integral form:  
\begin{align}\label{AZMed}
   \begin{cases}
      \theta_{1}(x,z) & = \, \alpha^{-1}(z) \, \omega^{2} \, \mu  \int_{D^{\star}} G_{\theta}(x,y)  \, \underset{y}{\nabla} \epsilon_{0}(y) \, \Phi_{k}(y,z)\, dy \\
      \theta_{2}(x,z) & = \frac{\alpha^{-1}(z)}{\epsilon_{\infty}} \, \int_{D^{\star}} G_{\theta}(x,y)  \,  \underset{y}{\nabla} \underset{y}{\nabla} (\Phi_{k}(y,z)) \cdot \nabla \epsilon_{0}(z) \, dy
    \end{cases},
\end{align}
where $G_{\theta}(\cdot,\cdot)$ is the kernel solution of: 
\begin{align}\label{PDEGTheta}
   \begin{cases}
      \Delta G_{\theta} \, = - \delta & \text{in $D^{\star}$} \\
     \quad G_{\theta} = 0 & \text{on $\partial D^{\star}$}
    \end{cases}.
\end{align}
The construction of $G_{\theta}(\cdot,\cdot)$, in the case of unit ball $B(0,1)$, is already done in \textbf{Subsection 2.2.4} of the reference \cite{Evans}. Now, without repeating the same proof, we can deduce  that: 
\begin{equation}\label{DefGTheta}
G_{\theta}(x,y) = \Phi_{0}(x,y)+\Phi_{R}(x,y), \quad  x,y \in D^{\star}, \,\, x \neq y, 
\end{equation}
where $\Phi_{0}$ is the fundamental solution of Laplace equation in the entire space  given by $\Phi_{0}(x,y) := \frac{1}{4 \, \pi \, \left\vert x - y \right\vert}$ and $\Phi_{R}$ is the remainder part given by\footnote{The construction idea of $\Phi_{R}(\cdot,\cdot)$ consist in \emph{inverting the singularity} from inside of $D^{\star}$ to outside of $D^{\star}$. This justify its regularity.}
\begin{equation}\label{CHT}
\Phi_{R}(x,y) := \frac{-1}{4 \, \pi} \,  \frac{a^{\star} \, \left\vert x - z \right\vert}{\left\vert \left\vert x - z \right\vert^{2} \, \left( y - z \right) - (a^{\star})^{2} \,  (x-z)  \right\vert} \in \mathbb{W}^{3,\infty}\left( D \times D \right).
\end{equation}
Next, we apply  $(\ref{DefGTheta})$ into $(\ref{AZMed})$ to get a dominant term and a remainder term for $\theta_{1,2}(x,z)$. 
\begin{enumerate}
\item Computation of $\theta_{1}(x,z)$.\\
From $(\ref{AZMed})$, we have: 
\begin{eqnarray*}
\theta_{1}(x,z) & = &  \alpha^{-1}(z) \, \omega^{2} \, \mu  \int_{D^{\star}} G_{\theta}(x,y)  \, \underset{y}{\nabla} \epsilon_{0}(y) \, \Phi_{k}(y,z)\, dy \\
& \overset{(\ref{DefGTheta})}{=} & \alpha^{-1}(z) \, \omega^{2} \, \mu  \int_{D^{\star}} \Phi_{0}(x,y)  \, \underset{y}{\nabla} \epsilon_{0}(y) \, \Phi_{k}(y,z)\, dy + \alpha^{-1}(z) \, \omega^{2} \, \mu  \int_{D^{\star}} \Phi_{R}(x,y) \, \underset{y}{\nabla} \epsilon_{0}(y) \, \Phi_{k}(y,z)\, dy \\
& \overset{(\ref{DefNDefM})}{=} &\, \omega^{2} \, \mu \alpha^{-1}(z) \, \, N\left( \nabla \epsilon_{0}(\cdot) \, \Phi_{k}(\cdot,z) \right)(x)  + \alpha^{-1}(z) \, \omega^{2} \, \mu  \int_{D^{\star}} \Phi_{R}(x,y) \, \underset{y}{\nabla} \epsilon_{0}(y) \, \Phi_{k}(y,z)\, dy.
\end{eqnarray*}
Set 
\begin{equation*}
\theta_{1,R}(x,z) := \alpha^{-1}(z) \, \omega^{2} \, \mu  \int_{D^{\star}} \Phi_{R}(x,y) \, \underset{y}{\nabla} \epsilon_{0}(y) \, \Phi_{k}(y,z)\, dy.
\end{equation*}
Then,  
\begin{eqnarray*}
\underset{x\in D^{\star}}{Sup}  \left\vert \underset{x}{\nabla} \theta_{1,R}(x,z) \right\vert &=& \underset{x\in D^{\star}}{Sup} \left\vert  \alpha^{-1}(z) \, \omega^{2} \, \mu  \int_{D^{\star}} \underset{x}{\nabla} \Phi_{R}(x,y) \otimes \underset{y}{\nabla} \epsilon_{0}(y) \, \Phi_{k}(y,z)\, dy \right\vert \\
& \lesssim & \left\Vert \, \nabla \epsilon_{0}(\cdot)  \right\Vert_{\mathbb{L}^{\infty}(D^{\star})} \,     \underset{x \in D^{\star}}{Sup} \, \underset{y \in D^{\star}}{Sup} \left\vert  \underset{x}{\nabla} \Phi_{R}(x,y) \right\vert \, \int_{D^{\star}} \, \frac{1}{\left\vert y-z \right\vert} dy. 
\end{eqnarray*}
Assuming that $\left\Vert \, \nabla \epsilon_{0}(\cdot)  \right\Vert_{\mathbb{L}^{\infty}(D^{\star})} = \mathcal{O}(1)$ and computing the last integral explicitly to end up with: 
\begin{equation}\label{RemaindergradTheta1}
\left\Vert \nabla \theta_{1,R}(\cdot,z) \right\Vert_{\mathbb{L}^{\infty}(D^{\star})} \lesssim \left( a^{\star} \right)^{2} \, \left\Vert \nabla \Phi_{R}(\cdot,\cdot) \right\Vert_{\mathbb{L}^{\infty}(D^{\star} \times D^{\star})} \overset{(\ref{CHT})}{<} + \infty. 
\end{equation}
\item Computation of $\theta_{2}(x,z)$.\\
From $(\ref{AZMed})$, we have: 
\begin{equation*}
      \theta_{2}(x,z) =  \frac{\alpha^{-1}(z)}{\epsilon_{\infty}} \, \int_{D^{\star}} G_{\theta}(x,y)  \,  \underset{y}{\nabla} \underset{y}{\nabla} (\Phi_{k}(y,z)) \cdot \nabla \epsilon_{0}(z) \, dy, \end{equation*}
which, after two times integration by parts and using the fact that ${G_{\theta}}_{|_{\partial D^{\star}}} = 0$, see $(\ref{PDEGTheta})$, becomes       
 \begin{eqnarray*}
 \theta_{2}(x,z) &=&  \frac{-\alpha^{-1}(z)}{\epsilon_{\infty}} \, \underset{x}{\nabla} \, \int_{D^{\star}} \underset{y}{\nabla} G_{\theta}(x,y)   \cdot \left( \Phi_{k}(y,z)  \nabla \epsilon_{0}(z) \right) dy \\
 \theta_{2}(x,z) & \overset{(\ref{DefGTheta})}{=} & \frac{-\alpha^{-1}(z)}{\epsilon_{\infty}} \, \underset{x}{\nabla} \, \int_{D^{\star}} \underset{y}{\nabla} \Phi_{0}(x,y)   \cdot \left( \Phi_{k}(y,z)  \nabla \epsilon_{0}(z) \right) dy \\ &+& \frac{-\alpha^{-1}(z)}{\epsilon_{\infty}} \, \underset{x}{\nabla} \, \int_{D^{\star}} \underset{y}{\nabla} \Phi_{R}(x,y)    \cdot \left( \Phi_{k}(y,z)  \nabla \epsilon_{0}(z) \right) dy  \\ 
& \overset{(\ref{DefNDefM})}{=} & - \frac{\alpha^{-1}(z)}{\epsilon_{\infty}} \,  \underset{x}{\nabla} M\left(\Phi_{k}(\cdot,z)  \nabla \epsilon_{0}(z) \right)(x) - \frac{\alpha^{-1}(z)}{\epsilon_{\infty}} \, \underset{x}{\nabla} \, \int_{D^{\star}} \underset{y}{\nabla} \Phi_{R}(x,y)    \cdot \left( \Phi_{k}(y,z)  \nabla \epsilon_{0}(z) \right) dy .
\end{eqnarray*}
Set
\begin{eqnarray*}
 \theta_{2,R}(x,z) &:=& - \frac{\alpha^{-1}(z)}{\epsilon_{\infty}} \, \underset{x}{\nabla} \, \int_{D^{\star}} \underset{y}{\nabla} \Phi_{R}(x,y)    \cdot \left( \Phi_{k}(y,z)  \nabla \epsilon_{0}(z) \right) dy \\
 &=& - \frac{\alpha^{-1}(z)}{\epsilon_{\infty}}  \int_{D^{\star}} \, \underset{x}{\nabla} \,\underset{y}{\nabla} \Phi_{R}(x,y) \cdot \nabla \epsilon_{0}(z)  \Phi_{k}(y,z) \, dy. 
\end{eqnarray*}
Then, 
\begin{eqnarray*}
\underset{x \in D^{\star}}{Sup} \left\vert \underset{x}{\nabla} \theta_{2,R}(x,z) \right\vert & \lesssim & \underset{x \in D^{\star}}{Sup} \, \int_{D^{\star}} \underset{y}{Sup} \left\vert \underset{x}{\nabla} \underset{x}{\nabla} \,\underset{y}{\nabla} \Phi_{R}(x,y) \right\vert \, \frac{1}{\left\vert y - z \right\vert}  dy \\
& = & \underset{x \in D^{\star}}{Sup} \,\underset{y \in D^{\star}}{Sup} \left\vert \underset{x}{\nabla} \underset{x}{\nabla} \,\underset{y}{\nabla} \Phi_{R}(x,y) \right\vert \, \int_{D^{\star}}  \frac{1}{\left\vert y - z \right\vert}  dy.
\end{eqnarray*}
Then, 
\begin{equation}\label{RemaindergradTheta2}
\left\Vert \nabla \theta_{2,R}(\cdot ,z)  \right\Vert_{\mathbb{L}^{\infty}(D^{\star})} \lesssim \left( a^{\star} \right)^{2} \, \left\Vert \nabla \nabla \Phi_{R}(\cdot,\cdot) \right\Vert_{\mathbb{L}^{\infty}(D^{\star} \times D^{\star})} \overset{(\ref{CHT})}{<} + \infty. 
\end{equation}
\end{enumerate}
Summarizing all this to obtain:
\begin{align}\label{Theta1andTheta2}
   \begin{cases}
       \theta_{1}(x,z)  =  \, \omega^{2} \, \mu \alpha^{-1}(z) \, \, N\left( \nabla \epsilon_{0}(\cdot) \, \Phi_{k}(\cdot,z) \right)(x) + \theta_{1,R}(x,z) &  \\
     \theta_{2}(x,z)  = - \frac{\alpha^{-1}(z)}{\epsilon_{\infty}} \,  \underset{x}{\nabla} M\left(\Phi_{k}(\cdot,z)  \nabla \epsilon_{0}(z) \right)(x) + \theta_{2,R}(x,z) & 
    \end{cases}.
\end{align}
The computation of an explicit formula for the dominant term of  $\theta_{3}(x,z)$, satisfying $(\ref{MLNL})$, is really difficult task. But, in singularity analysis point of view, the dominant part of the right hand side of $(\ref{MLNL})$ behaves as $\sim \Phi_{0}^{2}(\cdot,z) \in \mathbb{L}^{\frac{3}{2}-\delta}(D^{\star})$. Then, as justified previously, we can prove that:
\begin{equation*}
\theta_{3}(\cdot,z) \in \mathbb{W}^{2,\frac{3}{2}-\delta}(D^{\star}),
\end{equation*} 
hence, 
\begin{equation}\label{gradtheta4-Regularity}
\nabla \theta_{3}(\cdot,z) \in \mathbb{W}^{1,\frac{3}{2}-\delta}\left( D^{\star} \right) \lhook\joinrel\xrightarrow{\text{continously}} \mathbb{L}^{\frac{3(3-2\delta)}{(3+2\delta)}}\left( D^{\star} \right).
\end{equation}
Summarizing all this results to obtain: 
\begin{eqnarray*}
\nonumber
\underset{x}{\nabla} \theta(x,z) &=& \underset{x}{\nabla} \theta_{1}(x,z) + \underset{x}{\nabla} \theta_{2}(x,z) + \underset{x}{\nabla} \theta_{3}(x,z) \\ \nonumber &=& \frac{1}{\alpha(z) \, \epsilon_{\infty}} \, \underset{x}{\nabla} \left[   \omega^{2} \, \mu \, \epsilon_{\infty}  \, N\left( \nabla \epsilon_{0}(\cdot) \, \Phi_{k}(\cdot,z) \right)(x) -  \underset{x}{\nabla} M\left(\Phi_{k}(\cdot,z)  \nabla \epsilon_{0}(z) \right)(x) \right] \\ &+& \underset{x}{\nabla} \theta_{1,R}(x,z)+ \underset{x}{\nabla} \theta_{2,R}(x,z)+ \underset{x}{\nabla} \theta_{3}(x,z).
\end{eqnarray*}  
Finally, $W_{2}(x,z)$ defined by $(\ref{decompositionW4})$, becomes near the ball $D^{\star}$,
\begin{eqnarray}\label{AS83}
\nonumber
W_{2}(x,z) &=& \frac{1}{\alpha(z) \, \epsilon_{\infty}} \, \, \underset{x}{\nabla} \left[   \omega^{2} \, \mu \, \epsilon_{\infty}  \, N\left( \nabla \epsilon_{0}(\cdot) \, \Phi_{k}(\cdot,z) \right)(x) -  \underset{x}{\nabla} M\left(\Phi_{k}(\cdot,z)  \nabla \epsilon_{0}(z) \right)(x) \right] \\ &+& \underset{x}{\nabla} \theta_{1,R}(x,z)+ \underset{x}{\nabla} \theta_{2,R}(x,z) + \underset{x}{\nabla} \theta_{3}(x,z)+ V(x,z). 
\end{eqnarray}
Thanks to $(\ref{RemaindergradTheta1}),(\ref{RemaindergradTheta2}), (\ref{V-Regularity})$ and $(\ref{gradtheta4-Regularity})$, the regularity of the remainder term $\nabla \theta_{1,R}(\cdot ,z)+ \nabla \theta_{2,R}(\cdot,z) + V(\cdot,z) + \nabla \theta_{3}(\cdot,z)$ in $(\ref{AS83})$ will be: 
\begin{equation}\label{V+gradTheta}
\nabla \theta_{1,R}(\cdot ,z)+ \nabla \theta_{2,R}(\cdot,z) + V(\cdot,z) + \nabla \theta_{3}(\cdot,z) \in  \mathbb{L}^{\frac{3(3-2\delta)}{(3+2\delta)}}\left( D^{\star} \right).
\end{equation}
In addition, remark that from the definition of $\theta_{1}(\cdot,z)$, see $(\ref{Theta1andTheta2})$, the source data $\nabla \epsilon_{0}(\cdot) \, \Phi_{k}(\cdot,z)$ is in $\mathbb{L}^{3 -\delta}\left( D^{\star} \right)$  and from the regularity of the operator $\nabla N\left( \cdot \right)$, we deduce that: 
\begin{equation}\label{gradNPhi}
\nabla \, N\left( \nabla \epsilon_{0}(\cdot) \, \Phi_{k}(\cdot,z) \right) \in \mathbb{W}^{1,3-\delta}\left( D^{\star} \right) \subset \mathbb{L}^{3-\delta}\left( D^{\star} \right) \subset \mathbb{L}^{\frac{3(3-2\delta)}{(3+2\delta)}}\left( D^{\star} \right).
\end{equation} 
Also, we have: 
\begin{equation}\label{DiffgradgradM}
\nabla \nabla M\left(\Phi_{k}(\cdot,z)  \nabla \epsilon_{0}(z) \right) = \nabla \nabla M\left(\Phi_{0}(\cdot,z)  \nabla \epsilon_{0}(z) \right) - \nabla \nabla \div N\left(\left( \Phi_{k} - \Phi_{0} \right)(\cdot,z) \nabla \epsilon_{0}(z) \right),
\end{equation}
and because that $\left( \Phi_{k} - \Phi_{0} \right)(\cdot,z) \nabla \epsilon_{0}(z) \in \mathbb{W}^{1,\infty}(D^{\star})$ we deduce that $N\left( \left( \Phi_{k} - \Phi_{0} \right)(\cdot,z) \nabla \epsilon_{0}(z)\right) \in \mathbb{W}^{3,\infty}(D^{\star})$ and, consequently, then:
\begin{equation}\label{LZNS}
 \nabla \nabla \div N\left(\left( \Phi_{k} - \Phi_{0} \right)(\cdot,z) \nabla \epsilon_{0}(z) \right) \in \mathbb{L}^{\infty}( D^{\star} ). 
\end{equation}
Now, we set $W_{3}$ to be:
\begin{eqnarray}\label{W3ExpressionRegularity}
\nonumber
W_{3} &:=& \bigg[ \frac{1}{\alpha(z) \, \epsilon_{\infty}} \nabla \nabla \div N\left(\left( \Phi_{k} - \Phi_{0} \right) \nabla \epsilon_{0}(z) \right) +  \frac{\omega^{2} \, \mu}{\alpha(z)} \,   \nabla  \, N\left( \nabla \epsilon_{0}(\cdot) \, \Phi_{k} \right)  \\ &+& \nabla \theta_{1,R} + \nabla \theta_{2,R} + V + \nabla \theta_{3} \bigg] \in \mathbb{L}^{\frac{3(3-2\delta)}{(3+2\delta)}}\left( D^{\star} \right),
\end{eqnarray} 
where its regularity is a straightforward consequence of $(\ref{V+gradTheta}), (\ref{gradNPhi})$ and $(\ref{LZNS})$. Finally, thanks to $(\ref{W3ExpressionRegularity})$, the expression of $W_{2}$ given by $(\ref{AS83})$ becomes: 
\begin{equation*}\label{AS83New}
W_{2}(x,z) = \frac{-1}{\alpha(z) \, \epsilon_{\infty}} \, \underset{x}{\nabla}   \underset{x}{\nabla} M\left(\Phi_{0}(\cdot,z)  \nabla \epsilon_{0}(z) \right)(x)  + W_{3}(x,z). 
\end{equation*}
This was to be demonstrated. 
\end{proof}

For the equation $(\ref{NLCS})$, the last term to analyze is $\Gamma^{\delta}(\cdot,z)$, solution of $(\ref{NLCSGammadelta})$, given by: 
\begin{equation*}
      \left( Curl \circ Curl - \alpha(\cdot) \, I \right) \Gamma^{\delta}(\cdot,z) =  \frac{\left( \epsilon_{\infty} - \epsilon_{0} (z) \right)}{\epsilon_{\infty}}  \underset{z}{\delta}(\cdot) \, I 
\end{equation*}
with the radiation conditions at infinity.
In straightforward manner, regarding only the equation satisfied by $G_{k}(\cdot,\cdot)$, see for instance $(\ref{equa1})$, we deduce that $\Gamma^{\delta}(\cdot,z)$ can be written as: 
\begin{equation}\label{Gammadelta}
\Gamma^{\delta}(\cdot,z) = \frac{\left( \epsilon_{\infty} - \epsilon_{0}(z) \right)}{\epsilon_{\infty}} G_{k}(\cdot,z).
\end{equation}
Gathering all this to get a compact expression for the function $\Gamma(\cdot,z)$. More precisely, we have: 
\begin{equation*}
\Gamma(\cdot,z) = W_{1}(\cdot,z) + W_{2}(\cdot,z) + \Gamma^{\delta}(\cdot,z), 
\end{equation*}
and thanks to Lemma $(\ref{DjHCV19})$, Lemma  $(\ref{ExpressionW2})$ and the expression of $\Gamma^{\delta}$ given by $(\ref{Gammadelta})$ we get: 
\begin{equation}\label{Gamma=formula+Vtilde}
\Gamma(x,z) = \frac{-1}{\alpha(z) \, \epsilon_{\infty}} \, \underset{x}{\nabla}   \underset{x}{\nabla} M\left(\Phi_{0}(\cdot,z)  \nabla \epsilon_{0}(z) \right)(x) + \frac{\left( \epsilon_{\infty} - \epsilon_{0}(z) \right)}{\epsilon_{\infty}} G_{k}(x,z) + W_{1}(x,z) + W_{3}(x,z),
\end{equation}  
where 
\begin{equation}\label{SYW1W3}
\left( W_{1} + W_{3}  \right)(\cdot,z) \in \mathbb{L}^{\frac{3(3-2\delta)}{(3+2\delta)}}(D).
\end{equation} 
In conclusion, the Green kernel $G_{k}(\cdot, \cdot)$, constructed as:  
\begin{equation*}
G_{k}(\cdot,z) = \Upsilon_{k}(\cdot,z) + \Gamma(\cdot,z),
\end{equation*}
takes, regarding $(\ref{Gamma=formula+Vtilde})$, the following form: 
\begin{eqnarray*}
G_{k}(\cdot,z) &=& \Upsilon_{k}(\cdot,z) - \frac{1}{\alpha(z) \, \epsilon_{\infty}} \, \nabla \nabla M\left(\Phi_{0}(\cdot,z)  \nabla \epsilon_{0}(z) \right)(\cdot) + \frac{\left( \epsilon_{\infty} - \epsilon_{0}(z) \right)}{\epsilon_{\infty}} G_{k}(\cdot,z) + W_{1}(\cdot,z) + W_{3}(\cdot,z) \\
G_{k}(\cdot,z) &=&\frac{\epsilon_{\infty}}{\epsilon_{0}(z)} \left[ \Upsilon_{k}(\cdot,z) - \frac{1}{\alpha(z) \, \epsilon_{\infty}} \, \nabla \nabla M\left(\Phi_{0}(\cdot,z)  \nabla \epsilon_{0}(z) \right)(\cdot) + W_{1}(\cdot,z) + W_{3}(\cdot,z) \right] \\
G_{k}(\cdot,z) & \overset{(\ref{equa3})}{=} & \frac{\epsilon_{\infty}}{\epsilon_{0}(z)} \, \Phi_{k}(\cdot,z) \, I + \frac{1}{\omega^{2} \, \mu \, \epsilon_{0}(z)} \; \nabla \, \nabla \left( \Phi_{k}\right)(\cdot,z)  - \frac{1}{\alpha(z) \, \epsilon_{0}(z)} \, \nabla \nabla M\left(\Phi_{0}(\cdot,z)  \nabla \epsilon_{0}(z) \right)(\cdot) \\ 
&+& \frac{\epsilon_{\infty}}{\epsilon_{0}(z)} \left[ W_{1}(\cdot,z) + W_{3}(\cdot,z) \right].
\end{eqnarray*}
Moreover, for the first term on the right hand side, we have: 
\begin{equation}\label{SY}
\Phi_{k}(\cdot,z) \, I \in \mathbb{L}^{3-\delta}\left(D\right) \subset \mathbb{L}^{\frac{3(3-2\delta)}{(3+2\delta)}}\left(D\right). 
\end{equation}
Also, in the same manner as $(\ref{DiffgradgradM})$, we have: 
\begin{equation*}
\nabla \, \nabla \left( \Phi_{k}\right) = \nabla \, \nabla \left( \Phi_{0}\right) + \nabla \, \nabla \left( \Phi_{k} - \Phi_{0} \right),
\end{equation*}
where, the difference term, 
\begin{equation}\label{SY1}
\nabla \, \nabla \left( \Phi_{k} - \Phi_{0} \right)(\cdot,z) \simeq \nabla \, \nabla \left( \Phi^{-1}_{0} \right)(\cdot,z) \in \mathbb{L}^{3-\delta}(D) \subset  \mathbb{L}^{\frac{3(3-2\delta)}{(3+2\delta)}}\left(D\right).
\end{equation}
Then, by setting $W_{4}(\cdot,z)$ to be: 
\begin{equation*}
W_{4}(\cdot,z) := \frac{\epsilon_{\infty}}{\epsilon_{0}(z)} \, \Phi_{k}(\cdot,z) \, I + \frac{1}{\omega^{2} \, \mu \, \epsilon_{0}(z)} \; \nabla \, \nabla \left( \Phi_{k} - \Phi_{0} \right)(\cdot,z)  +  \frac{\epsilon_{\infty}}{\epsilon_{0}(z)} \left[ W_{1}(\cdot,z) + W_{3}(\cdot,z) \right],
\end{equation*}
 we end up with:  
\begin{equation}\label{ExpansionofGkII}
G_{k}(\cdot,z)  =  \frac{1}{\omega^{2} \, \mu \, \epsilon_{0}(z) } \; \nabla \, \nabla \left( \Phi_{0}\right)(\cdot,z)  - \frac{1}{\alpha(z) \, \epsilon_{0}(z)} \, \nabla \nabla M\left(\Phi_{0}(\cdot,z)  \nabla \epsilon_{0}(z) \right)(\cdot) + W_{4}(\cdot,z),
\end{equation}
where, from $(\ref{SY}),(\ref{SY1})$ and $(\ref{SYW1W3})$, we have $W_{4}(\cdot,z) \in \mathbb{L}^{\frac{3(3-2\delta)}{(3+2\delta)}}\left(D\right)$. This ends the proof of  \textbf{Theorem $\ref{DH}$}. 
\subsection{Justification of the representation $(\ref{SK0})$} \label{Subsection-Distribution}
The goal of this subsection is to give sense of the Lippmann-Schwinger equation, For this, we recall that:
\begin{equation}\label{FLKT}
u_{1}(x) + \omega^{2}  \, \int_{D} G_{k}(x,y) \cdot u_{1}(y) \, (\bm{n}^{2}_{0}(y)-\bm{n}^{2}(y)) \, dy = u_{0}(x), \quad x \in \mathbb{R}^{3}, 
\end{equation}
where the kernel $G_{k}(\cdot,\cdot)$ is solution of 
\begin{equation*}
\underset{y}{\nabla} \times \underset{y}{\nabla} \times G_{k}(x,y) -  \omega^{2} \, \bm{n}^{2}_{0}(y) \, G_{k}(x,y) = \underset{x}{\delta}(y) I,  
\end{equation*}
satisfying the Silver-M\"{u}ller radiation condition at infinity: 
\begin{equation*}
\lim_{\left\vert x \right\vert \rightarrow +\infty} \;\; \left\vert x \right\vert \;\; \left( \underset{y}{\nabla} \times G_{k}(x,y) \times \frac{x}{\left\vert x \right\vert} - i \, k \, G_{k}(x,y) \right) = 0.
\end{equation*}
Now, for fixed $z$, we split $G_{k}(\cdot,z)$ as: 
\begin{equation}\label{ASHESP}
G_{k}(\cdot,z) = \Upsilon_{k}(\cdot,z) + \Gamma(\cdot,z),
\end{equation} 
where $\Upsilon_{k}(\cdot,z)$ is solution of 
\begin{equation*}
\underset{y}{\nabla} \times \underset{y}{\nabla} \times \Upsilon_{k}(y,z) -  k^{2} \, \Upsilon_{k}(y,z) = \underset{z}{\delta}(y) I,
\end{equation*}
satisfying the Silver-M\"{u}ller radiation condition at infinity: 
\begin{equation*}
\lim_{\left\vert x \right\vert \rightarrow +\infty} \;\; \left\vert x \right\vert \;\; \left( \underset{y}{\nabla} \times \Upsilon_{k}(x,z) \times \frac{x}{\left\vert x \right\vert} - i \, k \, \Upsilon_{k}(x,z) \right) = 0.
\end{equation*}
It is known from the literature that 
\begin{equation}\label{CC0}
\Upsilon_{k}(y,z) = \Phi_{k}(y,z) \, I + \frac{1}{k^{2}} \, \underset{y}{\nabla}  \underset{y}{\nabla} \cdot (\Phi_{k}(y,z) \, I),
\end{equation} 
or, using the fact that $\nabla \times \nabla \times () + \Delta () = \nabla \nabla \cdot ()$ and $\Phi_{k}$ is the fundamental solution for the Helmholtz equation, in the following form
\begin{equation}\label{CC}
\Upsilon_{k}(y,z)  = \frac{1}{k^{2}} \, \underset{y}{\nabla} \times \underset{y}{\nabla} \times (\Phi_{k}(y,z) \, I) - \frac{1}{k^{2}} \, \underset{z}{\delta}(y) \, I.
\end{equation}
By construction $\Gamma(\cdot,z)$, needless to say that satisfy the Silver-M\"{u}ller radiation condition at infinity, will be solution of
\begin{eqnarray}\label{ADZKB}
\nonumber
\underset{y}{\nabla} \times \underset{y}{\nabla} \times \Gamma(y,z) - \omega^{2} \, \, \bm{n}^{2}_{0}(y) \, \Gamma(y,z) &=& \left(\omega^{2} \, \, \bm{n}^{2}_{0}(y) - k^{2} \right) \, \Upsilon_{k}(y,z) \\
&\overset{(\ref{CC})}{=}& \frac{\left(\omega^{2} \, \, \bm{n}^{2}_{0}(y) - k^{2} \right)}{k^{2}} \,\left[ \underset{y}{\nabla} \times \underset{y}{\nabla} \times (\Phi_{k}(y,z) \, I) -   \underset{z}{\delta}(y) \, I \right].
\end{eqnarray} 
This suggest us to split $\Gamma(\cdot,z)$ into two parts $\Gamma(\cdot,z) = \Gamma_{1}(\cdot,z) + \Gamma^{\delta}(\cdot,z)$, where: 
\begin{align}\label{KBFRK}
   \begin{cases}
      \underset{y}{\nabla} \times \underset{y}{\nabla} \times \Gamma_{1}(y,z) - \omega^{2} \, \, \bm{n}^{2}_{0}(y) \, \Gamma_{1}(y,z) & = \frac{\left(\omega^{2} \, \, \bm{n}^{2}_{0}(y) - k^{2} \right)}{k^{2}} \,\underset{y}{\nabla} \times \underset{y}{\nabla} \times (\Phi_{k}(y,z) \, I) \text{\quad for $ y \neq z$} \\
      \underset{y}{\nabla} \times \underset{y}{\nabla} \times \Gamma^{\delta}(y,z) - \omega^{2} \, \, \bm{n}^{2}_{0}(y) \, \Gamma^{\delta}(y,z) & =- \frac{\left(\omega^{2} \, \, \bm{n}^{2}_{0}(y) - k^{2} \right)}{k^{2}} \, \underset{z}{\delta}(y) \, I=- \frac{\left(\omega^{2} \, \, \bm{n}^{2}_{0}(z) - k^{2} \right)}{k^{2}} \, \underset{z}{\delta}(y) \, I
    \end{cases} 
    . 
\end{align}
As done in the previous section, we can prove that $\Gamma_{1}(\cdot,z)$ is in $\mathbb{L}^{p}(\Omega)$, with $p = \frac{3}{2} - \delta$. This justify the existence, at least in distributional sense, of integrals containing the kernel $\Gamma_{1}(\cdot,\cdot)$. Moreover, straightforward calculation allow us to deduce that: 
\begin{equation*}
\Gamma^{\delta}(y,z) = - \frac{\left(\omega^{2} \, \, \bm{n}^{2}_{0}(z) - k^{2} \right)}{k^{2}} \, G_{k}(y,z).
\end{equation*}
Using this representation we rewrite $(\ref{ASHESP})$ as:  
\begin{equation*}
G_{k}(\cdot,z) = \frac{k^{2}}{\omega^{2} \, \bm{n}^{2}_{0}(z) } \, \left[ \Upsilon_{k}(\cdot ,z) + \Gamma_{1}(\cdot ,z) \right].
\end{equation*}
Consequently, $(\ref{FLKT})$ becomes,   
\begin{equation*}
u_{1}(x) + k^{2}  \, \int_{D} \Upsilon_{k}(x ,y) \cdot u_{1}(y) \, \left[ 1-\frac{\bm{n}^{2}(y)}{\bm{n}^{2}_{0}(y)} \right] \, dy + k^{2}  \, \int_{D}  \Gamma_{1}(x,y)  \cdot u_{1}(y) \, \left[ 1-\frac{\bm{n}^{2}(y)}{\bm{n}^{2}_{0}(y)} \right] \, dy = u_{0}(x).
\end{equation*}
Thanks to $(\ref{CC0})$, the definition of the operators $N^{k}(\cdot)$ and $\nabla M^{k}()$, see $(\ref{DefNDefMk})$, we rewrite the previous equation as: 
\begin{eqnarray}\label{DefU1}
\nonumber
u_{1}(x) &+& k^{2}  \, N^{k}\left( u_{1}(\cdot) \, \left[ 1-\frac{\bm{n}^{2}(\cdot)}{\bm{n}^{2}_{0}(\cdot)} \right] \right)(x) - \nabla M^{k}\left( u_{1}(\cdot) \, \left[ 1-\frac{\bm{n}^{2}(\cdot)}{\bm{n}^{2}_{0}(\cdot)} \right] \right)(x) \\ &+& k^{2}  \, \int_{D}  \Gamma_{1}(x,y)  \cdot u_{1}(y) \, \left[ 1-\frac{\bm{n}^{2}(y)}{\bm{n}^{2}_{0}(y)} \right] \, dy = u_{0}(x).
\end{eqnarray}
Now, we check that $u_{1}(\cdot)$, defined by $(\ref{DefU1})$, is solution in the distributional sense of the Maxwell's system 
\begin{equation*}
\nabla \times \nabla \times u_{1}(x) - \omega^{2} \, \bm{n}^{2}(x) \, u_{1}(x) = 0, \qquad x \in \mathbb{R}^{3}.  
\end{equation*}  
For this, taking the $\nabla \times \nabla \times (\cdot)$ on both sides of $(\ref{DefU1})$ and the inner product with respect to test function  $\phi \in \mathcal{D}\left( \mathbb{R}^{3}\right)$, to obtain:  
\begin{eqnarray}\label{DefU1CurlCurl}
\nonumber
\langle \nabla \times \nabla \times u_{1}; \phi \rangle &+& k^{2}  \, \langle \nabla \times \nabla \times N^{k}\left( u_{1}(\cdot) \, \left[ 1-\frac{\bm{n}^{2}(\cdot)}{\bm{n}^{2}_{0}(\cdot)} \right] \right); \phi \rangle  \\ &+& k^{2} \langle \nabla \times \nabla \times \int_{D}  \Gamma_{1}(\cdot ,y)  \cdot u_{1}(y) \, \left[ 1-\frac{\bm{n}^{2}(y)}{\bm{n}^{2}_{0}(y)} \right] \, dy; \phi \rangle = \langle \nabla \times \nabla \times u_{0} ; \phi \rangle.
\end{eqnarray}
Here, we need to analyze the two last terms on the left hand side of the previous equation. 
\begin{enumerate}
\item Analyzing the term: 
\begin{eqnarray}\label{TermJ1}
\nonumber
\bm{J_{1}} &:=& \langle \nabla \times \nabla \times N^{k}\left( u_{1}(\cdot) \, \left[ 1-\frac{\bm{n}^{2}(\cdot)}{\bm{n}^{2}_{0}(\cdot)} \right] \right); \phi \rangle \\ \nonumber &=& \langle  \left( - \Delta + \nabla \div \right) N^{k}\left( u_{1}(\cdot) \, \left[ 1-\frac{\bm{n}^{2}(\cdot)}{\bm{n}^{2}_{0}(\cdot)} \right] \right); \phi \rangle \\\nonumber
&=& k^{2} \langle  N^{k}\left( u_{1}(\cdot) \, \left[ 1-\frac{\bm{n}^{2}(\cdot)}{\bm{n}^{2}_{0}(\cdot)} \right] \right); \phi \rangle + \langle   u_{1}; \phi \rangle \\\nonumber
&-& \langle   u_{1}(\cdot) \, \frac{\bm{n}^{2}(\cdot)}{\bm{n}^{2}_{0}(\cdot)} ; \phi \rangle - \langle  \nabla M^{k}\left( u_{1}(\cdot) \, \left[ 1-\frac{\bm{n}^{2}(\cdot)}{\bm{n}^{2}_{0}(\cdot)} \right] \right); \phi \rangle  \\
&\overset{(\ref{DefU1})}{=}&  \langle   u_{0}; \phi \rangle -  \langle   u_{1}(\cdot) \, \frac{\bm{n}^{2}(\cdot)}{\bm{n}^{2}_{0}(\cdot)} ; \phi \rangle - k^{2}  \langle  \int_{D} \Gamma_{1}(\cdot,y) \cdot u_{1}(y) \, \left[ 1-\frac{\bm{n}^{2}(y)}{\bm{n}^{2}_{0}(y)} \right] \, dy; \phi \rangle. 
\end{eqnarray}
\item Analyzing the term:
\begin{eqnarray*}
\bf{J_{2}} &:=& \langle \nabla \times \nabla \times \int_{D}  \Gamma_{1}(\cdot ,y)  \cdot u_{1}(y) \, \left[ 1-\frac{\bm{n}^{2}(y)}{\bm{n}^{2}_{0}(y)} \right] \, dy; \phi \rangle \\
&=& \langle  \int_{D}  \Gamma_{1}(\cdot ,y)  \cdot u_{1}(y) \, \left[ 1-\frac{\bm{n}^{2}(y)}{\bm{n}^{2}_{0}(y)} \right] \, dy; \nabla \times \nabla \times \phi \rangle\\
&:=& \int_{\mathbb{R}^{3}} \, \int_{D}  \Gamma_{1}(x,y)  \cdot u_{1}(y) \, \left[ 1-\frac{\bm{n}^{2}(y)}{\bm{n}^{2}_{0}(y)} \right] \, dy \cdot \nabla \times \nabla \times \overline{\phi}(x) \, dx \\
&=& \int_{D}  u_{1}(y) \, \left[ 1-\frac{\bm{n}^{2}(y)}{\bm{n}^{2}_{0}(y)} \right]  \cdot \int_{\mathbb{R}^{3}}  \underset{x}{\nabla} \times \underset{x}{\nabla} \times \Gamma_{1}(x,y)  \cdot \overline{\phi}(x) \, dx\, dy. 
\end{eqnarray*}
Gathering $(\ref{ADZKB})$ with $(\ref{KBFRK})$ to 
obtain: 
\begin{eqnarray*}
\bf{J_{2}} &=& \omega^{2} \, \int_{D}  u_{1}(y) \, \left[ 1-\frac{\bm{n}^{2}(y)}{\bm{n}^{2}_{0}(y)} \right]  \cdot \int_{\mathbb{R}^{3}} \Gamma_{1}(x,y) \cdot \bm{n}^{2}_{0}(x) \, \overline{\phi}(x) \, dx \, dy \\
&+& \omega^{2} \, \int_{D}  u_{1}(y) \, \left[ 1-\frac{\bm{n}^{2}(y)}{\bm{n}^{2}_{0}(y)} \right]  \cdot \left[ N^{k}\left( \bm{n}^{2}_{0}(\cdot)   \overline{\phi} \right)(y) - \frac{1}{k^{2}} \nabla M^{k}\left( \bm{n}^{2}_{0}(\cdot)   \overline{\phi} \right)(y) \right] dy \\
&-& k^{2} \int_{D}  u_{1}(y) \, \left[ 1-\frac{\bm{n}^{2}(y)}{\bm{n}^{2}_{0}(y)} \right]  \cdot \left[ N^{k}\left( \overline{\phi} \right)(y) - \frac{1}{k^{2}} \nabla M^{k}\left( \overline{\phi} \right)(y) \right] dy \\
&+& \int_{D}  u_{1}(y) \, \left[ 1-\frac{\bm{n}^{2}(y)}{\bm{n}^{2}_{0}(y)} \right]  \cdot  \frac{\left( \omega^{2} \,\bm{n}^{2}_{0}(y) - k^{2}\right)}{k^{2}}  \overline{\phi}(y)  \, dy. 
\end{eqnarray*}
Taking the adjoint of the operators $N^{k}(\cdot)$ and $\nabla M^{k}(\cdot)$ and the equation  
$(\ref{DefU1})$ to obtain: 
\begin{eqnarray}\label{TermJ2}
\nonumber
\bf{J_{2}} &=& k^{2} \langle \int_{D} \Gamma_{1}(\cdot ,y) \cdot u_{1}(y) \, \left[ 1-\frac{\bm{n}^{2}(y)}{\bm{n}^{2}_{0}(y)} \right] dy; \phi \rangle \\
&+& \frac{\omega^{2}}{k^{2}} \, \langle u_{0}\, \bm{n}^{2}_{0}(\cdot);   \phi \rangle  - \langle  u_{0}; \phi \rangle + \langle  u_{1} \, \left[ \frac{\bm{n}^{2}(\cdot)}{\bm{n}^{2}_{0}(\cdot)} - \frac{\omega^{2} \bm{n}^{2}(\cdot)}{k^{2}} \right]; \phi \rangle. 
\end{eqnarray}
\end{enumerate} 
Plugging the expression of $\bf{J_{1}}$ and $\bf{J_{2}}$, see $(\ref{TermJ1})$ and $(\ref{TermJ2})$, into the equation $(\ref{DefU1CurlCurl})$  to obtain: 
\begin{equation*}
\langle \nabla \times \nabla \times u_{1}(\cdot) - \omega^{2} \, \bm{n}^{2}(\cdot) \, u_{1}(\cdot); \phi \rangle  = \langle \nabla \times \nabla \times u_{0}(\cdot) - \omega^{2} \, \bm{n}^{2}_{0}(\cdot) \, u_{0}(\cdot) ; \phi \rangle =0.
\end{equation*}
This was to be proved.  

\section{Proof of \textbf{Proposition $\ref{Proposition2.1}$}}\label{proof-pro-2.1}

\subsection{A priori estimates}\label{Section4}\
By $u_{0}(\cdot)$ we denote the incident electromagnetic field in the absence of particles inside the domain $\Omega$ which is of solenoidal type, i.e. $\div(u_{0}) = 0$, and by $u_{1}(\cdot)$ the electromagnetic field after injecting one particle inside $\Omega$.\\
Now, we start with the following Lippmann-Schwinger integral equation: 
\begin{equation*}
u_{1}(x) +  \omega^{2} \, \mu \, \left(\epsilon_{0}(z) - \epsilon_{p} \right) \, \int_{D} \,G_{k}(x,y) \cdot u_{1}(y) \, dy = u_{0}(x), \qquad x \in \,  D. 
\end{equation*} 
Thanks to the expansion formula for the Green kernel $G_{k}(\cdot,\cdot)$, see for instance $(\ref{DecompositionGreenKernel})$, we rewrite the previous equation as: 
\begin{equation}\label{MLPASM}
u_{1}(x) + \omega^{2} \, \mu \, \left(\epsilon_{0}(z) - \epsilon_{p} \right) \,  \int_{D} \Upsilon(x,y) \cdot u_{1}(y) \, dy = u_{0}(x) + Err_{\Gamma}(x),
\end{equation}
where the remainder part $Err_{\Gamma}(x)$, for $x \in D$, is given by 
\begin{equation*}\label{YGC}
Err_{\Gamma}(x) := - \omega^{2} \, \mu \, \left(\epsilon_{0}(z) - \epsilon_{p} \right) \, \int_{D}  \Gamma(x,y) \cdot u_{1}(y) \, dy.
\end{equation*} 
Using the definition of the Magnetization operator $\nabla M(\cdot)$, see $(\ref{DefNDefM})$, the definition of the kernel $\Upsilon(\cdot,\cdot)$, see $(\ref{AI-Homogeneous})$, the definition of the function $\eta(\cdot)$, see $(\ref{Defeta})$, and scaling the equation $(\ref{MLPASM})$ to the domain $B$, to get  
\begin{equation*}\label{Cavalaclaque?}
\tilde{u}_{1}(x) - \eta(z) \, \nabla M(\tilde{u}_{1})(x) = \tilde{u}_{0}(x) + \widetilde{Err}_{\Gamma}(x). 
\end{equation*} 
In the sequel, we project the previous equation in each subspace given by $(\ref{L2-decomposition})$. 
\begin{enumerate}
\item Taking the $\mathbb{L}^{2}(B)$-inner product with respect to $e^{(1)}_{n}(\cdot)$,
\begin{equation*}
\langle \tilde{u}_{1}, e^{(1)}_{n} \rangle - \eta(z) \, \langle \nabla M(\tilde{u}_{1}), e^{(1)}_{n} \rangle = \langle \tilde{u}_{0}, e^{(1)}_{n} \rangle +\langle \widetilde{Err}_{\Gamma}, e^{(1)}_{n} \rangle. 
\end{equation*} 
Using the fact that
$\nabla M\left( e^{(1)}_{n} \right) = 0$, see Lemma \ref{MagnetizationOperator}, to reduce the previous equation into: 
\begin{equation*}
\langle \tilde{u}_{1}, e^{(1)}_{n} \rangle = \langle \tilde{u}_{0}, e^{(1)}_{n} \rangle +\langle \widetilde{Err}_{\Gamma}, e^{(1)}_{n} \rangle. 
\end{equation*} 
After taking the modulus, we obtain
\begin{equation*}
\left\vert \langle \tilde{u}_{1}, e^{(1)}_{n} \rangle \right\vert  \lesssim  \left\vert \langle \tilde{u}_{0}, e^{(1)}_{n} \rangle \right\vert  + \left\vert \langle \widetilde{Err}_{\Gamma}, e^{(1)}_{n} \rangle \right\vert,
\end{equation*}
then
\begin{eqnarray}\label{AOT}
\nonumber
\sum_{n} \left\vert \langle \tilde{u}_{1}, e^{(1)}_{n} \rangle \right\vert^{2}  & \lesssim & \sum_{n} \left\vert \langle \tilde{u}_{0}, e^{(1)}_{n} \rangle \right\vert^{2} + \sum_{n} \left\vert \langle \widetilde{Err}_{\Gamma}, e^{(1)}_{n} \rangle \right\vert^{2} \\
& \overset{(\ref{EstimationErrGamma1})}{\lesssim} & \sum_{n} \left\vert \langle \tilde{u}_{0}, e^{(1)}_{n} \rangle \right\vert^{2} + a^{\frac{2(6-5\delta-2\delta^{2})}{(3-2\delta)}} \,  \left\Vert  \tilde{u}_{1} \right\Vert^{2}_{\mathbb{L}^{2}(B)}. 
\end{eqnarray}
\item Taking the $\mathbb{L}^{2}(B)$-inner product with respect
to $e^{(2)}_{n}(\cdot)$,
\begin{equation*}
\langle \tilde{u}_{1}, e^{(2)}_{n} \rangle - \eta(z) \, \langle \nabla M(\tilde{u}_{1}), e^{(2)}_{n} \rangle  = \langle \tilde{u}_{0}, e^{(2)}_{n} \rangle + \langle \widetilde{Err}_{\Gamma}; e^{(2)}_{n} \rangle.
\end{equation*}
Since $\tilde{u}_{0} \in \mathbb{H}\left( \div = 0 \right) = \left( \mathbb{H}_{0}\left( \div = 0 \right) \oplus \nabla \mathcal{H}armonic \right) \perp \mathbb{H}_{0}\left( Curl = 0 \right)$ and thanks to Lemma \ref{MagnetizationOperator}, the previous equation will be reduced to 
\begin{equation*}
\frac{\epsilon_{p}}{\epsilon_{0}(z)} \langle \tilde{u}_{1} , e_{n}^{(2)} \rangle =   \langle \widetilde{Err}_{\Gamma}, e_{n}^{(2)} \rangle,
\end{equation*}
and then 
\begin{equation}\label{2V111}
\sum_{n} \left\vert \langle \tilde{u}_{1} , e_{n}^{(2)} \rangle \right\vert^{2} = \left\vert \frac{ \epsilon_{0}(z)}{\epsilon_{p}} \right\vert^{2} \sum_{n} \left\vert \langle  \widetilde{Err}_{\Gamma}, e^{(2)}_{n} \rangle \right\vert^{2} \overset{(\ref{RDSN})}{=} \mathcal{O}\left(a^{2} \, \left\Vert \tilde{u}_{1} \right\Vert^{2}_{\mathbb{L}^{2}(B)} \right).
\end{equation}
\item Taking the $\mathbb{L}^{2}(B)$-inner product with respect to $e^{(3)}_{n}(\cdot)$,
\begin{equation*}
\langle \tilde{u}_{1}, e^{(3)}_{n} \rangle - \eta(z) \,  \langle \nabla M(\tilde{u}_{1}), e^{(3)}_{n} \rangle  = \langle \tilde{u}_{0}, e^{(3)}_{n} \rangle + \langle \widetilde{Err}_{\Gamma}, e^{(3)}_{n} \rangle, 
\end{equation*} 
which can be rewritten, knowing that $\nabla M( e^{(3)}_{n}) = \lambda^{(3)}_{n} \,  e^{(3)}_{n}$, as
\begin{equation*}
\langle \tilde{u}_{1}, e^{(3)}_{n} \rangle \left( 1 - \eta(z) \lambda^{(3)}_{n} \right) \overset{(\ref{Defeta})}{=} \langle \tilde{u}_{1}, e^{(3)}_{n} \rangle \left( 1 - \frac{\left(\epsilon_{0}(z) - \epsilon_{p} \right)}{\epsilon_{0}(z)} \lambda^{(3)}_{n} \right) = \langle \tilde{u}_{0}, e^{(3)}_{n} \rangle + \langle  \widetilde{Err}_{\Gamma}, e^{(3)}_{n} \rangle. 
\end{equation*} 
Afterwards, 
\begin{equation*}
\left\vert \langle \tilde{u}_{1}, e^{(3)}_{n} \rangle \right\vert  \lesssim  \frac{ \left\vert \langle \tilde{u}_{0}, e^{(3)}_{n} \rangle \right\vert +  \left\vert \langle  \widetilde{Err}_{\Gamma}, e^{(3)}_{n} \rangle \right\vert }{\left\vert \epsilon_{0}(z) -  \left(\epsilon_{0}(z) - \epsilon_{p} \right) \lambda^{(3)}_{n} \right\vert},
\end{equation*}
hence 
\begin{equation*}
\sum_{n} \left\vert \langle \tilde{u}_{1}, e^{(3)}_{n} \rangle \right\vert^{2}  \lesssim \sum_{n} \frac{  \left\vert \langle \tilde{u}_{0}, e^{(3)}_{n} \rangle \right\vert^{2} + \left\vert \langle  \widetilde{Err}_{\Gamma}, e^{(3)}_{n} \rangle \right\vert^{2} }{\left\vert \epsilon_{0}(z) - \left(\epsilon_{0}(z) - \epsilon_{p} \right) \, \lambda^{(3)}_{n} \right\vert^{2}}.
\end{equation*}
As we are approaching the $\lambda^{(3)}_{n_{0}}$ eigenvalue we have:
\begin{equation}\label{approximationoflambdan0}
\left\vert \epsilon_{0}(z) -  \left(\epsilon_{0}(z) - \epsilon_{p} \right) \, \lambda^{(3)}_{n} \right\vert \sim \begin{cases} a^{h} & \text{if} \quad n = n_{0} \\ 1 & \text{if} \quad n \neq n_{0}
\end{cases}.
\end{equation}
Then,
\begin{equation*}
\sum_{n} \left\vert \langle \tilde{u}_{1}, e^{(3)}_{n} \rangle \right\vert^{2}  \lesssim a^{-2h} \; \sum_{n}  \left\vert \langle \tilde{u}_{0}, e^{(3)}_{n} \rangle \right\vert^{2}  + a^{-2h} \sum_{n} \left\vert \langle  \widetilde{Err}_{\Gamma}, e^{(3)}_{n} \rangle \right\vert^{2}. 
\end{equation*}
Now, using the relation $(\ref{TMCAA})$, we obtain: 
\begin{equation}\label{3V111}
\sum_{n} \left\vert \langle \tilde{u}_{1}, e^{(3)}_{n} \rangle \right\vert^{2}  \lesssim a^{-2h} \; \sum_{n}  \left\vert \langle \tilde{u}_{0}, e^{(3)}_{n} \rangle \right\vert^{2} + a^{2-2h} \, \left\Vert \tilde{u}_{1} \right\Vert^{2}_{\mathbb{L}^{2}(B)}. 
\end{equation}
\end{enumerate}
Combining $(\ref{AOT}), (\ref{2V111})$ and $(\ref{3V111})$ we get finally, under the condition $h < 1$, the following a priori estimate: 
\begin{equation*}\label{4V111}
\left\Vert \tilde{u}_{1} \right\Vert_{\mathbb{L}^{2}(B)}  \lesssim  a^{-h} \, \left\Vert \tilde{u}_{0}  \right\Vert_{\mathbb{L}^{2}(B)}.
\end{equation*}
This was to be demonstrated. \\
The formula $(\ref{2V111})$ shows us the smallness of the resulting electromagnetic field, generated by a solenoidal vector field, in the subspace $\mathbb{H}_{0}\left(Curl = 0 \right)$. The coming lemma, which can be considered as a consequence of Lemma $\ref{Lemmavanishingintegral}$, give us more precisions about the intensity of the incident electromagnetic field.   
\begin{lemma}\label{ProjU012}
For $j=1,2$ the following estimation holds: 
\begin{equation}\label{Norm-P12U0}
\left\Vert \overset{j}{\mathbb{P}} \left(  \tilde{u}_{0} \right) \right\Vert_{\mathbb{L}^{2}(B)} = \mathcal{O}\left(a \right).
\end{equation}
\end{lemma}
\begin{proof} For $j=1,2$, we have: 
\begin{equation*}
\left\Vert \overset{j}{\mathbb{P}} \left(  u_{0} \right) \right\Vert^{2}_{\mathbb{L}^{2}(D)} = \sum_{n} \left\vert \langle u_{0};e^{(j,D)}_{n} \rangle \right\vert^{2} = \sum_{n} \left\vert \int_{D} u_{0}(x) \cdot e^{(j,D)}_{n}(x) \, dx \right\vert^{2}, 
\end{equation*}
where $\left\{ e^{(j,D)}_{n}(\cdot) \right\}_{n \in \mathbb{N} \atop j=1,2}$ are the orthonormal basis defined in $D$.  
By Taylor expansion, we obtain: 
\begin{eqnarray*}
\left\Vert \overset{j}{\mathbb{P}} \left( u_{0} \right) \right\Vert^{2}_{\mathbb{L}^{2}(D)} &=& \sum_{n} \left\vert \int_{D}  \left(  u_{0}(z) + \int_{0}^{1} \, \nabla u_{0}(z+t(x-z)) \cdot (x-z) \, dt \right) \cdot e^{(j,D)}_{n}(x) \, dx \right\vert^{2} \\ 
& \overset{(\ref{vanishingintegral})}{=} & \sum_{n} \left\vert \int_{D}   \int_{0}^{1} \, \nabla u_{0}(z+t(x-z)) \cdot (x-z) \, dt  \cdot e^{(j,D)}_{n}(x) \, dx \right\vert^{2} \\
& \leq &  \left\Vert \int_{0}^{1} \, \nabla u_{0}(z+t(\cdot -z)) \cdot (\cdot -z) \, dt  \right\Vert^{2}_{\mathbb{L}^{2}(D)} = \mathcal{O}\left( a^{5} \right).
\end{eqnarray*}
We skip the remainder of the proof because it consist in scaling  to $B$. 
\end{proof}
\begin{remark} Fortunately, similar relation to $(\ref{Norm-P12U0})$ cannot be true for $j=3$ since, in general, $\int_{B}  e^{(3)}_{n}(x) dx \neq 0$. This last properties is the key step in the proof of Lemma \ref{ProjU012}.    
\end{remark}

\subsection{Estimate the $\mathbb{L}^{2}(B)$-norm of $\widetilde{Err_{\Gamma}}$}

We have need to estimate the $\mathbb{L}^{2}(B)$-norm of $\widetilde{Err_{\Gamma}}$. For this, we project this expression into each subspace decomposing the $\mathbb{L}^{2}(B)$-space.   
\begin{enumerate}
\item Estimation of $\underset{n}{\sum} \left\vert \langle \widetilde{Err_{\Gamma}}; e^{(1)}_{n} \rangle \right\vert^{2}$. \\
From $(\ref{Err-Gamma})$, we have\footnote{Recall that, with respect to the size $a$, we have $ - \, \omega^{2} \, \mu \, (\epsilon_{0}(z)-\epsilon_{p}) \sim 1.$} 
\begin{eqnarray*}
Err_{\Gamma}(x) &:=& - \omega^{2} \, \mu \, \int_{D} \Gamma(x,y) \cdot u_{1}(y) \, (\epsilon_{0}(y) - \epsilon_{p}) \, dy \\
& \overset{(\ref{AY})}{=} &  \int_{D} \underset{x}{\nabla}  \underset{x}{\nabla} M\left( \Phi_{0}(\cdot, y) \nabla \epsilon_{0}(y) \right)(x) \cdot u_{1}(y)  \, \frac{(\epsilon_{0}(y) - \epsilon_{p})}{\left(\epsilon_{0}(y)\right)^{2}} \, dy \\ &-& \omega^{2} \, \mu \,  \int_{D} W_{4}(x,y) \cdot u_{1}(y)  \,  (\epsilon_{0}(y) - \epsilon_{p}) \, dy,
\end{eqnarray*}
or, after an integration by parts,
\begin{eqnarray*}
Err_{\Gamma}(x) & = & - \nabla \, \div \, N \left( \div N\left(\nabla \epsilon_{0} \otimes u_{1} \frac{(\epsilon_{0} - \epsilon_{p})}{\left(\epsilon_{0}\right)^{2}} \right) \right)(x) \\ &+& \nabla \, \div \, SL\left( \nu \cdot N\left(\nabla \epsilon_{0} \otimes u_{1} \frac{(\epsilon_{0} - \epsilon_{p})}{\left(\epsilon_{0}\right)^{2}} \right) \right)(x) \\ &-& \omega^{2} \, \mu \,  \int_{D} W_{4}(x,y) \cdot u_{1}(y) \,  (\epsilon_{0}(y) - \epsilon_{p}) \, dy.
\end{eqnarray*}
Set 
\begin{equation}\label{Defphi1}
\phi_{1} := - \div \, N \left( \div N\left(\nabla \epsilon \otimes u_{1} \frac{(\epsilon_{0} - \epsilon_{p})}{\left(\epsilon_{0}\right)^{2}} \right) \right) +  \div \, SL\left( \nu \cdot N\left(\nabla \epsilon \otimes u_{1} \frac{(\epsilon_{0} - \epsilon_{p})}{\left(\epsilon_{0}\right)^{2}} \right) \right)
\end{equation}
and $\phi_{2}(x) := \int_{D} W_{4}(x,y) \cdot u_{1}(y)\, dy$,
 then 
\begin{equation}\label{ComapctErrGamma}
Err_{\Gamma}(x)  = \nabla \phi_{1}(x) - \omega^{2} \, \mu \,  \phi_{2}(x).
\end{equation}
Remark that $\nabla \phi_{1} \in \left( \mathbb{H}_{0}\left(Curl=0 \right) \overset{\perp}{\oplus} \nabla \mathcal{H}armonic \right) \perp \mathbb{H}_{0}\left(\div=0 \right)$, then regardless on it's scale we get $\langle \widetilde{\nabla \phi_{1}} ; e^{(1)}_{n} \rangle = 0$ and in that case $\langle \widetilde{Err_{\Gamma}} ; e^{(1)}_{n} \rangle = - \omega^{2} \, \mu \, \langle \widetilde{ \phi_{2}} ; e^{(1)}_{n} \rangle$. Next, let's focus on the scale of $\phi_{2}$. For this, we assume that $W_{4}(x,y) \simeq \, \left\vert x-y \right\vert^{-\alpha}$, where $\alpha > 0$ is chosen such that $W_{4}(\cdot,y) \in \mathbb{L}^{\frac{3(3-2\delta)}{(3+2\delta)}}(D)$. Basic calculus on the integrability of $W_{4}$ allows us to fix $\alpha = \frac{3+2\delta}{3-2\delta} - \delta$ and then,  
\begin{equation}\label{scalephi2}
\widetilde{\phi_{2}}(x) = a^{\frac{6-5\delta-2\delta^{2}}{(3-2\delta)}} \, \int_{B} \widetilde{W_{4}}(x,y) \cdot \tilde{u}_{1}(y) \, dy.
\end{equation} 
Finally, 
\begin{eqnarray}\label{EstimationErrGamma1}
\nonumber
\sum_{n} \left\vert\langle \widetilde{Err_{\Gamma}} ; e^{(1)}_{n} \rangle \right\vert^{2} = \omega^{4} \, \mu^{2} \, \sum_{n} \left\vert\langle \widetilde{\phi_{2}} ; e^{(1)}_{n} \rangle \right\vert^{2} &=& a^{\frac{2(6-5\delta-2\delta^{2})}{(3-2\delta)}} \,  \sum_{n} \left\vert\langle \int_{B} \widetilde{W_{4}}(\cdot ,y) \cdot \tilde{u}_{1}(y) \, dy ; e^{(1)}_{n} \rangle \right\vert^{2} \\ \nonumber
& \leq & a^{\frac{2(6-5\delta-2\delta^{2})}{(3-2\delta)}} \,  \left\Vert \int_{B} \widetilde{W_{4}}(\cdot ,y) \cdot \tilde{u}_{1}(y) \, dy  \right\Vert^{2}_{\mathbb{L}^{2}(B)} \\
& = & \mathcal{O}\left( a^{\frac{2(6-5\delta-2\delta^{2})}{(3-2\delta)}} \,  \left\Vert  \tilde{u}_{1} \right\Vert^{2}_{\mathbb{L}^{2}(B)}\right).
\end{eqnarray}
\item Estimation of $\underset{n}{\sum} \left\vert \langle e^{(2)}_{n}; \widetilde{Err}_{\Gamma} \rangle \right\vert^{2}$.\\
Let's recall from $(\ref{ComapctErrGamma})$ that we have
$Err_{\Gamma}(x) =   \, \nabla \phi_{1}(x) - \omega^{2} \, \mu \,  \phi_{2}(x)$,
then, after scaling to the domain $B$ and taking the inner product with respect to $e^{(2)}_{n}(\cdot)$, we get
\begin{equation}\label{TMCA}
\langle \widetilde{Err}_{\Gamma} ;e^{(2)}_{n} \rangle  = \langle \widetilde{\nabla \phi}_{1}; e^{(2)}_{n} \rangle - \omega^{2} \, \mu \,  \langle \widetilde{\phi}_{2} ;e^{(2)}_{n} \rangle,
\end{equation}
where $\widetilde{\phi}_{2}(\cdot)$ is the vector field given by $(\ref{scalephi2})$ and by scaling $(\ref{Defphi1})$ we obtain 
\begin{equation}\label{LAMLM}
\widetilde{\phi}_{1} = - a^{2} \, \div \, N \left( \div N\left( \widetilde{u^{\star}_{1}} \right) \right) + a^{3} \,  \div \, SL\left( \nu \cdot N\left( \widetilde{ u^{\star}_{1}}  \right) \right)
\end{equation}
where 
\begin{equation}\label{AATNMB}
u^{\star}_{1} = \nabla \epsilon \otimes u_{1} \frac{(\epsilon_{0} - \epsilon_{p})}{\left(\epsilon_{0}\right)^{2}}.
\end{equation}
Clearly the term $\langle \widetilde{\phi}_{2} ;e^{(2)}_{n} \rangle$ is negligible compared to $\langle \widetilde{\nabla \phi}_{1}; e^{(2)}_{n} \rangle$, and thanks to the equation $(\ref{LAMLM})$ we approximate $(\ref{TMCA})$ as,   
\begin{equation*}
\langle \widetilde{Err}_{\Gamma} ;e^{(2)}_{n} \rangle  \simeq  - a \, \langle \nabla  \div \, N \left( \div N\left( \widetilde{u^{\star}_{1}} \right) \right); e^{(2)}_{n} \rangle +  a^{2} \, \langle \nabla  \div \, SL\left( \nu \cdot N\left( \widetilde{ u^{\star}_{1}}  \right) \right); e^{(2)}_{n} \rangle.
\end{equation*}
By taking the square modulus and then the series with respect to the index $n$, we get: 
\begin{eqnarray*}
\sum_{n} \left\vert  \langle \widetilde{Err}_{\Gamma} ;e^{(2)}_{n} \rangle  \right\vert^{2} & \lesssim &  a^{2} \,\sum_{n} \left\vert \langle \nabla  \div \, N \left( \div N\left( \widetilde{u^{\star}_{1}} \right) \right); e^{(2)}_{n} \rangle \right\vert^{2} \\ &+& a^{4} \,\sum_{n} \left\vert \langle \nabla  \div \, SL\left( \nu \cdot N\left( \widetilde{ u^{\star}_{1}}  \right) \right); e^{(2)}_{n} \rangle \right\vert^{2}, 
\end{eqnarray*}
hence, 
\begin{equation*}
\sum_{n} \left\vert  \langle \widetilde{Err}_{\Gamma} ;e^{(2)}_{n} \rangle  \right\vert^{2}  \lesssim   a^{2} \, \left\Vert \nabla \nabla \cdot N \left( \div N\left( \widetilde{u^{\star}_{1}} \right) \right)  \right\Vert^{2}_{\mathbb{L}^{2}(B)}+a^{4}  \left\Vert  \nabla  \div \, SL\left( \nu \cdot N\left( \widetilde{ u^{\star}_{1}}  \right) \right) \right\Vert^{2}_{\mathbb{L}^{2}(B)}.
\end{equation*}
Thanks to Calderon-Zygmund inequality and the continuity of the operator $\nabla  \div \, SL(\cdot)$ we deduce that:
\begin{equation*}
\sum_{n} \left\vert  \langle \widetilde{Err}_{\Gamma} ;e^{(2)}_{n} \rangle  \right\vert^{2}  \lesssim  a^{2} \, \left\Vert  \div N\left( \widetilde{u^{\star}_{1}} \right)  \right\Vert^{2}_{\mathbb{L}^{2}(B)}+a^{4}  \left\Vert   \nu \cdot N\left( \widetilde{ u^{\star}_{1}}  \right)  \right\Vert^{2}_{\mathbb{H}^{1/2}(\partial B)}.
\end{equation*}
Now, using the continuity of $\div N(\cdot)$ operator and the trace operator we obtain: 
\begin{equation*}
\sum_{n} \left\vert  \langle \widetilde{Err}_{\Gamma} ;e^{(2)}_{n} \rangle  \right\vert^{2}  \lesssim  a^{2} \, \left\Vert  \widetilde{u^{\star}_{1}}  \right\Vert^{2}_{\mathbb{L}^{2}(B)}+a^{4}  \left\Vert    N\left( \widetilde{ u^{\star}_{1}}  \right)  \right\Vert^{2}_{\mathbb{H}^{1}( B)}.
\end{equation*}
Now, we use the continuity of the Newtonian operator and the definition of $u^{\star}_{1}$, see $(\ref{AATNMB})$, to reduce the previous inequality to:  
\begin{equation}\label{RDSN}
\sum_{n} \left\vert  \langle \widetilde{Err}_{\Gamma} ;e^{(2)}_{n} \rangle  \right\vert^{2}  = \mathcal{O}\left(  a^{2} \, \left\Vert   \tilde{u}_{1} \right\Vert^{2}_{\mathbb{L}^{2}(B)}  \right) \overset{(\ref{aprioriestimateregimemoderate})}{=}  \mathcal{O}\left(a^{2-2h}\right).
\end{equation}
\item Estimation of $\underset{n}{\sum} \left\vert \langle e^{(3)}_{n}; \widetilde{Err}_{\Gamma} \rangle \right\vert^{2}$.\\
Similarly to $(\ref{TMCA})$, we have
$\langle \widetilde{Err}_{\Gamma} ;e^{(3)}_{n} \rangle  =  \langle \widetilde{\nabla \phi}_{1}; e^{(3)}_{n} \rangle - \omega^{2} \, \mu  \langle \widetilde{\phi}_{2} ;e^{(3)}_{n} \rangle$
and similar to the calculus done for  $\langle \widetilde{Err}_{\Gamma} ;e^{(2)}_{n} \rangle$ allows us to prove that: 
\begin{equation}\label{TMCAA}
\sum_{n} \left\vert  \langle \widetilde{Err}_{\Gamma} ;e^{(3)}_{n} \rangle  \right\vert^{2}  \simeq \sum_{n} \left\vert  \langle \widetilde{Err}_{\Gamma} ;e^{(2)}_{n} \rangle  \right\vert^{2}  = \mathcal{O}\left(a^{2-2h}\right).
\end{equation} 
\end{enumerate}

\subsection{End of the proof of \textbf{Proposition $\ref{Proposition2.1}$}}\label{ADZ}
We split the calculus into two steps. 
\begin{enumerate}
\item Estimation of the scattering matrix $\int_{D} W(x) dx$.\\
We have\footnote{We recall that for two arbitrary vectors $A \in \mathbb{R}^{n}$ and $B \in \mathbb{R}^{m}$ their tonsorial product $A \otimes B$ is the $n \times m$ matrix given by $(A \otimes B)_{ij} = A_{i} \, B_{j}$.}, 
\begin{equation*}
\int_{D} W(x) \, dx = a^{3} \;  \int_{B} \widetilde{W}(x) \, dx = a^{3} \, \sum_{n} \langle \widetilde{W} ; e^{(3)}_{n} \rangle_{\mathbb{L}^{2}(B)} \, \otimes \, \int_{B} e^{(3)}_{n}(x) \, dx, 
\end{equation*}
Now, using the definition of the matrix $W(\cdot)$, see $(\ref{SAH})$, we obtain: 
\begin{equation*}
\int_{D} W(x) \, dx =  a^{3} \, \sum_{n} \langle \left[ I -  \overline{\eta(z)} \,   \nabla M \right]^{-1}\left( I \right) ; e^{(3)}_{n} \rangle_{\mathbb{L}^{2}(B)} \, \otimes \, \int_{B} e^{(3)}_{n}(x) \, dx. 
\end{equation*}
Taking the adjoint operator of
$\left[ I - \overline{\eta(z)} \,   \nabla M \right]^{-1}$,
using the fact that the Magnetization operator $\nabla M$ is self-adjoint,  $\nabla M\left( e^{(3)}_{n} \right) = \lambda^{(3)}_{n} \, e^{(3)}_{n}$ and the definition of the function $\eta(\cdot)$ to obtain: 
\begin{eqnarray*}
\int_{D} W(x) \, dx & = & a^{3} \, \sum_{n} \frac{\epsilon_{0}(z)}{\left( \epsilon_{0}(z) - \left(\epsilon_{0}(z) - \epsilon_{p} \right) \,   \lambda^{(3)}_{n} \right)} \, \int_{B} e^{(3)}_{n}(x) \, dx \, \otimes \, \int_{B} e^{(3)}_{n}(x) \, dx \\
&\overset{(\ref{approximationoflambdan0})}{=}& a^{3} \,\frac{\epsilon_{0}(z)}{\left( \epsilon_{0}(z) - \left(\epsilon_{0}(z) - \epsilon_{p} \right) \,   \lambda^{(3)}_{n_{0}} \right)} \, \, \int_{B} e^{(3)}_{n_{0}}(x) \, dx \, \otimes \, \int_{B} e^{(3)}_{n_{0}}(x) \, dx + \mathcal{O}\left( a^{3} \right).
\end{eqnarray*} 
\item Estimation of $\left\Vert W \right\Vert_{\mathbb{L}^{2}(D)}$.\\
From the definition of $W(\cdot)$, see $(\ref{SAH})$, we have
$W - \overline{\eta(z)} \,   \nabla M \left( W \right) =  \bm{I}$,
then, after taking the inner product with respect to $e^{(j)}_{n}$, where $j=1$ or $j=2$, we obtain:\footnote{The notation $\delta_{2,j}$ is the Kronecker symbol, i.e: $\delta_{2,j} = 1$ for $j=2$ and $\delta_{2,j} = 0$ for $j \neq 2$.} 
\begin{equation*}
\langle W; e^{(j)}_{n} \rangle \left(1 - \overline{\eta(z)}  \, \delta_{2,j}  \right) = \langle  \bm{I}, e^{(j)}_{n} \rangle \overset{(\ref{vanishingintegral})}{=} 0. 
\end{equation*}
This implies $W \in \nabla \mathcal{H}armonic$. Then, 
\begin{eqnarray*}
\left\Vert W \right\Vert^{2}_{\mathbb{L}^{2}(D)} &=& a^{3} \, \left\Vert \widetilde{W} \right\Vert^{2}_{\mathbb{L}^{2}(B)} = a^{3} \; \underset{n}{\sum} \left\vert \langle \widetilde{W} ; e^{(3)}_{n} \rangle \right\vert^{2} \\ & \overset{(\ref{SAH})}{=} & a^{3} \; \underset{n}{\sum} \left\vert \langle \left( I - \overline{\eta(z)} \,  \, \nabla M \right)^{-1}(I) ; e^{(3)}_{n} \rangle \right\vert^{2} \\
& = & a^{3} \; \underset{n}{\sum} \frac{\left\vert \epsilon_{0}(z) \right\vert^{2} \, \left\vert \int_{B} e^{(3)}_{n}(x) \, dx \right\vert^{2}}{\left\vert \epsilon_{0}(z) - (\epsilon_{0}(z) - \epsilon_{p}) \lambda_{n}^{(3)} \right\vert^{2}} \overset{(\ref{approximationoflambdan0})}{=} \mathcal{O}\left( a^{3-2h} \right).
\end{eqnarray*}
Finally, 
\begin{equation*}\label{NormWD}
\left\Vert W \right\Vert_{\mathbb{L}^{2}(D)} = \mathcal{O}\left( a^{\frac{3}{2}-h} \right).
\end{equation*}
\end{enumerate}
This ends the proof of \textbf{Proposition $\ref{Proposition2.1}$}.

\section{Appendices}\label{Appendix}
\subsection{A spectral decomposition of the $\mathbb{L}^{2}(D)$ space}
\
\newline
\\
In our work, we have used a particular, but natural, spectral decomposition of the $\mathbb{L}^{2}(D)$ space.
We endowed this space with a basis constructed from the two main operators appearing in our model, mainly the vector Newtonian and Magnetization operators, defined by $(\ref{DefNDefM})$. Here, we give a justification of that spectral decomposition.   
\bigskip

We use the following direct sum of\footnote{To note short we use the notation $\mathbb{L}^{2}(D)$ instead of $\left( \mathbb{L}^{2}(D) \right)^{3} := \mathbb{L}^{2}(D) \times \mathbb{L}^{2}(D) \times \mathbb{L}^{2}(D).$} $\mathbb{L}^{2}(D)$ vector fields, where we assume that $D$ is sufficiently smooth domain of $\mathbb{R}^{3}$ with outward unit normal $\nu$, see \cite{Dautry-Lions}, page 314, 
\begin{equation}\label{L2-decomposition}
\mathbb{L}^{2}(D)  = \mathbb{H}_{0}\left(\div=0 \right) \overset{\perp}{\oplus} \mathbb{H}_{0}\left(Curl=0 \right) \overset{\perp}{\oplus} \nabla \mathcal{H}armonic
\end{equation}   
where 
\begin{eqnarray*}
\mathbb{H}_{0}\left(\div=0 \right) &:=& \left\lbrace E \in  \mathbb{L}^{2}(D), \, \div E = 0, \, \nu \cdot E = 0 \, \; \text{on} \;\, \partial D \right\rbrace \\
\mathbb{H}_{0}\left(Curl =0 \right) &:=& \left\lbrace E \in \mathbb{L}^{2}(D), \, Curl \, E = 0, \, \nu \times E = 0 \, \; \text{on} \;\, \partial D \right\rbrace
\end{eqnarray*}
and 
\begin{equation*}
\nabla \mathcal{H}armonic := \left\lbrace E: \; E = \nabla \psi, \, \psi \in \mathbb{H}^{1}(D), \, \Delta\psi=0 \right\rbrace.
\end{equation*}
Other decompositions can be found in \cite{Dautry-Lions} and the references therein. The coming proposition gives us more precisions about the choice of $(\ref{L2-decomposition})$ among other available decompositions.
\begin{proposition}\label{PropositionSpectralTheory}
\begin{enumerate}
\item The Newtonian potential operator $N(\cdot)$ admits an orthonormal basis, noted by $\{ e^{(1)}_{n} \}_{n \geq 0}$, on the subspace $\mathbb{H}_{0}\left(\div=0 \right)$ and another one orthonormal basis, noted by $\{ e^{(2)}_{n} \}_{n \geq 0}$, on the subspace $\mathbb{H}_{0}\left(Curl =0 \right)$, i.e:
\begin{equation*}\label{b1b2}
N( e^{(j)}_{n} ) = \lambda^{(j)}_{n} \,\,  e^{(j)}_{n}, \,\, j=1,2. \,\, 
\end{equation*}
\item The Magnetization operator $\nabla M(\cdot)$ admits an orthonormal basis, noted by $\{ e^{(3)}_{n} \}_{n \geq 0}$, on the subspace $\nabla \mathcal{H}armonic $, i.e: 
\begin{equation}\label{b3}
\nabla M( e^{(3)}_{n} ) = \lambda^{(3)}_{n} \,\,  e^{(3)}_{n}.
\end{equation}
\end{enumerate}
\end{proposition}
Before we move to the proof of the previous proposition we need to note the fascinating remark. 
\begin{remark}\label{LimSeq}
The value $1 / 2$ is the only limit point for the sequence $\{ \lambda^{(3)}_{n} \}_{n \in \mathbb{N}}$.    
\end{remark}
\begin{proof}
of \textbf{Proposition $\ref{PropositionSpectralTheory}$}.
The proof of the point (2) can be found in \cite{AhnDyaRae99} where an explicit expression of the eigenvalues and eigenfunctions of the Magnetization operator are given. Here, we outline the proof of point (1). First, we recall the following lemma which can be found as Theorem 7 in \cite{Dautry-Lions}, for instance.
\begin{lemma}\label{THM-Dautry & Lions}
Let $V \hookrightarrow H$ be two Hilbert spaces with compact injection, $V$ being dense in $H$, $a(\cdot,\cdot)$ a continuous hermitian sesquilinear from $V \times V$ and coercive on $V$ and let $A$ the unbounded self adjoint operator defined by 
\begin{enumerate}
\item[i)] $a(u,v) = (Au,v) \, \forall u \in D(A) \; and \; v \in V$. 
\item[ii)] $D(A) = \left\lbrace u \in V, such \; that \, v \to a(u,v) is \, continuous \, on \, V \, for \, the \, topology \, on \, H  \right\rbrace.$
\end{enumerate} 
Then $\sigma(A) = \sigma_{p}(A) = \left\lbrace \lambda_{k} \right\rbrace_{k \in \mathbb{N}}$ with $0 < \alpha \leq \lambda_{k}$.
\end{lemma}
Now, in the space $\mathbb{L}^{2}_{0}\left(\div=0 \right)$ we study the equation:
\begin{equation}\label{equa-1st-subspace}
N(E) = \lambda \, E, \quad \text{in} \quad D,
\end{equation}
which, after taking the Laplacian operator, becomes
\begin{equation}\label{equa-1st-subspace+}
E = - \, \lambda \, \Delta E.
\end{equation}
Hence, by combining $(\ref{equa-1st-subspace})$ and $(\ref{equa-1st-subspace+})$, we get:\footnote{The abbreviation $IBP$ refer to 'Integration by parts'.}
\begin{equation}\label{HATB}
E = - N(\Delta E) \overset{IBP}{=} E - SL(\partial_{\nu} E) + DL(E) \Rightarrow 0 = - SL(\partial_{\nu} E) + DL(E) \quad \text{in} \; D,
\end{equation}
where $SL\left( \cdot \right)$ denote the single-layer operator with vanishing frequencies 
and $DL\left( \cdot \right)$ is the double-layer operator defined by: 
\begin{equation*}
DL\left( F \right)(x) := \int_{\partial D} \frac{\Phi_{0}(x,y)}{\partial \nu(y)} \, F(y) \, d\sigma(y), \quad x \in \mathbb{R}^{3} \setminus \partial D.
\end{equation*}
Now, in $(\ref{HATB})$ we let $ x \to \partial D$ from inside and we use the jump relations, see \cite{colton2019inverse}, to get: 
\begin{equation*}\label{B-CDT-1st-subspace}
0 = - SL(\partial_{\nu} E) - \frac{1}{2} E + DL(E) \quad \text{on} \; \partial \, D.
\end{equation*}   
Finally, we end up with the following PDE: 
\begin{equation}\label{PDEDpartialD}
\begin{cases}
\lambda^{-1} \, E = - \Delta E \; \, \qquad \, \qquad \qquad \qquad \qquad \text{in} \, D \\
\qquad \, 0 =  - SL(\partial_{\nu} E) - \frac{1}{2} E + DL(E) \quad \, \text{on} \; \partial \, D
\end{cases}.
\end{equation}
We write the variational formulation of the last PDE to get: 
\begin{equation}\label{equa-var-domain}
\langle - \Delta E; F \rangle_{\mathbb{L}^{2}(D)} = \langle \nabla E ; \nabla F \rangle_{\mathbb{L}^{2}(D)} - \langle F ; \partial_{\nu} E \rangle_{\mathbb{L}^{2}(\partial D)}
\end{equation}
and from the boundary conditions, see for instance $(\ref{PDEDpartialD})$, on $\partial D$ , we have
\begin{equation}\label{equa-var-boundary}
\partial_{\nu} E = SL^{-1}\left[- \frac{1}{2} E + DL(E) \right].
\end{equation} 
Combining $(\ref{equa-var-domain})$ and $(\ref{equa-var-boundary})$ we obtain:  
\begin{equation*}
\langle - \Delta E; F \rangle_{\mathbb{L}^{2}(D)} = \langle \nabla E ; \nabla F \rangle_{\mathbb{L}^{2}(D)} - \langle F ; SL^{-1}\left[- \frac{1}{2} E + DL(E) \right] \rangle_{\mathbb{L}^{2}(\partial D)} := a(E,F). 
\end{equation*}
Without difficulties we can check that the bilinear form $a(\cdot,\cdot)$ is continuous and positive in $\mathbb{L}^{2}(D)$. Moreover, with help of the relation $DL \, SL = SL \, DL^{\star}$, see \cite{Costabel} for instance, we prove that it's also symmetric. Set
\begin{equation}\label{aalpha}
a_{\alpha}(\cdot,\cdot) = a(\cdot,\cdot) + \alpha \, < \cdot,\cdot >_{\mathbb{L}^{2}(D)}, \quad \alpha >0.
\end{equation}  
Then $a_{\alpha}(\cdot,\cdot)$ inherit the properties of $a(\cdot,\cdot)$ and, in addition, it's a coercive bilinear form in $\mathbb{L}^{2}(D)$. This proves one of the hypotheses of  \textbf{Theorem} $\ref{THM-Dautry & Lions}$. Another hypotheses of the same theorem is given by the following compact injection:  
\begin{equation}\label{CompactInj}
\mathbb{H}^{1}_{0}\left(\div=0 \right) \lhook\joinrel\xrightarrow{\text{Bounded}} \mathbb{H}^{1} \lhook\joinrel\xrightarrow{\text{Compact}} \mathbb{L}^{2} \lhook\joinrel\xrightarrow{\text{Projection}} \mathbb{L}^{2}_{0}\left(\div=0 \right),
\end{equation}
and we need also the next denseness result. 
\begin{lemma}
We have, 
\begin{equation}\label{Closure of H-0-1}
\overline{\mathbb{H}^{1}_{0}\left(\div=0 \right)}^{\Vert \cdot \Vert_{\mathbb{L}^{2}_{0}\left(\div=0 \right)}} = \mathbb{L}^{2}_{0}\left(\div=0 \right).
\end{equation}
\end{lemma}
\begin{proof}
Set\footnote{By $\mathcal{D}(\Omega)$ we denote the space of $\mathcal{C}^{\infty}$ functions with compact support.} $U := \left\lbrace  E \in \left( \mathcal{D}(\Omega) \right)^{3}, \; \div E = 0 \right\rbrace$ and referring to \cite{Foias1978} we deduce: 
\begin{equation}\label{Closure of U}
\overline{U}^{\Vert \cdot \Vert_{\mathbb{L}^{2}}} = \mathbb{L}^{2}_{0}\left(\div=0 \right).
\end{equation}
It's easy to see that $U \subset \mathbb{H}^{1}_{0}\left(div=0 \right) \subset \mathbb{L}^{2}_{0}\left(div=0 \right)$ and then by taking the closure with respect to $\mathbb{L}^{2}$-norm on both sides and using the fact that $\mathbb{L}^{2}_{0}\left(\div=0 \right)$ is closed subspace of $\mathbb{L}^{2}$ we obtain:
\begin{equation*}
\mathbb{L}^{2}_{0}\left(\div=0 \right) \overset{(\ref{Closure of U})}{=} \overline{U}^{\Vert \cdot \Vert_{\mathbb{L}^{2}}} \subset \overline{\mathbb{H}^{1}_{0}\left(\div=0 \right)}^{\Vert \cdot \Vert_{\mathbb{L}^{2}}} \subset \overline{\mathbb{L}^{2}_{0}\left(\div=0 \right)}^{\Vert \cdot \Vert_{\mathbb{L}^{2}}} = \mathbb{L}^{2}_{0}\left(\div=0 \right),
\end{equation*} 
this implies 
\begin{equation*}
\overline{\mathbb{H}^{1}_{0}\left(\div=0 \right)}^{\Vert \cdot \Vert_{\mathbb{L}^{2}}}  = \mathbb{L}^{2}_{0}\left(\div=0 \right).
\end{equation*}
The relation $(\ref{Closure of H-0-1})$ follows from  the previous equation and the fact that in $\mathbb{L}^{2}_{0}\left(\div=0 \right)$ the two norms $\Vert \cdot \Vert_{\mathbb{L}^{2}_{0}\left(\div=0 \right)}$ and $\Vert \cdot \Vert_{\mathbb{L}^{2}}$ are equivalent. More exactly, is the same one. 
\end{proof}
At this stage, regarding the constructed bilinear form $a_{\alpha}(\cdot,\cdot)$ given by $(\ref{aalpha})$, the relations $(\ref{CompactInj})$ and $(\ref{Closure of H-0-1})$, we can apply \textbf{Theorem} $\ref{THM-Dautry & Lions}$, with $H:=\mathbb{L}_0^2\left(\div = 0 \right)$ and $V:=\mathbb{H}_0\left(\div = 0 \right)$, to justify the existence of an eigen-system for the operator $N(\cdot)$ on the subspace $\mathbb{H}_{0}\left(\div = 0 \right)$. With minor modifications on the used spaces, the existence of an eigen-system for the operator $N(\cdot)$ on the subspace $\mathbb{H}_{0}\left(Curl = 0 \right)$ can be proved. This ends the proof of \textbf{Proposition} \ref{PropositionSpectralTheory}. 
\end{proof} 
\bigskip

Let $D$ be a domain of radius $a$. Then, in effortless manner, using only the definition of the Newton operator $N(\cdot)$, we can prove that: 
\begin{equation*}\label{NormNewton}
\left\Vert N \right\Vert_{\mathcal{L}\left(\mathbb{L}^{2}(D);\mathbb{L}^{2}(D) \right)} = \mathcal{O}\left( a^{2} \right), 
\end{equation*}
and, consequently, we obtain  
\begin{equation}\label{NormgradN}
\left\Vert \nabla N \right\Vert_{\mathcal{L}\left(\mathbb{L}^{2}(D);\mathbb{L}^{2}(D)\right)} = \mathcal{O}\left( a \right). 
\end{equation}
\newline
 In the next lemma we collect some properties for the Magnetization  operator.
\begin{lemma}\label{MagnetizationOperator} The Magnetization operator $\nabla M (\cdot)$ is self-adjoint and bounded. 
\begin{enumerate} 
\item It satisfies the properties
\begin{equation}\label{NormMagnetization}
\left\Vert \nabla M \right\Vert_{\mathcal{L}\left(\mathbb{L}^{2}(D);\mathbb{L}^{2}(D) \right)} = 1,
\end{equation}
and
\begin{equation*}\label{grad-M-1st-2nd}
\nabla M_{\displaystyle|_{\mathbb{H}_{0}\left(\div=0 \right)}} = 0 \, \quad \text{and} \quad \nabla M_{\displaystyle|_{\mathbb{H}_{0}\left(Curl =0 \right)}} = I.
\end{equation*}
\item  The subspace $\nabla \mathcal{H}armonic$ is invariant, i.e.
\begin{equation*}
\nabla M \left( \nabla \mathcal{H}armonic \right) \subset  \nabla \mathcal{H}armonic.
\end{equation*}
\item Its spectrum $\sigma(\nabla M)$ is part of  $]0,1]$.  
\end{enumerate}  
\end{lemma}
\begin{proof}
The proof of the lemma and other nice properties for the Magnetization operator can be found in \cite{Raevskii1994, Dyakin-Rayevskii, friedman1980mathematical, friedman1981mathematical} and \cite{10.2307/2008286}.  
\end{proof} 
\begin{lemma}\label{Lemmavanishingintegral} The eigenfunctions $e^{(j)}_{n}(\cdot),\, j=1,2,\,$ satisfies  the following mean vanishing integral properties: 
\begin{equation}\label{vanishingintegral}
\int_{B} e^{(j)}_{n}(x) \, dx = 0.
\end{equation}
\end{lemma}
\begin{proof} 
The integral $(\ref{vanishingintegral})$ is nothing that $\langle \bm{I} ; e^{(j)}_{n} \rangle_{\mathbb{L}^{2}(B)}$ and, obviously, the identity matrix $\bm{I}$ is an element in the  $\nabla \mathcal{H}armonic$ subspace which is, from the decomposition $(\ref{L2-decomposition})$, orthogonal to 
\begin{equation*}
\mathbb{H}_{0}\left(\div=0 \right) \overset{\perp}{\oplus} \mathbb{H}_{0}\left(Curl=0 \right).
\end{equation*} 
\end{proof}

\subsection{Justification of $(\ref{approximationoflambdan0})$}
We start by recalling the Lorentz model for the permittivity, see \cite{PhysRevA.82.055802} formula (4) or \cite{engheta2006metamaterials} formula (1.3), 
\begin{equation}\label{Lorentzdef}
\epsilon_{p}(\omega, \gamma) = \epsilon_{\infty} \left[1 + \frac{\omega^{2}_{p}}{\omega^{2}_{0} - \omega^{2} + i \gamma \, \omega} \right],
\end{equation} 
where $\omega_{p}^{2}$ is the electric plasma frequency, $\omega_{0}^{2}$ is the undamped resonance frequency and $\gamma$ is the electric damping parameter. We have the following lemma.

\begin{lemma}\label{LemmaClaim}
Let $n_{0} \in \mathbb{N}$ be any fixed index such that 
\begin{equation*}
\lambda^{(3)}_{n_{0}} \not \in \left] \frac{1}{2} - a^{\frac{h}{2}} ; \frac{1}{2} + a^{\frac{h}{2}} \right[.
\end{equation*} 

We have the following properties.
\bigskip

\begin{enumerate}
\item There exists a unique solution $(\omega, \gamma):=(\omega_{n_0}, \gamma_{n_0})$ to the equation
\begin{equation*}
\epsilon_{0}(z) - \lambda^{(3)}_{n_{0}} \left(\epsilon_{0}(z) - \epsilon_{p}(\omega, \gamma) \right) = 0.
\end{equation*}
Under the assumption\footnote{Such a condition is not restrictive as it is satisfied in the natural tissues and materials.}$\Re\left(\epsilon_{0}(z) \right) - \epsilon_{\infty} > 0$, this solution satisfies the following estimates:
\begin{equation*}
\omega_{0} < \omega < \sqrt{\omega^{2}_{0} + \omega^{2}_{p}} := \omega_{max} \quad \text{and} \quad 0 < \gamma <  \omega_{max}\, \left\Vert \frac{\Im\left(\epsilon_{0}(\cdot)\right)}{\Re\left(\epsilon_{0}(\cdot)\right)}  \right\Vert_{\mathbb{L}^{\infty}(\Omega)}:= \gamma_{max}.
\end{equation*}

\item We have the monotonicity property
\begin{equation*}\label{monotonicity}
\lambda^{(3)}_{n}\; <\; \lambda^{(3)}_{m} \Rightarrow \omega_{n}\;<\;\omega_{m}.
\end{equation*}
\item In addition, if $\omega=\omega_{n_0}\pm a^h$ and $\gamma=\gamma_{n_0}\pm a^h$, then the following estimates are fulfilled
\begin{equation}\label{eigval==1}
\left\vert\epsilon_{0}(z) - \lambda^{(3)}_{n} \left(\epsilon_{0}(z) - \epsilon_{p} \right) \right\vert  \sim 
\begin{cases}
  a^{h} & \text{if} \quad n = n_{0} \\
  1     & \text{if} \quad n \neq n_{0}
\end{cases}.
\end{equation}
\end{enumerate} 
\end{lemma}
\begin{proof}
(1). The equation
\begin{equation*}
\epsilon_{0}(z) - \lambda^{(3)}_{n_{0}} \left(\epsilon_{0}(z) - \epsilon_{p}(\omega, \gamma) \right) = 0,
\end{equation*}
becomes, after using the Lorentz model for the permittivity $(\ref{Lorentzdef})$, 
\begin{equation}\label{Fibladi}
f_{n_{0}}(\omega,\gamma) := \epsilon_{0}(z) \left(1 - \lambda^{(3)}_{n_{0}} \right) + \lambda^{(3)}_{n_{0}} \, \epsilon_{\infty} \left[1 + \frac{\omega^{2}_{p}}{\omega^{2}_{0} - \omega^{2} + i \gamma \, \omega} \right] = 0.
\end{equation}
We split the preceding equation into a real part given by:
\begin{equation}\label{MDSRE}
\Re\left(\epsilon_{0}(z)\right) \, \left(1 - \lambda^{(3)}_{n_{0}} \right) + \lambda^{(3)}_{n_{0}} \, \epsilon_{\infty}  + \frac{\omega^{2}_{p} \, \lambda^{(3)}_{n_{0}} \, \epsilon_{\infty} \, (\omega^{2}_{0} - \omega^{2})}{(\omega^{2}_{0} - \omega^{2})^{2} + ( \gamma \, \omega)^{2}}  = 0
\end{equation}
and imaginary part equation given by:
\begin{equation}\label{MDSIM}
\Im\left(\epsilon_{0}(z)\right) \, \left(1 - \lambda^{(3)}_{n_{0}} \right)  - \frac{\omega^{2}_{p} \, \lambda^{(3)}_{n_{0}} \, \epsilon_{\infty} \, \gamma \omega}{(\omega^{2}_{0} - \omega^{2})^{2} + ( \gamma \, \omega)^{2}}  = 0.
\end{equation}
From $(\ref{MDSIM})$ we get:
\begin{equation*}
\frac{\Im\left(\epsilon_{0}(z)\right) \, \left(1 - \lambda^{(3)}_{n_{0}} \right)}{\gamma \omega} = \frac{\omega^{2}_{p} \, \lambda^{(3)}_{n_{0}} \, \epsilon_{\infty} }{(\omega^{2}_{0} - \omega^{2})^{2} + ( \gamma \, \omega)^{2}}
\end{equation*}
and we plug it into $(\ref{MDSRE})$ to obtain: 
\begin{equation*}\label{QuadEqua}
\gamma \, \omega \left[ \Re\left(\epsilon_{0}(z)\right) \, \left(1 - \lambda^{(3)}_{n_{0}} \right) + \lambda^{(3)}_{n_{0}} \, \epsilon_{\infty} \right] + \Im\left(\epsilon_{0}(z)\right) \, \left(1 - \lambda^{(3)}_{n_{0}} \right) \, (\omega^{2}_{0} - \omega^{2})  = 0.
\end{equation*}
This implies, 
\begin{equation}\label{omegamma}
\gamma \, \omega = \frac{\Im\left(\epsilon_{0}(z)\right) \, \left(1 - \lambda^{(3)}_{n_{0}} \right) \, (\omega^{2} - \omega^{2}_{0})}{\left[ \Re\left(\epsilon_{0}(z)\right) \, \left(1 - \lambda^{(3)}_{n_{0}} \right) + \lambda^{(3)}_{n_{0}} \, \epsilon_{\infty} \right]},
\end{equation}
and, consequently, 
\begin{equation*}
\frac{1}{\left( \omega^{2} - \omega^{2}_{0} \right)^{2}+\left(\gamma \, \omega\right)^{2}} = \frac{\left[ \Re\left(\epsilon_{0}(z)\right) \, \left(1 - \lambda^{(3)}_{n_{0}} \right) + \lambda^{(3)}_{n_{0}} \, \epsilon_{\infty} \right]^{2}}{\left( \omega^{2} - \omega^{2}_{0} \right)^{2} \, \left\vert \epsilon_{0}(z) \, \left(1 - \lambda^{(3)}_{n_{0}} \right) + \lambda^{(3)}_{n_{0}} \, \epsilon_{\infty} \right\vert^{2}}.
\end{equation*}
Gathering this last equation with $(\ref{MDSRE})$ to obtain:  
\begin{equation*}
\left( \omega^{2}_{0} - \omega^{2} \right) \,  \left\vert \epsilon_{0}(z) \, \left(1 - \lambda^{(3)}_{n_{0}} \right) + \lambda^{(3)}_{n_{0}} \, \epsilon_{\infty} \right\vert^{2} +   \omega^{2}_{p} \, \lambda^{(3)}_{n_{0}} \, \epsilon_{\infty}  \left[ \Re\left(\epsilon_{0}(z)\right) \, \left(1 - \lambda^{(3)}_{n_{0}} \right) + \lambda^{(3)}_{n_{0}} \, \epsilon_{\infty} \right] = 0, 
\end{equation*}
and solving the obtained equation with respect to $\omega$, that we denote in the sequel by $\omega_{n_{0}}$ to highlight its dependence with respect to $\lambda^{(3)}_{n_{0}}$, to get:  
\begin{equation}\label{HAomega}
\omega_{n_{0}} = \left[ \omega^{2}_{0} + \frac{\omega^{2}_{p} \, \lambda^{(3)}_{n_{0}} \, \epsilon_{\infty}  \left[ \Re\left(\epsilon_{0}(z)\right) \, \left(1 - \lambda^{(3)}_{n_{0}} \right) + \lambda^{(3)}_{n_{0}} \, \epsilon_{\infty} \right]}{ \left\vert \epsilon_{0}(z) \, \left(1 - \lambda^{(3)}_{n_{0}} \right) + \lambda^{(3)}_{n_{0}} \, \epsilon_{\infty} \right\vert^{2}}  \right]^{\frac{1}{2}}.
\end{equation}
In addition, from the relation $ \Re\left(\epsilon_{0}(z) \right) - \epsilon_{\infty} > 0$, we deduce that $\omega_{n_{0}} > \omega_{0}$ which is a lower bound for $\omega_{n_{0}}$ and, next, we compute an upper bound for $\omega_{n_{0}}$ as follow:
\begin{equation}\label{omega2leqomega0}
\omega^{2}_{n_{0}} = \omega^{2}_{0} + \frac{\omega^{2}_{p} \, \lambda^{(3)}_{n_{0}} \, \epsilon_{\infty}  \left[ \Re\left(\epsilon_{0}(z)\right) \, \left(1 - \lambda^{(3)}_{n_{0}} \right) + \lambda^{(3)}_{n_{0}} \, \epsilon_{\infty} \right]}{ \left\vert \epsilon_{0}(z) \, \left(1 - \lambda^{(3)}_{n_{0}} \right) + \lambda^{(3)}_{n_{0}} \, \epsilon_{\infty} \right\vert^{2}}  \leq  \omega^{2}_{0} + \frac{\omega^{2}_{p}  \, \epsilon_{\infty}}{ \left[ \Re\left( \epsilon_{0}(z) \right) \, \left(1 - \lambda^{(3)}_{n_{0}} \right) + \lambda^{(3)}_{n_{0}} \, \epsilon_{\infty} \right]}.
\end{equation}
As $\lambda^{(3)}_{n_{0}} \in ]0,1]$ and $\Re\left( \epsilon_{0}(z) \right) - \epsilon_{\infty} > 0$ we get $\left[ \Re\left( \epsilon_{0}(z) \right) \, \left(1 - \lambda^{(3)}_{n_{0}} \right) + \lambda^{(3)}_{n_{0}} \, \epsilon_{\infty} \right] > \, \epsilon_{\infty}$
and, consequently, we obtain 
from $(\ref{omega2leqomega0})$ the following upper bound 
\begin{equation*}\label{Iomegamax}
\omega_{n_{0}} < \sqrt{\omega^{2}_{0}+\omega^{2}_{p}} := \omega_{max}.
\end{equation*}
Now, as we have an expression of $\omega_{n_{0}}$, see $(\ref{HAomega})$, we plug it into $(\ref{omegamma})$ to obtain the value of the corresponding damping frequencies $\gamma_{n_{0}}$. More precisely,  
\begin{equation*}
\gamma_{n_{0}} = \frac{\Im\left(\epsilon_{0}(z)\right) \, \left(1 - \lambda^{(3)}_{n_{0}} \right) \, \omega^{2}_{p} \, \lambda^{(3)}_{n_{0}} \, \epsilon_{\infty}}{\left\vert \epsilon_{0}(z) \, \left(1 - \lambda^{(3)}_{n_{0}} \right) + \lambda^{(3)}_{n_{0}} \, \epsilon_{\infty} \right\vert \, \bm{Q_{n_{0}}}},
\end{equation*}
where $\bm{Q_{n_{0}}}$ is such that
\begin{eqnarray*}
\bm{Q_{n_{0}}} = \left[ \omega^{2}_{0}  \left\vert \epsilon_{0}(z) \, \left(1 - \lambda^{(3)}_{n_{0}} \right) + \lambda^{(3)}_{n_{0}} \, \epsilon_{\infty} \right\vert^{2} 
+ \omega_{p}^{2} \, \lambda^{(3)}_{n_{0}} \, \epsilon_{\infty} \left[ \Re\left(\epsilon_{0}(z)\right) \, \left(1 - \lambda^{(3)}_{n_{0}} \right) + \lambda^{(3)}_{n_{0}} \, \epsilon_{\infty} \right] \right]^{\frac{1}{2}}.
\end{eqnarray*}
As we did with the frequency $\omega_{n_{0}}$, we also need to compute an upper bound for $\gamma_{n_{0}}$. For this we recall from $(\ref{omegamma})$ that:   
\begin{equation*}
0 < \gamma_{n_{0}} \frac{\omega_{n_{0}}}{(\omega^{2}_{n_{0}} - \omega^{2}_{0})} = \frac{\Im\left(\epsilon_{0}(z)\right) \, \left(1 - \lambda^{(3)}_{n_{0}} \right)}{\left[ \Re\left(\epsilon_{0}(z)\right) \, \left(1 - \lambda^{(3)}_{n_{0}} \right) + \lambda^{(3)}_{n_{0}} \, \epsilon_{\infty} \right]} < \frac{\Im\left(\epsilon_{0}(z)\right)}{ \Re\left(\epsilon_{0}(z)\right)}, 
\end{equation*}
and, knowing that $\omega_{n_{0}} > 0, \, \omega^{2}_{n_{0}} - \omega^{2}_{0} > 0$, and $\omega_{max}$ is an upper bound for $\omega_{n_{0}}$, we deduce
\begin{equation*}
0 < \gamma_{n_{0}}   <  \frac{(\omega^{2}_{n_{0}} - \omega^{2}_{0})}{\omega_{n_{0}}} \, \frac{\Im\left(\epsilon_{0}(z)\right)}{ \Re\left(\epsilon_{0}(z)\right)} <  \frac{(\omega^{2}_{max} - \omega^{2}_{0})}{\omega_{max}} \, \frac{\Im\left(\epsilon_{0}(z)\right)}{ \Re\left(\epsilon_{0}(z)\right)} < \omega_{max} \, \left\Vert \frac{\Im\left(\epsilon_{0}(\cdot)\right)}{ \Re\left(\epsilon_{0}(\cdot)\right)}\right\Vert_{\mathbb{L}^{\infty}(\Omega)}:= \gamma_{max}.
\end{equation*}
This proves that in the square $\left( \omega_{0},\omega_{max}\right) \times \left( 0,\gamma_{max} \right)$ the dispersion equation $f_{n_{0}}(\omega,\gamma)=0$, given by $(\ref{Fibladi})$, admits a unique solution $\left( \omega_{n_{0}},\gamma_{n_{0}} \right)$. 
\smallskip

(2). From the expression of $\omega_{n_{0}}$, see for instance $(\ref{HAomega})$, which now be noted by $\omega_{n_{0}}:=\omega\left( \lambda^{(3)}_{n_{0}} \right)$, we can derive its monotonicity with respect to the index $n$, or equivalently with respect to the sequence of eigenvalues $\lambda^{(3)}_{n}$. More precisely, we have:  
\begin{equation*}
\omega^{2}\left( \lambda^{(3)}_{n} \right) = \omega^{2}_{0} + \frac{\omega^{2}_{p} \, \lambda^{(3)}_{n} \, \epsilon_{\infty}  \left[ \Re\left(\epsilon_{0}(z)\right) \, \left(1 - \lambda^{(3)}_{n} \right) + \lambda^{(3)}_{n} \, \epsilon_{\infty} \right]}{ \left\vert \epsilon_{0}(z) \, \left(1 - \lambda^{(3)}_{n} \right) + \lambda^{(3)}_{n} \, \epsilon_{\infty} \right\vert^{2}}, 
\end{equation*}
and, by computing the derivative with respect to the eigenvalue $\lambda^{(3)}_{n}$, then 
\begin{eqnarray*}
\partial_{\lambda^{(3)}_{n}} \left( \omega^{2} \right) \left( \lambda^{(3)}_{n} \right) &=& \frac{\omega^{2}_{p} \, \epsilon_{\infty} \, \Re\left(\epsilon_{0}(z)\right) \, \left\vert \epsilon_{0}(z) \, \left(1 - \lambda^{(3)}_{n} \right) + \lambda^{(3)}_{n} \, \epsilon_{\infty}  \right\vert^{2}}{ \left\vert \epsilon_{0}(z) \, \left(1 - \lambda^{(3)}_{n} \right) + \lambda^{(3)}_{n} \, \epsilon_{\infty}  \right\vert^{4}} \\
&+& \frac{\omega^{2}_{p} \, \epsilon_{\infty} \, 2 \, \lambda^{(3)}_{n} \,\, \left( 1 - \lambda^{(3)}_{n} \right) \, \epsilon_{\infty} \, \left( \Im\left(\epsilon_{0}(z) \right) \right)^{2}  }{ \left\vert \epsilon_{0}(z) \, \left(1 - \lambda^{(3)}_{n} \right) + \lambda^{(3)}_{n} \, \epsilon_{\infty}  \right\vert^{4}} >0.
\end{eqnarray*}
Hence, 
\begin{equation*}
\partial_{\lambda^{(3)}_{n}} \left( \omega^{2} \right) \left( \lambda^{(3)}_{n} \right) = 2 \,\, \partial_{\lambda^{(3)}_{n}} \left( \omega \right) \left( \lambda^{(3)}_{n} \right) \,\, \omega  \left( \lambda^{(3)}_{n} \right) >0.
\end{equation*}
As we know that $\omega\left( \cdot \right) > 0$, we deduce in straightforward manner that $ \partial_{\cdot} \left( \omega \right) \left( \cdot \right) > 0$ and then $\omega\left( \lambda^{(3)}_{n} \right)$ is strictly increasing function. This implies, $\lambda^{(3)}_{n} < \lambda^{(3)}_{m} \Rightarrow \omega_{n} < \omega_{m}$. 
\bigskip

(3).  Now, if we choose $\omega = \omega_{n_{0}} \pm a^{h}$ and $\gamma = \gamma_{n_{0}} \pm a^{h}$ we obtain from $(\ref{Fibladi})$ the following relation:
\begin{eqnarray*}
f_{n_{0}}(\omega_{n_{0}} \pm a^{h},\gamma_{n_{0}} \pm a^{h}) &=& \epsilon_{0}(z) \left(1 - \lambda^{(3)}_{n_{0}} \right) + \lambda^{(3)}_{n_{0}} \, \epsilon_{\infty} \left[1 + \frac{\omega^{2}_{p}}{\omega^{2}_{0} - \left( \omega_{n_{0}} \pm a^{h} \right)^{2} + i \left( \gamma_{n_{0}} \pm a^{h} \right) \, \left( \omega_{n_{0}} \pm a^{h} \right) } \right] \\
&=& - \frac{\lambda^{(3)}_{n_{0}} \, \epsilon_{\infty} \omega^{2}_{p}}{\omega^{2}_{0} - \omega^{2}_{n_{0}} + i \gamma_{n_{0}} \,  \omega_{n_{0}}  }  + \frac{\lambda^{(3)}_{n_{0}} \, \epsilon_{\infty} \omega^{2}_{p}}{\omega^{2}_{0} - \left( \omega_{n_{0}} \pm a^{h} \right)^{2} + i \left( \gamma_{n_{0}} \pm a^{h} \right) \, \left( \omega_{n_{0}} \pm a^{h} \right) } \\
&=& \mp \frac{\lambda^{(3)}_{n_{0}} \, \epsilon_{\infty} \omega^{2}_{p} \, a^{h} \, \left(  \omega_{n_{0}} (i-2) +  \gamma_{n_{0}}  \right)}{\left[\omega^{2}_{0} - \omega^{2}_{n_{0}} + i \gamma_{n_{0}} \,  \omega_{n_{0}}  \right]^{2}} + \mathcal{O}\left(a^{2h} \right),
\end{eqnarray*}
which after taking the absolute value we obtain $f_{n_{0}}(\omega_{n_{0}} \pm a^{h}, \gamma_{n_{0}} \pm a^{h}) = \mathcal{O}\left( a^{h} \right)$, or equivalently, 
\begin{equation*}
\left\vert\epsilon_{0}(z) - \lambda^{(3)}_{n_{0}} \left(\epsilon_{0}(z) - \epsilon_{p} \right) \right\vert  \sim  a^{h}. 
\end{equation*}
This proves the first part of $(\ref{eigval==1})$. To finish with the estimation of $(\ref{eigval==1})$ we need to prove, for $n \neq n_{0}$, that  
\begin{equation*}
\left\vert\epsilon_{0}(z) - \lambda^{(3)}_{n} \left(\epsilon_{0}(z) - \epsilon_{p} \right) \right\vert  \sim  1. 
\end{equation*}
For this computing 
\begin{equation*}
f_{n}(\omega_{n_{0}} \pm a^{h},\gamma_{n_{0}} \pm a^{h}) := \epsilon_{0}(z) \left(1 - \lambda^{(3)}_{n} \right) + \lambda^{(3)}_{n} \, \epsilon_{\infty} \left[1 + \frac{\omega^{2}_{p}}{\omega^{2}_{0} - \left( \omega_{n_{0}} \pm a^{h} \right)^{2} + i \left( \gamma_{n_{0}} \pm a^{h} \right) \, \left( \omega_{n_{0}} \pm a^{h} \right) } \right],
\end{equation*}
which, by the use of $(\ref{Fibladi})$, becomes,
\begin{eqnarray*}
f_{n}(\omega_{n_{0}} \pm a^{h},\gamma_{n_{0}} \pm a^{h}) &=&   \frac{-\lambda^{(3)}_{n_{0}} \, \epsilon_{\infty} \, \omega^{2}_{p} \,}{\omega^{2}_{0} - \omega^{2}_{n_{0}} + i \gamma_{n_{0}} \,  \omega_{n_{0}}  }  + \frac{\lambda^{(3)}_{n} \, \epsilon_{\infty} \, \omega^{2}_{p} \,}{\omega^{2}_{0} - \left( \omega_{n_{0}} \pm a^{h} \right)^{2} + i \left( \gamma_{n_{0}} \pm a^{h} \right) \, \left( \omega_{n_{0}} \pm a^{h} \right)}  \\ &+& \left(\epsilon_{0}(z) - \epsilon_{\infty} \right) \,  \left(\lambda^{(3)}_{n_{0}} - \lambda^{(3)}_{n} \right) \\
&=& \left[ \frac{\epsilon_{\infty} \, \omega^{2}_{p} }{\omega^{2}_{0} - \omega^{2}_{n_{0}} + i \gamma_{n_{0}} \,  \omega_{n_{0}}  } - \epsilon_{0}(z) + \epsilon_{\infty}  \,  \right] \,  \left(\lambda^{(3)}_{n} - \lambda^{(3)}_{n_{0}} \right) + \mathcal{O}\left(a^{h} \right).
\end{eqnarray*}
As, 
\begin{equation*}
\left[ \frac{\epsilon_{\infty} \, \omega^{2}_{p} }{\omega^{2}_{0} - \omega^{2}_{n_{0}} + i \gamma_{n_{0}} \,  \omega_{n_{0}}  } - \epsilon_{0}(z) + \epsilon_{\infty}  \,  \right] \sim 1,
\end{equation*}
we deduce that 
\begin{equation}\label{BLADZ}
f_{n}(\omega_{n_{0}} \pm a^{h}, \gamma_{n_{0}} \pm a^{h}) = \mathcal{O}  \left(\left\vert \lambda^{(3)}_{n} - \lambda^{(3)}_{n_{0}} \right\vert \right) + \mathcal{O}\left(a^{h} \right).
\end{equation}
Here, to get more precisions on the estimation of the last formula, we recall from \textbf{Remark} \ref{LimSeq} that $1/2$ is the only accumulation point for the sequence $\{ \lambda^{(3)}_{n} \}_{n \in \mathbb{N}}$. This suggest to take $\lambda^{(3)}_{n_{0}} \, \not\in ]\frac{1}{2} - a^{\frac{h}{2}}; \frac{1}{2} + a^{\frac{h}{2}} [$ in order to get $\left\vert \lambda^{(3)}_{n} - \lambda^{(3)}_{n_{0}} \right\vert \sim 1$ and then $(\ref{BLADZ})$ will be reduced to:  
\begin{equation*}
f_{n}(\omega_{n_{0}} \pm a^{h}, \gamma_{n_{0}} \pm a^{h}) =  \mathcal{O}\left(1 \right).
\end{equation*}
This ends the proof of \textbf{Lemma}\, \ref{LemmaClaim}.  
\end{proof}
It is worth emphasizing that the previous derivation of the frequencies $\omega$ can be made also by using the Drude model for the permittivity, see for instance \cite{engheta2006metamaterials} formula 1.5, instead of the Lorentz model.
\smallskip

\begin{remark}\label{Other-options}
In the previous computations, as we have seen, it's mandatory to vary both the frequencies $\omega$ and the damping coefficient $\gamma$. This is expensive from the point of view of actual applications as we have to change the nano-particle to change the damping frequency. Here, we describe two ways to overcome  this issue.
\begin{enumerate}
\item In the case where $\Im\left(\epsilon_{0}(z) \right)$ is small (or mathematically $\Im\left(\epsilon_{0}(z) \right) = 0$), we can take the undamped frequency $\gamma$ fixed but small. In this case, we deduce from $(\ref{MDSRE})$, the corresponding frequencies as follows:  
\begin{equation*}
\omega_{n_{0}} = \left[ \omega^{2}_{0} + \frac{\lambda_{n_{0}}^{(3)} \epsilon_{\infty} \omega^{2}_{p}}{\left(\lambda_{n_{0}}^{(3)}  \,  \epsilon_{\infty}  + \left( 1 - \lambda_{n_{0}}^{(3)} \right) \, \Re\left(\epsilon_{0}(z) \right) \right)}\right]^{\frac{1}{2}}. 
\end{equation*}
Recall that the usual photo-acoustic experiment applies to targets that are electrically conducting, i.e. with the imaginary part of the 'permittivity' highly pronounce. However, the photo-acoustic experiment using contrast agents, as described in this work, can apply to tissues which are electrically low conducting. This situation is known for early stage anomalies (as the benign cancer).
\bigskip
 
\item We allow the frequencies $\omega$ to be in the complex plan and in this case we get: 
\begin{equation*}
\omega_{n_{0}} = \frac{i \gamma \pm \sqrt{\Delta^{\star}}}{2},
\end{equation*} 
where\footnote{For $z \in \mathbb{C}$, given by $z= r \, e^{i \, \phi}$ with $- \pi < \phi \leq \pi$, the principal square root of $z$ is defined to be: $\sqrt{z}= \sqrt{r} \, e^{i \, \frac{\phi}{2}}$.} 
\begin{equation*}
\Delta^{\star} = - \gamma^{2} + 4 \left(\omega^{2}_{0} + \frac{\lambda^{(3)}_{n}\, \epsilon_{\infty} \, \omega^{2}_{p}}{\left[ \epsilon_{0}(z) \left(1 - \lambda^{(3)}_{n} \right) + \epsilon_{\infty} \, \lambda^{(3)}_{n} \right]} \right).
\end{equation*}
In this case, the damping frequency $\gamma$ can be taken fixed and even small.
\end{enumerate}
\end{remark}

\end{document}